\documentclass[12pt]{article}  

\newcommand{\al}{\alpha}

\newcommand{\ga}{\gamma}
\newcommand{\Ga}{\Gamma}

\newcommand{\eps}{\epsilon}
\newcommand{\var}{\varphi}
\newcommand{\del}{\delta}

\newcommand{\f}{\frac}

\newcommand{\subs}{\subseteq}
\newcommand{\ovs}{\overset}

\newcommand{\BN}{\Bbb{N}}

\newcommand{\cH}{{\cal{H}}}

\newcommand{\lo}{\longrightarrow}
\newcommand{\Lo}{\Longrightarrow}

\newcommand{\rig}{\rightarrow}

\newcommand{\Max}{\operatorname{Max}}

\newcommand{\Sign}{\operatorname{Sign}}

\newcommand{\bS}{{\bold{S}}}
\newcommand{\bT}{{\bold{T}}}
\newcommand{\bA}{{\bold{A}}}

\input{epsf}

\usepackage{graphicx}
\usepackage{amsmath}
\usepackage{amsthm}
\usepackage{amssymb}
\usepackage{latexsym}

\newtheorem{theorem}{Theorem}[section]
\newtheorem{lemma}[theorem]{Lemma}
\newtheorem{proposition}[theorem]{Proposition}
\newtheorem{corollary}[theorem]{Corollary}
\theoremstyle{definition}
\newtheorem{definition}[theorem]{Definition}
\newtheorem{example}[theorem]{Example}
\newtheorem{remark}[theorem]{Remark} 
\usepackage{amsmath,amsthm,amssymb}
\usepackage{enumerate}
\usepackage{amsmath}
\usepackage{graphicx}

\def\titlerunning#1{\gdef\titrun{#1}}
\makeatletter
\def\author#1{\gdef\autrun{\def\and{\unskip, }#1}\gdef\@author{#1}}
\def\address#1{{\def\and{\\\hspace*{18pt}}\renewcommand{\thefootnote}{}%
\footnote {#1}}%
\markboth{\autrun}{\titrun}}
\makeatother
\def\email#1{e-mail: #1}
\def\subjclass#1{{\renewcommand{\thefootnote}{}%
\footnote{\emph{Mathematics Subject Classification (2010):} #1}}}
\def\keywords#1{\par\medskip
\noindent\textbf{Keywords.} #1}

\theoremstyle{definition}

\numberwithin{equation}{section}

\begin{document}


\baselineskip=17pt


\titlerunning{}

\title{ Using Supervised Learning to Construct the Best Regularization Term and the Best Multiresolution Analysis}
\author{ Saman Khoramian}
\date{}
\maketitle
\address{ \email{Saman.Khoramian@gmail.com}}
\begin{abstract}
By the recent advances in computer technology leading to the invention of more 
powerful processors, the importance of creating models using data training is even 
greater than ever. Given the significance of this issue, this work tries to establish a connection between “the ongoing research direction of Learning and Inverse Problems” and “Applied Harmonic Analysis”. Inspired by methods introduced in [12, 17, 22, 30], which are 
connections between Wavelet and Inverse Problems, we offer a model with the 
capability of learning in terms of an application in signal processing. In order to reach this model, a bi-level optimization 
problem will have to be faced. For solving this, a sequence of step functions 
is presented that its convergence to the solution will be proved. Each of these 
step functions derives from several constrained optimization problems on 
$\mathbb{R}^n$ that will be introduced here.

\end{abstract}
\section{Introduction}
In recent years, growing volumes and varieties of available data, computational processing that is cheaper and more powerful, as well as affordable data storage has laid the foundation for producing the models that can analyze bigger, more complex data and deliver faster, more accurate results - even on a very large scale. All of these items could be a great incentive for researchers in different fields to present models with self-upgrading ability by learning from data. In fact, this work arises from the impact of machine's power on {\it Inverse Problems} [26]. 
 
Because of new computing technologies, the importance of learning methods today is not like what it was in the past.  In recent years, increasing data availability has allowed learning models to be trained on a large pool of examples, accompanied by increasing computer processing power which has supported the analytical capabilities of these models. With these developments, researchers of different fields are likely to be attracted to use data learning  in their works. In the context of inverse problems, in addition to data usage for informing and validating inverse problems solutions [6, 9, 29], applying learning ability from training data to inform model selection in inverse problems has more recently been given real attention in the community [14, 15, 20, 21, 23, 32, 34, 35, 54, 55, 56, 57, 58]. Among these challenges, the papers [20, 34] are pioneer works in signal denoising by solving a bi-level optimization problem that arises from a merge between learning and inverse problems. Parameter choice in variational regularization methods is the topic of these papers. In fact, they consider the problem of parameter learning for variational image denoising models which lead to bi-level optimization problems that for solving them the semi-smooth Newton methods were applied. Here, like [20, 34] the signal denoising will be addressed and we will select the model by learning from training data as well. However, we will connect “learning and inverse problems” with “multi-resolution analysis” in the realm of applied harmonic analysis and choose the whole penalty term instead of merely the regularization parameter - taking into account the existence of huge data training and powerful computing instruments that they are already available or we are moving toward its availability.\\

Apart from what has been discussed up to here, in a more general perspective, it should be mentioned that this note can also be considered in the view of merging different disciplines. In the last decade, we have witnessed significant advances in many individual core areas of data analysis, including Machine Learning, Signal Processing, Optimization and Harmonic Analysis. It appears highly likely that the next major breakthroughs will occur at the intersection of these disciplines (from Applied Harmonic Analysis, Massive Data Sets, Machine Learning [25, 42], and Signal Processing). Indeed, this paper is to that aim and in this regard, a scheme for it will be presented in the form of a signal denoising problem in continue.\\

In signal processing, noise is a general term for unwanted (and, in general, unknown) modifications that a signal may suffer from during capture, storage, transmission, processing or conversion, and noise reduction is the recovery of the original signal from the noise-corrupted one. In this paper, a special method is found from the perspective of removing noises that occur during data transfer. Since the notion of noise is quite vague, and it is often very difficult to model its structure or its origins, it is important to present a model as close to reality as possible. To make denoising model closer to reality, here, the approach is learning from sample signals.  In the remainder of this section, the main problem of this paper will be introduced in terms of removing the signal's noise in a practical application:\\
\\
Suppose $f$ is an image which is recorded in a place. This image will be 
transferred to another place. What is received is an image $g$
that is a noisy 
version of image  $f$, i.e. $g=f+e$
where $e$  is an unknown noise. In general, there is not 
much information about noise $e$. Noise $e$  can be different in different 
conditions such as different types of weather,
different distances, etc. But suppose $f_1,f_2,\cdots , f_m$   
are certain contractual signals between the sender and the receiver 
(in other words, the sender and the receiver have these finite signals in 
advance) and these pre-defined signals are sent at the outset of each 
connection before sending the recorded image. Because the receiver also has 
these signals, this action can be beneficial to the process. 
Suppose 
$g_1,g_2,\cdots , g_m$
are the signals that are received by the receiver (corresponding 
signals of $f_1,f_2,\cdots ,f_m$).
Then, because the receiver has  $f_1,f_2, \cdots , f_m$ in 
advance, they become aware of a collection of noises: $e_i=g_i-f_i$
for every $i;~1\leq i\leq m$. 
Therefore; because this action always occurs just before a recorded image is 
sent, by having this collection of noises, the receiver will have more 
information around noise $e=g-f$
and can act much better in the process of denoising. 
Here, we define $(f_i,g_i)$
for $i=1,\cdots ,m$  a {\it training set}. However, the challenge is finding a 
denoising mathematical model in which this training set can play a role in the 
process of denoising for $g$. From the point of machine learning language, $f_i$'s and $g_i$'s are called objects (inputs) and labels (outputs) respectively and the challenge is finding an appropriate supervised learning algorithm regarding them.

At this stage the question of ``how can this prior knowledge about the noises be 
embedded in the denoising process?" will be addressed.
Regarding denoising problems, several approaches have been introduced in 
applied mathematics literature such as filtering approach [1, 46]
and the variational/PDE-based approaches [2, 13, 44]. Here, among the different models of 
denoising, we will focus on one of them which finds a function 
$\tilde{f}$  that minimizes 
over all possible functions $f$  the functional 
\begin{equation*}
\|f-g\|_{L_2(I)}^2+\lambda \|f\|_Y
\tag*{$(1\cdot 1)$}
\end{equation*}
where
$$\|f-g\|_{L_2(I)}:=\Biggl(\int_I|f(x)-g(x)|^2dx\Biggr)^{\frac{1}{2}}$$
is the root-mean-square {\it {error}} (or more generally difference) 
between $f$  and $g$, and $\|f\|_Y$ is the norm of the 
approximation $f$
in a {\it {smoothness space}} $Y.$ The original image (signal) $g$ could be 
noisy, while $\tilde{f}$
would be a denoised version of $g$. The amount of noise removal is 
determined by the parameter $\lambda$,
if $\lambda$ is large, then necessarily $\|f\|_Y$ must be smaller 
at the minimum, i.e. $f$
must be smoother, while when $\lambda$ is small, $f$ can be rough, 
with $\|f\|_Y$
large, and one achieves a small error at the minimum; see [12, 22].

As a result, choosing the right $\lambda$
in (1.1) will help to a better denoising. In continue, among all possible formats for (1.1), we focus on the following format:
\begin{align*}
& \Phi_\lambda:\mathcal{H}\longrightarrow \mathbb{R}\\
& \Phi_\lambda(f)=\|f-g\|_{\mathcal{H}}^2+\lambda|||f|||_{W,p}^p
\tag*{$(1\cdot 2)$}
\end{align*}
where $\mathcal{H}$ is a Hilbert space and
\begin{equation*}
|||f|||_{W,p}=\Biggl(\sum_{\gamma\in\Gamma} w_\gamma|\langle 
f,\varphi_\gamma\rangle|^p\Biggr)^{\frac{1}{p}}
\tag*{$(1\cdot 3)$}
\end{equation*}
for $1\leq p\leq 2$, 
is a weighted $\ell_p$-norm of the coefficients of $f$  with respect to an 
orthonormal basis $(\varphi_\gamma)_{\gamma\in\Gamma}$
of $\mathcal{H}$,
 and a sequence of strictly positive weights $W=(w_\gamma)_{\gamma\in 
\Gamma}$.

An idea arising at this time is that if two or more constraints (approximations for smoothness spaces norms) in a 
minimization problem can simultaneously be considered, we can gain better 
results for denoising. Again, among all possible minimization problems that involve more than one constraint in the process of denoising, we consider the minimization of the following functional:
\begin{equation*}
\Phi_{\lambda_1,\cdots , 
\lambda_n}(f)=\|f-g\|_{\mathcal{H}}^2+\lambda_1|||f|||_{W_1,p_1}^{p_1}+\cdots
+\lambda_n
|||f|||_{W_n,p_n}^{p_n}.
\tag*{$(1\cdot 4)$}
\end{equation*}
This is the same minimization problem raised in [30] for $K=I$,
where $K$  is a linear 
bounded operator from $\mathcal{H}$  to $\mathcal{H}'$.
 In [30] an iterative sequence in 
$\mathcal{H}$  has been 
presented that converges strongly to the minimizer of $\Phi_{\lambda_1,\cdots , 
	\lambda_n}$.
However, for the case  $K=I$,
the minimization problem can be solved simpler and more straightforward 
without having to going through the iterative process. We have summarized this 
in the following lemma:
\begin{lemma}\label{1.1}
Define the function $\Phi$ as follows:
$$\begin{array}{c}
\Phi:\mathcal{H}\longrightarrow \mathbb{R}\\
\Phi(f)=\|f-g\|_\mathcal{H}^2+|||f|||_{W_1,p_1}^{p_1}+\cdots +|||f|||_{W_n,p_n}^{p_n}
\end{array}$$
where $\mathcal{H}$ is a Hilbert space and
$$|||f|||_{W_i,p_i}=\Biggl(\sum_{\gamma\in\Gamma}w_{i,\gamma}|\langle 
f,\varphi_\gamma\rangle|^{p_i}\Biggr)
^{\frac{1}{p_i}}$$
for $1\leq i\leq n$, $1\leq p_i\leq 2$, 
is a weighted $\ell_{p_i}$-norm of the coefficients of $f$  with respect to an 
orthonormal basis $(\varphi_\gamma)_{\gamma\in\Gamma}$
of $\mathcal{H}$, and a sequence of strictly positive weights  
$W_i=(w_{i,\gamma})_{\gamma\in\Gamma}$. Then 
$f_m=\sum_{\gamma\in\Gamma}S_ {(w_{1,\gamma},\cdots , w_{n,\gamma}),(p_1,\dots, 
	p_n)}(g_\gamma)\varphi_\gamma$ is 
a minimizer of $\Phi$, where the function 
$S_{(c_1,\dots,c_n),(p_1,\dots,p_n)}$ from $\mathbb{R}$ to itself is defined by 
$$S_{(c_1,\dots,c_n),(p_1,\dots,p_n)} (t) =
\begin{cases}
F^{-1}(t) & B=\phi\\
F_1^{-1}(t) & B\neq\phi, t\in \left(\f{\sum_{i\in B}c_i}{2},+\infty\right) \\
0 & B\neq\phi, t\in \left[-\f{\sum_{i\in B}c_i}{2}, \f{\sum_{i\in B} 
c_i}{2}\right] \\
F_2^{-1}(t) & B\neq\phi, t\in\left(-\infty, \f{-\sum_{i\in B}c_i}{2}\right) 
\end{cases}$$
where $B=\{i| p_i=1\}$ and the functions $F_1,F_2,F$ are defined by 
\begin{align*}
&F(x)=x+\Sign x\f{\sum_{i=1}^n p_ic_i |x|^{p_i-1}}{2} \quad \text{for}\quad 
x\in R,\\
&F_1(x)=x+\f{\sum_{i\in B}c_i+\sum_{i\not\in B}p_ic_i|x|^{p_i-1}}{2}
\quad \text{for}\quad x>0,\\
&F_2(x)=x-\f{\sum_{i\in B}c_i+\sum_{i\not\in B}p_ic_i|x|^{p_i-1}}{2}
\quad \text{for}\quad x<0.
\end{align*}
\end{lemma}
\begin{proof}
By using the shorthand notations $f_\gamma$
for  $\langle f,\varphi_\gamma\rangle$ and $g_{\gamma}$ 
for $\langle g,\varphi_\gamma\rangle$, we have:
$$\Phi(f)=\sum_{\gamma\in\Gamma}|f_\gamma-g_\gamma|^2+\sum_{\gamma\in\Gamma}
w_{1,\gamma}|\langle f,\varphi_\gamma\rangle|^{p_1}+\cdots 
+\sum_{\gamma\in\Gamma}w_{n,\gamma}|\langle f,\varphi_\gamma\rangle|^{p_n}.$$
Therefore, 
$$\Phi(f)=\sum_{\gamma\in\Gamma}f^2_\gamma+g_\gamma^2-2f_\gamma 
g_\gamma+w_{1,\gamma}|\langle
f,\varphi_\gamma\rangle|^{p_1}+\dots +w_{n,\gamma}|\langle 
f,\varphi_\gamma\rangle|^{p_n}.$$
We define $M_{g_\gamma;(w_{1,\gamma},\dots ,w_{n,\gamma})}$ as follows:
$$M_{g_\gamma;(w_{1,\gamma},\dots ,w_{n,\gamma})}(x)=x^2+g_\gamma^2-2g_\gamma 
x+w_{1,\gamma}|x|^{p_1}+\dots +w_{n,\gamma}|x|^{p_n}.$$
By this definition, we will have:
$$\Phi(f)=\sum_{\gamma\in\Gamma}M_{g_\gamma;(w_{1,\gamma},\dots 
	,w_{n,\gamma})}(f_\gamma).$$
By Lemma 2.1 in [30], we know $S_{(w_{1,\gamma},\cdots , w_{n,\gamma}),(p_1,\dots,
	p_n)}(g_\gamma)$ is  
a minimizer of $M_{g_\gamma;(w_{1,\gamma},\dots
 ,w_{n,\gamma})}$. Then by the following 
Fact, the proof is compeleted:\\
\\
Fact: {\textit {Suppose the function $\Phi$ is as follows:
\begin{align*}
&\Phi:\mathcal{H}\longrightarrow \mathbb{R}\\
&\Phi(f)=\sum_{\gamma\in\Gamma}M_\gamma(f_\gamma)
\end{align*}
where $\mathcal{H}$
is a Hilbert space with the orthonormal basis 
$(\varphi_\gamma)_{\gamma\in\Gamma}$
and $M_\gamma$ is a function from  $\mathbb{R}$
to $\mathbb{R}$
for every $\gamma\in\Gamma$.
Moreover, $t_\gamma$ is a minimizer of the function $M_\gamma$ for every 
$\gamma\in\Gamma$. Then  $f_m=\sum_{\gamma\in\Gamma}t_\gamma\varphi_\gamma$ is 
a minimizer of the function $\Phi$.}}
\end{proof}
As we said about $\lambda$
in (1.1), choosing the right $\lambda_1,\cdots ,\lambda_n$ in (1.4) will 
help to better denoising too. An idea for the selection of 
$(\lambda_1,\cdots,\lambda_n)$
is using a 
learning process: We consider a training set of pairs $(g_i,f_i);~i=1,2,\cdots 
, m$. In this set as it 
has already been mentioned, $g_i$'s
are noisy signals (images) received in a 
certain condition, and $f_i$'s
represent the pre-defined signals (images). The effort to attain the optimal choice of $(\lambda_1,\cdots ,\lambda_n)$ naturally leads to finding the minimizer of the following function:
\begin{align*}
&\mathfrak{I}:\mathbb{R}^n\longrightarrow \mathbb{R}\\
& 
\mathfrak{I}(\lambda_1,\cdots,\lambda_n)=\sum_{i=1}^m\|\tilde{f}_{i,(\lambda_1,\cdots,
\lambda_n)}-f_i\|^2\\
&
;\tilde{f}_{i,(\lambda_1,\cdots ,\lambda_n)}=\text{arg-min}_{f\in 
\mathcal{H}}\|f-g_i\|^2+
\sum_{j=1}^n\lambda_j|||f|||_{W_j,p_j}^{p_j}.
\tag*{$(1\cdot 5)$}
\end{align*}

The minimization problem of functions in the type of (1.5) is named {\it bilevel optimization problems} [4]. 
Now, a more general form of this view is proposed: Instead of finding 
$(\lambda_1,\cdots ,\lambda_n)$, we 
are going to find an appropriate $\psi$  for the following functional
\begin{equation*}
\|f-g\|^2+\psi(f).
\tag*{$(1\cdot 6)$}
\end{equation*}
Suppose $\mathcal{H}$ is a Hilbert space with the orthonormal basis $(\varphi_\gamma)_{\gamma\in\Gamma}$ and $\textbf{X}$ is the set of all functions $\psi:\mathcal{H}\longrightarrow
 \mathbb{R}$. Moreover, we assume $\textbf{Z}$ is a subset of $\textbf{X}$ such that for every $\psi \in \mathbf{Z}$, the 
problem of finding the minimizer of (1.6) is soluble and present the following function: 
\begin{align*}
&\mathfrak{I}:\textbf{Z}\longrightarrow \mathbb{R}\\
& \mathfrak{I}(\psi)=
\sum_{i=1}^m\|\tilde{f}_{i,\psi}-f_i\|^2\\
&
;\tilde{f}_{i,\psi}=\text{arg-min}_{f\in 
	\mathcal{H}}\|f-g_i\|^2+
\psi(f).
\tag*{$(1\cdot 7)$}
\end{align*}
Clearly, the best choice of the $\psi\in \textbf{Z}$
for (1.6) is a $\psi$  where $\mathfrak{I}(\psi)$ have the lowest 
possible amount or its equivalency, the value of the
$\sum_{i=1}^m\|\tilde{f}_{i,\psi}-f_i\|^2$
is very low. A $\mathbf{Z}$ is appropriate providing that by considering it in the problem (1.7), make it soluble. The set $B_{W,P}$ introduced below is from this type: 
\begin{example}\label{1.2}
Let $\textbf{Z}=B_{W,P}$ which $B_{W,P}$ is defined in the following:
$$B_{W,P}=\{\psi:\mathcal{H}\longrightarrow
\mathbb{R}|~\psi(f)=\sum_{i=1}^n\lambda_i|||f|||_{W_i,p_
i}^{p_i};~\lambda_i\in \mathbb{R}^+ ~ \text{for}~ i\in \{1,\cdots , n\}\}$$
where  $W=\{W_1,\cdots, W_n\}, P=\{p_1,\cdots , p_n\}, 
W_i=\{w_{i,\gamma}\}_{\gamma\in\Gamma}$ and 
$|||f|||_{W_i,p_i}=(\sum_{\gamma\in\Gamma}w_{i,\gamma}|\langle 
f,\varphi_\gamma\rangle|^{p_i})
^{\frac{1}{p_i}}$ for every $i; ~1\leq i\leq n$. Taking into account that $\textbf{Z}=B_{W,P}$,
the best $\psi$ is obtained by finding the minimizer of 
the function (1.5) as explained before. Moreover, we know after finding such 
$\psi$ the 
problem of finding the minimizer of (1.6) is soluble by using Lemma 1.1. For a discussion of this example, see Section 6.  
\end{example}
It is obvious that if the subset $\textbf{Z}$ of $\textbf{X}$
is larger,  $\inf\mathfrak{I}_{\psi\in \textbf{Z}}(\psi)$ will be less. Therefore, 
by assuming the possession of big data of signals and powerful computational 
tools, it would be desirable to seek an optimal $\psi_0$
among the largest subset  $\textbf{Z}$ of   $\textbf{X}$ 
provided that the following minimization problem is soluble for every  
$g\in \mathcal{H}$
$$\text{arg-min}_{f\in \mathcal{H}}\|f-g\|^2+\psi_0(f).$$
This paper addresses the solving of the problem (1.7) by putting $\textbf{Z}=\textbf{C}$ which is much bigger than $B_{W,P}$ introduced in Example 1.2 and is defined as follows:
\begin{align*}
\textbf{C}=\{\psi:\mathcal{H}\longrightarrow \mathbb{R}|&~\psi(f)=\sum_{\gamma\in\Gamma}
\psi_\gamma(f_\gamma);~\forall 
\gamma\in\Gamma:~\psi_\gamma:\mathbb{R}\longrightarrow \mathbb{R}\\
& \text{is a {\it regulated} and 
	{\it quasiconvex function}}\}
\end{align*}
where a function $f:\mathbb{S}\longrightarrow \mathbb{R}$ defined on a convex subset $\mathbb{S}$ of a real vector space is called a regulated 
		function if it has one-sided limits  at every point $x\in \mathbb{S}$ and is called 
		quasiconvex if  for all $x,y\in \mathbb{S}$ and $\lambda\in [0,1]$ we have 
		$$f(\lambda x+(1-\lambda)y)\leq \text{Max} \{f(x), f(y)\}.$$
If furthermore	 
		$$f(\lambda x+(1-\lambda)y)< \text{Max} \{f(x), f(y)\}$$
for all $x \neq y$ and $\lambda\in (0,1)$, then $f$ is {\it strictly quasiconvex}.\\

In Sections 2, we present a sequence $\{\psi_n\}_{n\in N}$ 
from $\textbf{C}$  such that:
\begin{align*}
\mathfrak{I}(\psi_n)\longrightarrow \inf_{\psi\in \textbf{C}} \mathfrak{I} (\psi)~~~~\text{when} 
~~~~n\longrightarrow 
\infty.
\end{align*}
In Section 3, the issue of selecting the basis type will be added to the 
learning process. Section 4 is an alternative for Section 2. Moreover, how classifiers assist in upgrading the learning model will be discussed. In Section 5, by providing algorithms, implementation of the methods presented in
Sections 2, 3 and 4 will be clarified. In Section 6, we will theoretically clarify the superiority of our method over other articles in the realm of Learning and Inverse Problems. For this, we will do it with the aid of the following inclusion relations: 
$$B^{\prime\prime}_{W,p}\subseteq B_{W,P} \subseteq \textbf{C},$$
$$B^{\prime\prime}_{W,P}\subseteq B^{\prime}_{W,P} \subseteq \textbf{C};$$ 
where $B^{\prime\prime}_{W,P}$ and $B^{\prime}_{W,P}$ will be defined there. Moreover, we will address the importance of parameter learning in the case of not having large training sets or a lack of sufficient processing instruments.   Finally, in Section 7, the paper will conclude with two ideas for utilizing the learning method in {\it Speech Processing} and {\it Satellite Image Processing}.
\section{Finding the best regularization term  for a training set}
Solving the minimization problem (1.7) originated from a learning process for a signal processing application introduced in the introduction is the main part of this section and
will be provided in Theorem 2.5. As already mentioned, a suitable regularization term regarding a training set will be found through this minimization problem. So, we begin with the following simple Facts which will be later used in proving this theorem: 

\paragraph{Fact 1.}
{\textit {If $X$ is a Metric space, $h:X\longrightarrow \mathbb{R}$ is a continuous function 
and bounded 
from below, $X_n\subseteq X$ such that $X_n\subseteq X_{n+1}$ for every 
$n\in\mathbb{N}$ and $\overline {\bigcup_{n=1}^{+\infty} X_n}=X$ and moreover 
$|h(x_n)-\text{inf}_{x\in X_n}h(x)|\longrightarrow 0$ when 
$n\longrightarrow\infty$, then 
$h(x_n)\longrightarrow 
\text{inf}_{x\in X} h(x)$ when $n\longrightarrow \infty$.}}
\paragraph{Fact 2.}
{\textit {If $\psi=\sum_{t=1}^na_t\chi_{A_t}$ is a step function, then $\psi$ is a 
quasiconvex function if and only if for every $r,s,t\in \{1,\cdots , n\}$ such 
that $r<s<t$, we have $a_s\leq \text{Max}\{a_r,a_t\}$.}}
\paragraph{Fact 3.}
{\textit {If $\mathcal{H}$
is a Hilbert space and $\{\varphi_\gamma\}_{\gamma\in \Gamma}$ is an
orthonormal 
basis for $\mathcal{H}$, then 
$$\|\sum_{\gamma\in\Gamma}<t,\varphi_\gamma>\varphi_\gamma\|^2=
\sum_{\gamma\in\Gamma}|<t,\varphi_\gamma>|^2$$
for every $t\in\mathcal{H}$.}} 
\paragraph{Fact 4.}
{\textit {If $M_1,M_2\in\mathbb{Z}$, $\mathcal{H}$ is  a Hilbert space and 
$\{\varphi_\gamma\}_{r\in\Gamma}$ 
is an 
orthonormal basis for $\mathcal{H}$ and also $F_\gamma:(M_1,M_2)\longrightarrow 
\mathbb{R}$ are functions and 
bounded from below for every $\gamma\in\Gamma$, then 
$$\text{arg-min}_{f\in{\mathcal{H}}_{(M_1,M_2)}}
\sum_{\gamma\in \Gamma}F_\gamma(\langle 
f,\varphi_\gamma\rangle)=\sum_{\gamma\in\Gamma}(\text{arg-min}_{x\in 
	(M_1,M_2]}F_\gamma(x))\varphi_\gamma$$
where ${\mathcal{H}}_{(M_1,M_2]}=\{f\in{\mathcal{H}}|~f_\gamma\in (M_1,M_2]~
\text{for  
	every }
\gamma\in\Gamma\}$.}} 
\paragraph{Fact 5.} 
{\textit {Suppose $a,b \in \mathbb{R}$ and $M_1,M_2\in \mathbb{Z}$ is large 
enough such 
that $a,b\in (M_1,M_2]$. Define the set ${\mathcal{A}}_n$ for every $n\in 
\mathbb{N}$ 
by 
${\mathcal{A}}_n=\{\varphi|~\varphi=\sum_{t=1}^{(M_2-M_1)\times 2^n} 
x_t\chi_{A_t}$;    
$A_t=(M_1+\frac{t-1}{2^n}, M_1+\frac{t}{2^n}], \varphi ~ \text{is a quasiconvex function}\}$. Furthermore, define the function $f$ as follows:
\begin{align*}
&f:(M_1,M_2]\longrightarrow \mathbb{R}\\
&f(x)=(x-b)^2.
\end{align*}
Moreover, assume that $\{\psi_n\}_{n=1}^{+\infty}$ is a sequence  such that 
$\psi_n\in{\mathcal{A}}_n$ for every $n\in\mathbb{N}$ and for $f^* \in \overline{\{f+\psi_n|~n\in\mathbb{N}\}}$, we have $f^*$ is a strictly quasiconvex
function. Then, if $|(f+\psi_n)(a)-\min\{(f+\psi_n)(x)|~x\in (M_1,M_2]\}|\longrightarrow 0$ when 
$n\longrightarrow 
+\infty$, we will have 
$$|a-\text{arg-min}_{x\in (M_1,M_2]}(f+\psi_n)(x)|\longrightarrow 0 \quad 
\text{as} ~~\; n\longrightarrow +\infty.$$}}\\

Moreover, we also need these following two lemmas for proving Theorem 2.5. By minor modifications of Theorem 4.3
in [3] and {\it Fact 2}, the following lemma 
is achieved:
\begin{lemma}\label{2.1}
Suppose $S_n=\{1,\cdots ,(M_2-M_1)\times 2^n\}$ and sets $\mathcal{A}$, $\mathcal{A}_n$ for every $n \in \mathbb{N}$  as 
follows:
\begin{align*}
&\mathcal{A}=\{ \psi \big |~\psi:(a,b]\rightarrow \mathbb{R},~\psi ~ \text{is 
	a regulated and quasiconvex function}\},\\
&\mathcal{A}_n=\Big
 \{ \psi \big |~\psi =\sum_{t=1}^{(b-a) \times 2^n}a_t \chi_{A_t};~ 
A_t =(a+\dfrac{t-1}{2^n},a+\dfrac{t}{2^n}] ~\text{for} ~1\leq t\leq (b-a) 
\times 2^n,\\
&~~~~~~~~~~~~~\forall (r,s,t) \in \lbrace (r,s,t)\big |~ r,s,t \in S_n,~ r<s<t \rbrace :~a_s\leq \max \{a_r,a_t \} \Big\}.
\end{align*}
Then
\begin{equation*}
\overline{\bigcup_{n=1}^{\infty} \mathcal{A}_n}=\mathcal{A}.
\end{equation*}
\end{lemma}
Furthermore, by {\it Fact 5}, the following lemma is achieved as well: 
\begin{lemma}\label{2.2}
Suppose $a_1,a_2,b_1,b_2 \in \mathbb{R}$ and $M_1,M_2 \in \mathbb{Z}$ are large 
enough
such that $a_1,a_2,b_1,b_2 \in (M_1,M_2]$. Define the set $\mathcal{A}_n$ 
for every $n \in \mathbb{N}$ by
\begin{align*}
\mathcal{A}_n&=\{\varphi \big|~\varphi = \sum_{t=1}^{(M_2-M_1)
\times 2^n}x_t 
\chi_{A_t};~ A_t =(M_1+\dfrac{t-1}{2^n},M_1+\dfrac{t}{2^n}],~\varphi ~\text{is quasiconvex} \}.
\end{align*}
Furthermore, define the functions $f,g,h_1,h_2$ as follows:
\begin{align*}
&f,g:(M_1,M_2]\rightarrow \mathbb{R}\\
&f(x)=(x-b_1)^2,\\
&g(x)=(x-b_2)^2,\\
\\
&h_1,h_2:\overline{\cup_{n=1}^{\infty} \mathcal{A}_n}\longrightarrow 
\mathbb{R}
\end{align*}
\begin{align*}
h_1(\psi)&=\vert
 (f+\psi)(a_1)- \min\lbrace (f+\psi)(x)\big|~ x \in (M_1,M_2] 
\rbrace \vert^2 \\
&~~+ \vert (g+\psi)(a_2)- \min \lbrace (g+\psi)(x)\big|~ x \in (M_1,M_2] \rbrace \vert^2,\\
h_2(\psi)&=\vert a_1 -\text{arg-min}_{x \in (M_1,M_2]}(f+\psi)(x) \vert^2 \\
&~~+ \vert a_2 -\text{arg-min}_{ x \in (M_1,M_2]}(g+\psi)(x) \vert^2.
\end{align*}
Moreover,
$\inf_{\psi \in \overline{\cup_{n=1}^{\infty} \mathcal{A}_n}}h_1(\psi)=0$
and 
suppose $\lbrace \psi_n \rbrace_{n=1}^{\infty}$ is a sequence such 
that $\psi_n \in \mathcal{A}_n$ for every
$n\in \mathbb{N}$ and for $f^* \in \overline{\{f+\psi_n|~n\in\mathbb{N}\}}$, $g^* \in \overline{\{g+\psi_n|~n\in\mathbb{N}\}}$ we have $f^*$, $g^*$ are strictly quasiconvex
functions. Then, if $h_1(\psi_n)\longrightarrow 
\inf_{\psi\in\overline{\cup_{n=1}^{\infty} \mathcal{A}_n}}h_1(\psi)$ when 
$n\rightarrow \infty$, we will have
$$h_2(\psi_n )\longrightarrow 
\inf_{\psi\in \overline{\cup_{n=1}^{\infty} \mathcal{A}_n}}h_2(\psi)~~
\text{as}~~ n\longrightarrow +\infty.$$
\end{lemma}
To prove Theorem 2.5, the following proposition plays a key role:
\begin{proposition}\label{2.3}
Suppose $a_1,a_2,b_1,b_2 \in \mathbb{R}$ and $M_1,M_2 \in \mathbb{Z}$ are large 
enough such that $a_1,a_2,b_1,b_2 \in (M_1,M_2]$. 
Moreover, define the set  $\mathcal{A}$ and the function  $h$  from  
$\mathcal{A}$ to $\mathbb{R}$ by
\begin{align*}
\mathcal{A}= \lbrace \psi \big |~\psi:(M_1,M_2]\rightarrow \mathbb{R};~&\psi ~ 
\text{is a regulated and quasiconvex function}\},
\end{align*}
\begin{align*}
&h:\mathcal{A}\rightarrow  \mathbb{R}\\
&h(\psi ) =\Big\vert \min \lbrace F_{\psi}(x)\vert ~x \in (M_1,M_2]\rbrace - F_{\psi}(a_1) \Big\vert^2\\
&~~~~~~~~+ \Big\vert \min \lbrace G_{\psi}(x)\vert ~x \in (M_1,M_2] \rbrace - G_{\psi}(a_2)\Big\vert^2
\end{align*}
where $F_{\psi}(x)  = (x - b_1)^2 + \psi(x) $ and $G_{\psi}(x) =(x-b_2)^2 
+\psi(x) $.
Then there exists a sequence of step functions $\lbrace \psi_n \rbrace_{n \in \mathbb{N}}$ in $\mathcal{A}$ such that 
$$h(\psi_n)\longrightarrow \inf_{\psi \in \mathcal{A}} h(\psi) \qquad \text{as} \qquad n\rightarrow 
\infty.$$
\end{proposition}
\begin{proof}
Suppose $S_n=\{1,\cdots ,(M_2-M_1)\times 2^n\}$ and $\lbrace 
\epsilon _n \rbrace_{n=1}^{\infty} $ is  a sequence such that $\lim 
_{n\rightarrow \infty} \epsilon _n =0$. Define the set $D_n$ by 
\begin{align*}
D_n=\Big\{ &(x_1,x_2, \cdots , x_{(M_2-M_1) \times 2^n}) \in 
\mathbb{R}^{(M_2-M_1) \times 2^n}\big |\\
& \forall (r,s,t) \in \lbrace (r,s,t)\big |~r,s,t \in S_n,~ r<s<t \rbrace ;~x_s\leq \max \{x_r,x_t \}
\Big\}.
\end{align*}
Also suppose $t_1=\left[ (b_1-M_1)\times 2^n \right] +1$, $s_1=\left[ 
(a_1-M_1)\times 2^n \right] +1$, $t_2=\left[ (b_2-M_1)\times 2^n \right] 
+1$, $s_2=\left[ (a_2-M_1)\times 2^n \right] +1$ and define the function $K_n$ 
from $D_n$ to $\mathbb{R}$ by 
\begin{align*}
K_n\left(x_1,x_2, \cdots , x_{(M_2-M_1) \times 2^n} \right) &=\Big\vert \min 
\lbrace (M_1 + \dfrac{t}{2^n}-b_1)^2+x_t -(a_1-b_1)^2 -x_{s_1} \big |~1\leq t 
\leq t_1-1 \rbrace \\
&\cup \lbrace x_{t_1}-(a_1-b_1)^2 -x_{s_1} \rbrace \\
&\cup \lbrace (M_1 + \dfrac{t-1}{2^n}+\epsilon_n -b_1)^2 +x_t -(a_1-b_1)^2 
-x_{s_1}\big |\\
&t_1+1\leq t \leq (M_2-M_1) \times 2^n \rbrace \Big\vert^2\\
&+ \Big\vert \min \lbrace (M_1 + \dfrac{t}{2^n}-b_2)^2+x_t -(a_2-b_2)^2 
-x_{s_2} \big |~1\leq t \leq t_2-1 \rbrace \\
&\cup \lbrace x_{t_2}-(a_2-b_2)^2 -x_{s_2} \rbrace \\
& \cup \{ (M_1 + \dfrac{t-1}{2^n}+\epsilon_n -b_2)^2 +x_t -(a_2-b_2)^2 
-x_{s_2}\big |\\
&t_2+1\leq t \leq (M_2-M_1) \times 2^n \} \Big\vert ^2.
\end{align*}
Let $(\mathcal{C}_1^n , \mathcal{C}_2^n , \cdots , \mathcal{C}^n_{(M_2-M_1) 
	\times 2^n})$ be a minimizer of the function $K_n$ for every  $n \in 
\mathbb{N}$. Define the step function  $\psi _n$ for every  $n \in \mathbb{N}$ 
on $(M_1,M_2]$ as follows:
\begin{equation*}
\psi _n = \sum_{t=1}^{(M_2-M_1) \times 2^n} \mathcal{C}_t^n 
\chi_{{A}_t}
\end{equation*}
such that 
$${A}_t = \left( M_1 + \dfrac{t-1}{2^n} , M_1 + \dfrac{t}{2^n} 
\right];~1\leq t \leq (M_2-M_1) \times 2^n .$$
Now, we define the sets  $\mathcal{A}_n$ for  $n \in \mathbb{N}$:
\begin{align*}
\mathcal{A}_n&= \{\varphi \big |~ \varphi = \sum_{t=1}^{(M_2-M_1) \times 2^n} 
x_t \chi_{{A}_t};~ \forall r,s,t \in S_n,~ r<s<t  ;~x_s\leq \max \{x_r,x_t \} \}.
\end{align*}
Given {\it Fact 1}, it will be enough to prove the following items:
\begin{itemize}
	\item[{\it (i)}] 
	$\mathcal{A}$ with {\it Supremum metric} is a metric space.
	\item[{\it (ii)}]
	$h:\mathcal{A}\rightarrow  \mathbb{R}$ is a continuous function and bounded 
	from below.
	\item[{\it (iii)}]
	$\mathcal{A}_n \subseteq \mathcal{A}_{n+1}$ for every  $n \in \mathbb{N}$.
	\item[{\it (iv)}]
	$\overline{\cup_{n=1}^{\infty} \mathcal{A}_n}=\mathcal{A}.$
	\item[{\it (v)}]
	$\vert h(\psi_n)- \inf_{\psi \in \mathcal{A}_n} h(\psi) \vert  $ 
	tends to $0$ as ~$n \rightarrow \infty$.
\end{itemize}
$(i),(ii)$ and  $(iii)$ are clearly correct and we have  $(iv)$ by 
Lemma 2.1.  It remains to prove only $(v)$: Pick arbitrary $\varphi \in \mathcal{A}_n$ and suppose $\left(x_1,x_2, \cdots , 
x_{(M_2-M_1) \times 2^n} \right) \in D_n$ such that $\varphi(x)= 
\sum_{t=1}^{(M_2-M_1) \times 2^n} x_t \chi_{{A}_t} (x)$ for  $x \in 
(M_1,M_2]$. We have 
$$F_{\varphi}(x)=(x-b_1)^2 + \varphi (x);\qquad x \in \left( M_1,M_2\right] $$
or equivalently,
$$F_{\varphi}(x)=(x-b_1)^2 + x_t;\quad x \in {A}_t=\left( M_1 + 
\dfrac{t-1}{2^n} , M_1 + \dfrac{t}{2^n} \right];~1\leq t \leq (M_2-M_1) \times 
2^n .$$
We consider separately two  different cases:\\

\textbf{Case 1}. $x \in (M_1,b_1]$. Because in this case the function  
$F_{\varphi}\big |_{{A}_t}$ is decreasing, the minimum value of the 
function over the interval ${A}_t$ will be at the end point $M_1 + 
\dfrac{t}{2^n} $ and consequently 
\begin{equation*}
\min_{x \in {A}_t} F_{\varphi} (x)=\left( M_1 + \dfrac{t}{2^n}  - b_1 
\right)^2 + x_t.
\end{equation*}
Since $(M_1,b_1]=\cup_{t=1}^{t_1-1} {A}_t \cup (M_1 + 
\dfrac{t_1-1}{2^n}, b_1] $ such that $t_1=\left[ (b_1-M_1)\times 2^n \right] 
+1$, we have 
\begin{align*}
\min_{x\in (M_1,b_1]}
F_{\varphi}(x) = \min[ \lbrace \left( M_1 + \dfrac{t}{2^n}  - b_1 
\right)^2 + x_t \big |~ 1\leq t \leq t_1-1 \rbrace \cup \lbrace x_{t_1} 
\rbrace]. 
\tag*{$(2\cdot 1)$} 
\end{align*}

\textbf{Case 2}. $x\in (b_1,M_2]$. Because in this case the function  
$F_{\varphi}\big |_{{A}_t}$ is increasing, the minimum value of the 
function over the interval  ${A}_t$ will be at the endpoint $M_1 + 
\dfrac{t-1}{2^n}$. However, since  $M_1 + \dfrac{t-1}{2^n} \not \in 
{A}_t $, we consider  $M_1 + \dfrac{t-1}{2^n}  + \epsilon_n$ as an 
approximation for the minimizer of  $F_{\varphi}\big |_{{A}_t}$ and 
hence 
$$\min_{x \in {A}_t} F_{\varphi}(x) \simeq (M_1 + \dfrac{t-1}{2^n}  + 
\epsilon_n - b_1)^2 + x_t.$$
Since $(b_1, M_2]=(b_1,M_1 + \dfrac{t_1}{2^n}) \cup
\displaystyle{\cup_{t=t_1+1}^{(M2-M_1)\times 2^n}} {A}_t $, we have 
\begin{align*}
\min_{x\in (b_1,M_2]}
F_{\varphi}(x) \simeq \min[&\lbrace (M_1+\dfrac{t-1}{2^n}+\epsilon_n
- b_1)^2 + x_t \big| ~t_1+1\leq t \leq (M_2-M_1)\times
2^n \rbrace \\ & \cup \lbrace 
\epsilon_n + x_{t_1} \rbrace ].
\tag*{$(2\cdot 2)$}
\end{align*}

On the other hand, since $a_1\in (M_1+\dfrac{[(a_1-M_1) \times 2^n ]}{2^n 
},M_1+\dfrac{[(a_1-M_1) \times 2^n ]+1}{2^n }]$, it follows $a_1 \in 
{A}_{s_1}$ such that  $s_1=[(a_1-M_1) \times 2^n ]+1$ and therefore,
\begin{equation*}
F_{\varphi}(a_1) = (a_1-b_1)^2 + x_{s_1}.
\tag*{$(2\cdot 3)$}
\end{equation*}

We deduce from  (2.1), (2.2)
and (2.3) 
that
\begin{align*}
\Big\vert  &\min_{x \in (M_1,M_2]} F_{\varphi}(x) -F_{\varphi}(a_1) \vert 
\simeq \vert \min \lbrace (M_1+\dfrac{t}{2^n}-b_1)^2+x_t -(a_1-b_1)^2 
-x_{s_1}\big |~1\leq t \leq t_1-1 \rbrace \\
& \cup \lbrace x_{t_1} - (a_1-b_1)^2 - x_{s_1} \rbrace \cup \lbrace 
(M_1+\dfrac{t-1}{2^n}+\epsilon_n -b_1)^2+x_t -(a_1-b_1)^2 -x_{s_1}\big |\\
&t_1+1\leq t \leq (M_2-M_1)\times 2^n \rbrace\Big\vert.
\tag*{$(2\cdot 4)$}
\end{align*}

By a similar process like what was done for $F_{\varphi}$, we will have:
\begin{align*}
\Big\vert  &\min_{x \in (M_1,M_2]} G_{\varphi}(x) -G_{\varphi}(a_2) \Big\vert 
\simeq \Big\vert \min \lbrace (M_1+\dfrac{t}{2^n}-b_2)^2+x_t -(a_2-b_2)^2 
-x_{s_2}\big |~1\leq t \leq t_2-1 \rbrace \\
& \cup \lbrace x_{t_2} - (a_2-b_2)^2 - x_{s_2} \rbrace \cup \lbrace 
(M_1+\dfrac{t-1}{2^n}+\epsilon_n -b_2)^2+x_t -(a_2-b_2)^2 -x_{s_2}\big |\\
&t_2+1\leq t \leq (M_2-M_1)\times 2^n \rbrace \Big\vert.
\tag*{$(2\cdot 5)$}
\end{align*}
Then since $\varphi = \sum_{t=1}^{(M_2-M_1)\times 2^n}x_t 
\chi_{{A}_t}$, by (2.4), (2.5), 
$$h\big |_{\mathcal{A}_n}(\varphi)\simeq K_n\left(x_1,x_2, \cdots , 
x_{(M_2-M_1) \times 2^n} \right) .$$
Consequently 
$$h(\psi_n)\simeq \inf _{\psi \in \mathcal{A}_n}h(\psi).$$
\\
Therefore, since  $\epsilon_n$~tends to $0$~as $n$ tends to $\infty$, we 
conclude:
\begin{equation*}
\Big\vert  h(\psi_n)- \inf _{\psi \in \mathcal{A}_n}h(\psi) \Big\vert 
\rightarrow 0\quad  \text{as} \quad n\rightarrow \infty .
\end{equation*}
\end{proof}

\begin{remark}\label{2.4}
To prove Theorem 2.5, a lemma is needed with assumptions weaker than those of Lemma 2.2, considering $\inf_{\psi \in \overline{\cup_{n=1}^{\infty} \mathcal{A}_n}}h_1( \psi)\geq 0$ instead of $\inf_{\psi \in \overline{\cup_{n=1}^{\infty} \mathcal{A}_n}}h_1( \psi)=0$ in Lemma 2.2. 
 Although it is not needed to provide more conditions to prove Lemma 2.2 than quasiconvexity for the step functions in the sets $\mathcal{A}_n$ for $n\in \mathbb{N}$ as well as $f^*,~g^*$ (introduced in Lemma 2.2) being strictly quasiconvex, we cannot be sure about the general case $\inf_{\psi \in \overline{\cup_{n=1}^{\infty} \mathcal{A}_n}}h_1( \psi)\geq 0$ that doesn't need more conditions. 
 Moreover, a condition to ensure that  $f^*$, $g^*$ (introduced in Lemma 2.2) be strictly quasiconvex will not be included in the remainder of the paper, because replacing inequalities with strict inequalities in the set $D_n$ in  Theorem 2.5 and consequently in (2.6), (2.7) in Example 2.6 - however not completely addressing theoretical concerns - will provide a sufficient guarantee for desired outcomes in computational works. 
  
However, we will proceed with the paper only by considering the quasiconvexity condition for the step functions in $\mathcal{A}_n$  for every $n\in \mathbb{N}$, because if the lemma needs to provide more conditions than the quasiconvexity on the step functions for completion to occur, there will be slight changes in the process of continuing the paper - adding the new conditions to the ones of the functions in sets $D_n,\textbf{C}$ in  Theorem 2.5, and consequently adding these conditions to the constraints of the constrained optimization problems in (2.6), (2.7) in Example 2.6, is needed. 
\end{remark}
For the remainder of this paper, suppose $\mathcal{H}$
be a Hilbert space and $(\varphi_\gamma)_{\gamma\in\Gamma}$     
be an orthonormal basis for $\mathcal{H}$
and $\mathcal{H}_{(M_1,M_2]}$ be as defined in {\it Fact 4}. Furthermore, we 
use the shorthand notation $f_\gamma$
for  $\langle f,\varphi_\gamma\rangle, (f_1)_{\gamma}$ 
for $\langle f_1,\varphi_\gamma\rangle$,  etc. and consider $\langle 
f,\varphi_\gamma\rangle\in \mathbb{R}$
for every $f\in \mathcal{H}$.

As mentioned in the introduction, the goal of this section is to find a 
denoising model by using the signals sets $\{f_1,\cdots ,f_m\}$
and $\{g_1,\cdots ,g_m\}$ where $m\in \mathbb{N}$. We will do it in the 
following theorem, however, for convenience, we will consider $m=2$. General 
condition will be achieved with some slight changes.  

\begin{theorem}\label{2.5}
Suppose $f_1,f_2,g_1,g_2 \in \mathcal{H}$ and $M_1, M_2 \in \mathbb{Z}$ is  
large enough such that $( f_1 )_{\gamma},( f_2 )_{\gamma},( g_1 )_{\gamma},( 
g_2 )_{\gamma} \in (M_1,M_2]$ for every $\gamma \in \Gamma$. 
Moreover, define the set  $\textbf{C}$ and the function  $\mathfrak{I}$  from  
$\textbf{C}$ to $\mathbb{R}$ by
\begin{align*}
\textbf{C}=\lbrace \psi:\mathcal{H} \rightarrow \mathbb{R}\big 
|~&\psi(f)=\sum_{\gamma \in \Gamma} \psi_{\gamma}(f_{\gamma});~ \forall \gamma 
\in \Gamma:~\psi_{\gamma}:(M_1,M_2]\rightarrow \mathbb{R}\\
~ &\text{is a regulated and quasiconvex function}\rbrace ,\\
&\mathfrak{I}:\textbf{C}\longrightarrow \mathbb{R}\\
&\mathfrak{I}(\psi)=\sum_{i=1}^2 \Vert \tilde{f}_{i,\psi} - f_i \Vert^2\\
&;\tilde{f}_{i,\psi} =\text{arg-min}_{f
	\in \mathcal{H}_{(M_1,M_2]}} 
\Vert f- g_i \Vert ^2 +\psi(f).
\end{align*}
Then there exists a sequence $\lbrace\psi_n\rbrace_{n \in {\mathbb{N}}}$ in $\textbf{C}$ such that the functions $\lbrace\psi_\gamma\rbrace_{\gamma \in {\Gamma}}$ corresponding to each $\psi_n$ are step functions and
$$\mathfrak{I}(\psi_n)\longrightarrow \inf_{\psi \in 
	\textbf{C}}\mathfrak{I}(\psi)\quad \text{as} \quad n\rightarrow \infty.$$
\end{theorem}

\begin{proof}
Suppose $S_n=\{1,\cdots ,(M_2-M_1)\times 2^n\}$ and 
$\lbrace \epsilon_n \rbrace_{n=1}^{\infty}$ is a sequence such that $\lim 
_{n\rightarrow \infty}\epsilon_n =0$. Define the set $D_n$by
\begin{align*}
D_n= \Big\{& (x_1,x_2, \cdots , x_{(M_2-M_1) \times 2^n}) \in 
\mathbb{R}^{(M_2-M_1) \times 2^n}\big |\\
& \forall (r,s,t) \in \lbrace (r,s,t)\big |~r,s,t \in S_n,~ r<s<t \rbrace ;~x_s\leq \max \{x_r,x_t \} \Big\}.
\end{align*}
Also, suppose $t_{1,\gamma}=[(( g_1 )_{\gamma} -M_1)\times 
2^n]+1$, $s_{1,\gamma}=[((
 f_1 )_{\gamma} -M_1)\times 2^n]+1$, $t_{2,\gamma}=[(( 
g_2 )_{\gamma} -M_1)\times 2^n]+1$, $s_{2,\gamma}=[(( f_2 )_{\gamma} 
-M_1)\times 2^n]+1$ and define the function  $K_n^{\gamma}$ from $D_n$ to 
$\mathbb{R}$ for every $\gamma \in \Gamma$ by\\
\begin{align*}
K_n^{\gamma} \left(x_1,x_2, \cdots , x_{(M_2-M_1) \times 2^n} \right) 
&=\Big\vert \min \lbrace (M_1 + \dfrac{t}{2^n}-( g_1 )_{\gamma})^2+x_t -(( f_1 
)_{\gamma}-( g_1 )_{\gamma})^2 -x_{s_{1,\gamma}} \big | \\
&~1\leq t \leq t_{1,\gamma}-1 \rbrace \cup \lbrace x_{t_{1,\gamma}}-(( f_1 
)_{\gamma}-( g_1 )_{\gamma})^2 -x_{s_{1,\gamma}} \rbrace \\
&\cup \lbrace (M_1 + \dfrac{t-1}{2^n}+\epsilon_n -( g_1 )_{\gamma})^2 +x_t -(( 
f_1 )_{\gamma}-( g_1 )_{\gamma})^2 -x_{s_{1,\gamma}}\big |\\
&t_{1,\gamma}+1\leq t \leq (M_2-M_1) \times 2^n \rbrace \Big\vert^2\\
&+ \Big\vert \min \lbrace (M_1 + \dfrac{t}{2^n}-( g_2 )_{\gamma})^2+x_t -(( f_2 
)_{\gamma}-( g_2 )_{\gamma})^2 -x_{s_{2,\gamma}} \big |\\
&~1\leq t \leq t_{2,\gamma}-1 \rbrace  \cup \lbrace x_{t_{2,\gamma}}-(( f_2 
)_{\gamma}-( g_2 )_{\gamma})^2 -x_{s_{2,\gamma}} \rbrace \\
& \cup \{ (M_1 + \dfrac{t-1}{2^n}+\epsilon_n -( g_2 )_{\gamma})^2 +x_t -(( f_2 
)_{\gamma}-( g_2 )_{\gamma})^2 -x_{s_{2,\gamma}}\big |\\
&t_{2,\gamma}+1\leq t \leq (M_2-M_1) \times 2^n \} \Big\vert ^2.
\end{align*}
Let $\left(\mathcal{C}_1^{n,\gamma},\mathcal{C}_2^{n,\gamma}, \cdots , 
\mathcal{C}_{(M_2-M_1) \times 2^n}^{n,\gamma} \right)$ be a minimizer of the 
function $K_n^{\gamma}$ for every $n\in \mathbb{N}$. Define the function 
$\psi_n$ for every $n\in \mathbb{N}$, from $\mathcal{H}$ to $\mathbb{R}$ as 
follows:
\begin{align*}
&\psi_n:\mathcal{H}\rightarrow \mathbb{R}\\
&\psi_n(f)=\sum_{\gamma \in \Gamma} \psi_{n,\gamma}(f_{\gamma}),
\end{align*}
where  $\psi_{n,\gamma}$ for every $n\in \mathbb{N}$ and $\gamma \in \Gamma$ is 
a step function from $(M_1,M_2]$ to  $\mathbb{R}$ which is defined by 
\begin{equation*}
\psi_{n,\gamma} = \sum_{t=1}^{(M_2-M_1) \times 2^n} \mathcal{C}_t^{n,\gamma} 
\chi_{{A}_t}
\end{equation*}
such that 
${A}_t = \left( M_1 + \dfrac{t-1}{2^n} , M_1 + \dfrac{t}{2^n} 
\right];~1\leq t \leq (M_2-M_1) \times 2^n .$
Now, we define the set  $\mathcal{A}$:
\begin{equation*}
\mathcal{A}=\{ \psi \big |~\psi:(M_1,M_2]\rightarrow \mathbb{R},~\psi ~ 
\text{is a regulated and quasiconvex function}\}.
\end{equation*}
Also the functions  $F_{\psi}^{\gamma}$ and  $G_{\psi}^{\gamma}$ from 
$\mathbb{R}$ to themselves are defined by 
$$F_{\psi}^{\gamma}(x)=\left( (g_1)_{\gamma} -x \right)^2 + \psi(x), \quad 
G_{\psi}^{\gamma}(x)=\left( (g_2)_{\gamma} -x \right)^2 + \psi(x).$$
By Proposition 2.3
for every $\gamma\in \Gamma$ we have :
\begin{align*}
&\Big\vert F_{\psi_{n ,\gamma}}^{\gamma}((f_1)_{\gamma}) - \min_{x \in 
	(M_1,M_2]}F_{\psi_{n ,\gamma}}^{\gamma}(x) \Big\vert^2 + \Big\vert 
G_{\psi_{n ,\gamma}}^{\gamma}((f_2)_{\gamma})
- \min_{x \in (M_1,M_2]}G_{\psi_{n ,\gamma}}^{\gamma}(x) \Big\vert^2\\
&\longrightarrow \inf_{\psi_{\gamma} \in \mathcal{A}}\Big\vert 
F_{\psi_\gamma} 
^{\gamma}((f_1)_{\gamma})-\min_{x \in (M_1,M_2]}F_{\psi_\gamma }
^{\gamma}(x) \Big\vert^2 +\Big\vert G_{\psi_\gamma }
^{\gamma}((f_2)_{\gamma}) - \min_{x \in (M_1,M_2]}G_{\psi_\gamma} 
^{\gamma}(x) \Big\vert^2\\ 
& \qquad \qquad \qquad  \qquad \qquad \qquad  \text{as} \quad n\rightarrow 
\infty.
\end{align*}
Since by {\it Fact 2} $~\psi_{n,\gamma}$ is quasiconvex for every  $\gamma \in 
\Gamma$, we deduce from Lemma 2.2 and Remark  2.4
that 
\begin{align*}
&\Big\vert (f_1)_{\gamma} -\text{arg-min}_{x \in 
	(M_1,M_2]}\left[(x-(g_1)_{\gamma})^2+\psi_{n ,\gamma}(x) \right] 
\Big\vert^2\\
&+\Big\vert (f_2)_{\gamma} -\text{arg-min}_{x \in 
	(M_1,M_2]}\left[(x-(g_2)_{\gamma})^2+\psi_{n ,\gamma}(x) \right] 
\Big\vert^2 \\
&\longrightarrow \inf_{\psi_{\gamma} \in \mathcal{A}}\Big\vert (f_1)_{\gamma} 
-\text{arg-min}_{x
	\in (M_1,M_2]}\left[(x-(g_1)_{\gamma})^2+\psi_{\gamma}(x) 
\right] \Big\vert^2 \\
&+\Big\vert (f_2)_{\gamma} -\text{arg-min}_{x \in 
	(M_1,M_2]}\left[(x-(g_2)_{\gamma})^2+\psi_{\gamma}(x) \right] 
\Big\vert^2 \\
& \qquad \qquad \qquad  \qquad \qquad \qquad  \text{as} \quad n\rightarrow 
\infty.
\end{align*}
This implies that
\begin{align*}
&\sum_{\gamma \in \Gamma}\Big\vert (f_1)_{\gamma} -\text{arg-min}_{x \in 
	(M_1,M_2]}\left[(x-(g_1)_{\gamma})^2+\psi_{n ,\gamma}(x) \right] 
\Big\vert^2 \\
&+\sum_{\gamma \in \Gamma}\Big\vert (f_2)_{\gamma} -\text{arg-min}_{x \in 
	(M_1,M_2]}\left[(x-(g_2)_{\gamma})^2+\psi_{n ,\gamma}(x) \right] 
\Big\vert^2 \\
&\longrightarrow \inf_{\lbrace \psi_{\gamma}\rbrace_{\gamma  \in 
		\Gamma}\subseteq \mathcal{A}}\sum_{\gamma \in \Gamma}\Big\vert 
(f_1)_{\gamma} 
-\text{arg-min}_{x
	\in (M_1,M_2]}\left[(x-(g_1)_{\gamma})^2+\psi_{\gamma}(x) 
\right] \Big\vert^2 \\
&+\sum_{\gamma \in \Gamma} \Big\vert (f_2)_{\gamma} -\text{arg-min}_{x \in 
	(M_1,M_2]}\left[(x-(g_2)_{\gamma})^2+\psi_{\gamma}(x) \right] 
\Big\vert^2 \\
& \qquad \qquad \qquad  \qquad \qquad \qquad  \text{as} \quad n\rightarrow 
\infty.
\end{align*}
Then by {\it Fact 3}, we will have
\begin{align*}
&\Big\Vert \sum_{\gamma \in \Gamma}\left[(f_1)_{\gamma} -\text{arg-min}_{x
	\in 
	(M_1,M_2]}(x-(g_1)_{\gamma})^2+\psi_{n ,\gamma}(x)  
\right]\varphi_{\gamma} \Big\Vert^2\\
&+\Big\Vert\sum_{\gamma\in\Gamma}
\left[(f_2)_{\gamma} -\text{arg-min}_{x \in 
	(M_1,M_2]}(x-(g_2)_{\gamma})^2+\psi_{n ,\gamma}(x) \right] 
\varphi_{\gamma} \Big\Vert^2 \\
&\longrightarrow \inf_{\lbrace \psi_{\gamma}\rbrace_{\gamma  \in 
		\Gamma}\subseteq \mathcal{A}} \Big\Vert \sum_{\gamma \in \Gamma} 
\left[(f_1)_{\gamma} -\text{arg-min}_{x
	\in (M_1,M_2]}(x-(g_1)_{\gamma})^2+\psi_{\gamma}(x) 
\right]\varphi_{\gamma} \Big\Vert^2 \\
&+
\Big\Vert\sum_{\gamma \in \Gamma}  \left[ (f_2)_{\gamma} -\text{arg-min}_{x
	\in 
	(M_1,M_2]}(x-(g_2)_{\gamma})^2+\psi_{\gamma}(x) 
\right]\varphi_{\gamma} \Big\Vert^2 \\
& \qquad \qquad \qquad  \qquad \qquad \qquad  \text{as} \quad n\rightarrow 
\infty.
\end{align*}
Therefore, by {\it Fact 4} we have
\begin{align*}
&\Big\Vert f_1 -\text{arg-min}_{f
	\in \mathcal{H}_{(M_1,M_2]}}\sum_{\gamma \in 
	\Gamma}(f_{\gamma}-(g_1)_{\gamma})^2+\sum_{\gamma \in \Gamma}\psi_{n 
	,\gamma}(f_{\gamma}) \Big\Vert^2 \\
&+\Big\Vert f_2 -\text{arg-min}_{f
	\in \mathcal{H}_{(M_1,M_2]}}\sum_{\gamma\in\Gamma}(f_{\gamma}-
(g_2)_{\gamma})^2+\sum_{\gamma\in\Gamma}
\psi_{n ,\gamma}(f_{\gamma}) \Big\Vert^2 
\\
&\longrightarrow \inf_{\lbrace \psi_{\gamma}\rbrace_{\gamma  \in 
		\Gamma}\subseteq \mathcal{A}}\Big\Vert f_1 -\text{arg-min}_{f \in 
	\mathcal{H}_{(M_1,M_2]}}\sum_{\gamma \in \Gamma}(f_{\gamma}-
(g_1)_{\gamma})^2+\sum_{\gamma \in \Gamma}\psi_{\gamma}(f_{\gamma})
\Big\Vert^2 \\
&+ \Big\Vert f_2 -\text{arg-min}
_{f \in \mathcal{H}_{(M_1,M_2]}}\sum_{\gamma\in\Gamma}(f_{\gamma}-
(g_2)_{\gamma})^2+\sum_{\gamma\in\Gamma}\psi_{\gamma}
(f_{\gamma})  \Big\Vert^2 \\
& \qquad \qquad \qquad  \qquad \qquad \qquad  \text{as} \quad n\rightarrow 
\infty.
\end{align*}
Finally, since 
$ \textbf{C}=\lbrace \psi \big |~\psi(f)=\sum_{\gamma \in \Gamma} 
\psi_{\gamma}(f_{\gamma});~\psi_{\gamma} \in \mathcal{A}~\text{for every} 
~\gamma \in \Gamma \rbrace ,$
we conclude:
\begin{align*}
&\Big\Vert f_1 -\text{arg-min}_{f
	\in \mathcal{H}_{(M_1,M_2]}} \Vert f-g_1 \Vert^2+\psi_n(f) \Big\Vert^2\\
&+\Big\Vert f_2 -\text{arg-min}_{f \in 
	\mathcal{H}_{(M_1,M_2]}} \Vert f-g_2 \Vert^2 +\psi_n(f) \Big\Vert^2 \\
&\longrightarrow 
\inf_{\psi \in \textbf{C}} \Big\Vert f_1 -\text{arg-min}_{f \in 
	\mathcal{H}_{(M_1,M_2]}} \Vert f-g_1 \Vert^2 +\psi(f) \Big\Vert^2\\
&+\Big\Vert f_2 -\text{arg-min}_{f \in \mathcal{H}_{(M_1,M_2]}} \Vert f-g_2 
\Vert^2 
+\psi(f) \Big\Vert^2 \\
& \qquad \qquad \qquad  \qquad \qquad \qquad  \text{as} \quad n\rightarrow 
\infty .
\end{align*}
\end{proof}
In  mathematical optimiztion, {\it Constrained Optimization} is the process of 
optimizating an objective function with respect to some variables in the 
presence of constraints on those variables. A general constrained minimization 
problem may be written as follows: 
\begin{align*}
& \text{min}  \quad f(x)\\
\text{subject to}&~~ g_i(x)=c_i~~ \text{for}~~ i=1,\cdots , n 
~~~\text{\it ``Equality constraints"}\\
&~~ h_j(x)\geq d_j~ \text{for}~ j=1,\cdots , m ~~~\text{\it ``Inequality constraints"}
\end{align*}
where $g_i(x)=c_i$ for $i=1,\cdots , n$ and $h_j(x)\geq d_j$ for $j=1,\cdots 
,m$  are constraints that are required to be satisfied, and $f(x)$ is the 
objective function that needs to be optimized subject to the constraints. 

Several optimization problems with inequality constraints need to be solved to 
successfully achieve each step function provided in Theorem 2.5. The 
following example serve to the clarification of this:
\begin{example}\label{2.6}Suppose that in the assumptions of Theorem 2.5, 
$\Gamma=\lbrace 
\gamma_1,\gamma_2 \rbrace ,M_1=1,M_2=3, \epsilon_n = \dfrac{1}{2^n}$ and  
$(f_1)_{\gamma_1},(f_2)_{\gamma_1},(g_1)_{\gamma_1},(g_2)_{\gamma_1},
(f_1)_{\gamma_2},(f_2)_{\gamma_2},(g_1)_{\gamma_2},(g_2)_{\gamma_2}$ are as 
follows:
\begin{align*}
&(f_1)_{\gamma_1}=2,\quad (f_2)_{\gamma_1}=2.8,\quad
(g_1)_{\gamma_1}=1.3,\quad 
(g_2)_{\gamma_1}=2.1\\
&(f_1)_{\gamma_2}=1.7,\quad (f_2)_{\gamma_2}=1.1,\quad 
(g_1)_{\gamma_2}=2.5,\quad (g_2)_{\gamma_2}=2.9.
\end{align*}
Then we will have:
$$\psi_{1,\gamma_1}=\sum_{t=1}^4 \mathcal{C}_t^{1,\gamma_1}\chi_{{A}_t},
\qquad \psi_{1,\gamma_2}=\sum_{t=1}^4 \mathcal{C}_t^{1,\gamma_2}\chi_{{A
	}_t}$$
where 
${A}_t=(1+\dfrac{t-1}{2},1+\dfrac{t}{2}];~1\leq t\leq 
4$,~$(\mathcal{C}_1^{1,\gamma_1},\mathcal{C}_2^{1,\gamma_1},\mathcal{C}_3^{1,
	\gamma_1},\mathcal{C}_4^{1,\gamma_1})$ and
$(\mathcal{C}_1^{1,\gamma_2},\mathcal{C}
_2^{1,\gamma_2},\mathcal{C}_3^{1,\gamma_2},\mathcal{C}_4^{1,\gamma_2})$ are 
respectively minimizer of the functions  $K_1^{\gamma_1},K_1^{\gamma_2}$ on 
$D_1$ which are  as follows:
\begin{align*}
&D_1=\Big\lbrace (x_1,x_2,x_3,x_4) \in \mathbb{R}^4\big |~x_2\leq \max \lbrace 
x_1,x_3 \rbrace, x_2\leq \max \lbrace x_1,x_4 \rbrace, \\
& \quad \quad \quad \quad \quad \quad \quad \quad \quad \quad \quad \quad x_3\leq \max \lbrace x_1,x_4 \rbrace, x_3\leq \max \lbrace x_2,x_4 \rbrace 
\Big\rbrace, \\
\\
&K_1^{\gamma_1} (x_1,x_2,x_3,x_4)  = \Big\vert \min \lbrace x_1-x_3-4.9 
\rbrace \cup \lbrace (\dfrac{t}{2}-0.3)^2 +x_t-X_3 -4.9\big |~1\leq t\leq 4 
\rbrace\Big\vert^2 \\
&+\Big\vert  \min \lbrace (\dfrac{t}{2}-1.1)^2
+x_t-x_4 -4.9\big |~1\leq t\leq 
2 \rbrace \cup \lbrace x_3-x_4 -4.9 \rbrace \cup \lbrace (\dfrac{t}{2}-1.1)^2  
-4.9 \rbrace \Big\vert^2 ,\\
\\
&K_1^{\gamma_2} (x_1,x_2,x_3,x_4)  = \Big\vert \min \lbrace 
(\dfrac{t}{2}-1.5)^2 +x_t -0.64-x_2\big |~1\leq t\leq 3 \rbrace \cup \lbrace 
x_4 -0.64 -x_2 \rbrace \Big\vert^2 \\
&+ \Big\vert  \min \lbrace (\dfrac{t}{2}-1.1)^2 -x_t-x_1 -3.24\big |~1\leq 
t\leq 3 \rbrace \cup \lbrace x_4-x_1 -3.24 \rbrace \Big\vert^2 .
\end{align*}
Since $Max \lbrace a,b \rbrace = \dfrac{\vert a-b \vert + (a+b)}{2}$ , 
~$(\mathcal{C}_1^{1,\gamma_1},\mathcal{C}_2^{1,\gamma_1},\mathcal{C}_3^{1,\gamma
	_1},\mathcal{C}_4^{1,\gamma_1})$ and $(\mathcal{C}_1^{1,\gamma_2},\mathcal{C}_2
^
{1,\gamma_2},\mathcal{C}_3^{1,\gamma_2},\mathcal{C}_4^{1,\gamma_2})$ can be 
considered as the solution of the following constrained optimization problems:
\begin{align*}
&\min K_1^{\gamma_1} (x_1,x_2,x_3,x_4) \\
\text{subject to}\quad & \vert x_1 - x_3 \vert + x_1 + x_3 -2x_2\geq 0\\
& \vert x_1 - x_4 \vert + x_1 + x_4 -2x_2\geq 0\\
& \vert x_1 - x_4 \vert + x_1 + x_4 -2x_3\geq 0\\
& \vert x_2 - x_4 \vert + x_2 + x_4 -2x_3\geq 0,
\tag*{$(2\cdot 6)$}
\end{align*}
\begin{align*}
&\min K_1^{\gamma_2} (x_1,x_2,x_3,x_4) \\
\text{subject to}\quad & \vert x_1 - x_3 \vert + x_1 + x_3 -2x_2\geq 0\\
& \vert x_1 - x_4 \vert + x_1 + x_4 -2x_2\geq 0\\
& \vert x_1 - x_4 \vert + x_1 + x_4 -2x_3\geq 0\\
& \vert x_2 - x_4 \vert + x_2 + x_4 -2x_3\geq 0.
\tag*{$(2\cdot 7)$}
\end{align*}
\end{example}
The functions $\psi_\gamma$
for every $\gamma\in\Gamma$
in the set $\textbf{C}$, as defined in Introduction, were 
functions on $\mathbb{R}$.
But, in Theorem 2.5, these functions are defined in the interval 
$(M_1,M_2]$. 
It should be noted that this difference does not affect the performance in 
terms of computing tasks, because we can consider the interval 
$(M_1,M_2]$ to be large 
enough so that it would be unreasonable for $\langle f,\varphi_\gamma\rangle$
such that $\gamma\in\Gamma$ to be placed outside 
the interval for every signal $f$. However, theoretically, we can provide a proof, by taking into account the set  
$\textbf{C}$ as
presented in Introduction. We have done this in the following theorem, 
the proof of which can be easily deduced from Theorem 2.5: 
\begin{theorem}\label{2.7}
Suppose $f_1,f_2,g_1,g_2 \in \mathcal{H}$ and $M\in \mathbb{N}$ is  large  
enough such that
$( f_1)_{\gamma},( f_2 )_{\gamma},( g_1 )_{\gamma},( g_2 
)_{\gamma} \in (-M,M]$
for every $\gamma \in \Gamma$. Moreover, define the sets  $\textbf{C}^M$, $\textbf{C}$ and the functions  
$\mathfrak{I}^M,\mathfrak{I}$  respectively from  $\textbf{C}^M$, $\textbf{C}$  
to $\mathbb{R}$ by
\begin{align*}
\textbf{C}^M=\lbrace \psi:\mathcal{H} \rightarrow \mathbb{R}\big 
|~&\psi(f)=\sum_{\gamma \in \Gamma} \psi_{\gamma}(f_{\gamma});~ \forall \gamma 
\in \Gamma: ~\psi_{\gamma}:(-M,M]\rightarrow \mathbb{R}\\
~ &\text{is a regulated and quasiconvex function}\rbrace ,\\
\textbf{C}=\lbrace \psi:\mathcal{H} \rightarrow \mathbb{R}\big 
|~&\psi(f)=\sum_{\gamma \in \Gamma} \psi_{\gamma}(f_{\gamma});~ \forall \gamma 
\in \Gamma: ~\psi_{\gamma}:\mathbb{R}\rightarrow \mathbb{R}\\
~ &\text{is a regulated and quasiconvex function}\rbrace ,\\
\\
&\mathfrak{I}^M:\textbf{C}^M\longrightarrow \mathbb{R},\\
&\mathfrak{I}^M(\psi)=\sum_{i=1}^2 \Vert \tilde{f}_{i,\psi} - f_i 
\Vert_{\mathcal{H}}^2\\
&;\tilde{f}_{i,\psi} =\text{arg-min}_{f
	\in \mathcal{H}_{(-M,M]}} 
\Vert f- g_i \Vert ^2 +\psi(f),
\end{align*}
\begin{align*}
&\mathfrak{I}:\textbf{C}\longrightarrow \mathbb{R}\\
&\mathfrak{I}(\psi)=\sum_{i=1}^2 \Vert \tilde{f}_{i,\psi} - f_i \Vert^2\\
&;\tilde{f}_{i,\psi} =\text{arg-min}_{f
	\in \mathcal{H}} 
\Vert f- g_i \Vert ^2 +\psi(f).
\end{align*}
Then there exist sequences $\lbrace\psi_n^M\rbrace_{n \in {\mathbb{N}}}$ in $\textbf{C}^M$ such that the functions $\lbrace\psi_\gamma\rbrace_{\gamma \in {\Gamma}}$ corresponding to each $\psi_n^M$ are step functions and
$$\mathfrak{I}^M(\psi_n^M)\longrightarrow \inf_{\psi \in 
	\textbf{C}}\mathfrak{I}(\psi)\quad \text{when} \quad n,M\rightarrow \infty.$$
\end{theorem}
\begin{proof}
Suppose $S_n=\{1,\cdots ,(M_2-M_1)\times 2^n\}$ and  $\lbrace 
\epsilon_n \rbrace_{n=1}^{\infty}$ is a sequence such that $\lim _{n\rightarrow 
	\infty}\epsilon_n =0 $. Define the set $D_n^M$by
\begin{align*}
D_n^M= \Big\{ & (x_1,x_2, \cdots , x_{(2M) \times 2^n}) \in 
\mathbb{R}^{(2M) \times \mathbb{R}^n}\big |\\
& \forall (r,s,t) \in \lbrace (r,s,t)\big |~r,s,t \in S_n,~ r<s<t \rbrace ;~x_s\leq \max \{x_r,x_t \}\Big\}.
\end{align*}
Also suppose $t_{1,\gamma}=[(( g_1 )_{\gamma} +M)\times 
2^n]+1$, $s_{1,\gamma}=[(( f_1 )_{\gamma} +M)\times 2^n]+1$, $t_{2,\gamma}=[(( 
g_2 )_{\gamma} +M)\times 2^n]+1$, $s_{2,\gamma}=[(( f_2 )_{\gamma} +M)\times 
2^n]+1$ and define the function  $K_n^{M,\gamma}$ from $D_n^M$ to $\mathbb{R}$ 
for every $\gamma \in \Gamma$ by
\begin{align*}
K_n^{M,\gamma} \left(x_1,x_2, \cdots , x_{(2M) \times 2^n} \right) &=\Big\vert 
\min \lbrace (-M + \dfrac{t}{2^n}-( g_1 )_{\gamma})^2+x_t -(( f_1 )_{\gamma}-( 
g_1 )_{\gamma})^2 -x_{s_{1,\gamma}} \big | \\
&~1\leq t \leq t_{1,\gamma}-1 \rbrace \cup \lbrace x_{t_{1,\gamma}}-(( f_1 
)_{\gamma}-( g_1 )_{\gamma})^2 -x_{s_{1,\gamma}} \rbrace \\
&\cup \lbrace (-M + \dfrac{t-1}{2^n}+\epsilon_n -( g_1 )_{\gamma})^2 +x_t -(( 
f_1 )_{\gamma}-( g_1 )_{\gamma})^2 -x_{s_{1,\gamma}}\big |\\
&t_{1,\gamma}+1\leq t \leq (2M) \times 2^n \rbrace \Big\vert^2\\
&+ \Big\vert \min \lbrace (-M + \dfrac{t}{2^n}-( g_2 )_{\gamma})^2+x_t -(( f_2 
)_{\gamma}-( g_2 )_{\gamma})^2 -x_{s_{2,\gamma}} \big |\\
&~1\leq t \leq t_{2,\gamma}-1 \rbrace  \cup \lbrace x_{t_{2,\gamma}}-(( f_2 
)_{\gamma}-( g_2 )_{\gamma})^2 -x_{s_{2,\gamma}} \rbrace \\
& \cup \{ (-M + \dfrac{t-1}{2^n}+\epsilon_n -( g_2 )_{\gamma})^2 +x_t -(( f_2 
)_{\gamma}-( g_2 )_{\gamma})^2 -x_{s_{2,\gamma}}\big |\\
&t_{2,\gamma}+1\leq t \leq (2M) \times 2^n \} \Big\vert ^2.
\end{align*}
Let $\left(\mathcal{C}_1^{M,n,\gamma},\mathcal{C}_2^{M,n,\gamma}, \cdots , 
\mathcal{C}_{(2M) \times 2^n}^{M,n,\gamma} \right)$ be a minimizer of the 
function $K_n^{M,\gamma}$ for every $n\in \mathbb{N}$. Define the function 
$\psi_n^M$ for every $n\in \mathbb{N}$, from $\mathcal{H}_{(-M,M]}$ to 
$\mathbb{R}$ as follows:
\begin{align*}
&\psi_n^M:\mathcal{H}\rightarrow \mathbb{R}\\
&\psi_n^M(f)=\sum_{\gamma \in \Gamma} \psi_{n,\gamma}^M(f_{\gamma}),
\end{align*}
where  $\psi_{n,\gamma}^M$ for every $n\in \mathbb{N}$ and $\gamma \in \Gamma$ 
is a step function from $(-M,M]$ to  $\mathbb{R}$ which is defined by 
\begin{equation*}
\psi_{n,\gamma}^M = \sum _{t=1}^{(2M) \times 2^n} \mathcal{C}_t^{M,n,\gamma} 
\chi_{{A}_t^M}
\end{equation*}
such that 
${A}_t ^M= \left( -M + \dfrac{t-1}{2^n} , -M + \dfrac{t}{2^n} 
\right];~1\leq t \leq (2M) \times 2^n.$
Now, suppose $M_0 \in \mathbb{N}$ is large enough such that $( f_1)_{\gamma},( f_2 
)_{\gamma},( g_1 )_{\gamma},( g_2 )_{\gamma} \in (-M_0,M_0]$. By using Theorem 
2.5
for every $M\geq M_0$ where  $M\in \mathbb{N}$, we will have:
$$\mathfrak{I}^M(\psi_n^M)\longrightarrow 
\text{inf}_{\psi\in {\textbf{C}}^M}\mathfrak{I}^M(\psi)\quad
 \text{when} \quad 
n\rightarrow\infty.$$
Since  $\mathbb{R}=\cup_{M=M_0}^{+\infty}(-M,M]$, we conclude
$$\mathfrak{I}^M(\psi_n^M)\longrightarrow 
\text{inf}_{\psi\in\textbf{C}}\mathfrak{I}(\psi)\quad \text{when} \quad 
n,M\rightarrow \infty.$$
\end{proof}
\begin{remark}\label{2.8}
As mentioned in Introduction, the constraint 
$\psi$ is desirable on 
the condition that the following minimization problem is soluble:
$$\inf_{f\in H}\|f-g\|_{\mathcal{H}}^2+\psi(f).$$
According to the definition of the set $\textbf{C}$,
this minimization problem uncouples 
into a family of one-dimensional minimizations:
\begin{align*}
&M_\gamma:\mathbb{R}\longrightarrow \mathbb{R}\\
&M_\gamma(x)=|x-g_\gamma|^2+\psi_\gamma(x);~~ \gamma\in \Gamma.
\end{align*}
Since $\psi_\gamma$
is regulated and quasiconvex for every $\gamma\in\Gamma$, $M_\gamma$
is also regulated and 
quasiconvex for every $\gamma\in \Gamma$ 
and the fact that the regulated and quasiconvex 
functions do not have a minimizer in general may seem to be a flaw. But, it can 
not cause a problem, due to the process of obtaining $\psi_\gamma$,
we work practically 
with an approximation of $\psi_\gamma$
($\psi_{n,\gamma}$ for a large enough $n$) that are quasiconvex step 
functions which they always have a minimizer. 
However, for the sake of theoretical solidity, a discussion similar to the 
definition of integral for measurable functions (using a sequence of step 
functions) can be done for the definition of a minimizer for the regulated and 
quasiconvex functions.
\end{remark}
\begin{remark}\label{2.9}
In some applications, there may be other degradation (such as 
blurring) in addition to the noise and the minimization problem may change as 
following:
\begin{equation*}
\inf_{f\in \mathcal{H}}\|Kf-g\|_{\mathcal{H}}^2+\psi(f)
\tag*{$(2\cdot 8)$}
\end{equation*}
where $K: {\mathcal{H}}\longrightarrow {\mathcal{H}}'$ is
a linear bounded operator.  Since in this minimization problem the 
minimizer should be obtained by an iterative procedure, solving the 
corresponding bi-level optimization problem requires a different method from 
what was presented in this paper. In fact, the minimizer of the following 
function should be achieved:
\begin{align*}
& \Phi:\textbf{C}\longrightarrow \mathbb{R}\\
& \Phi(\psi)=\sum_{i=1}^n\|\tilde{f}_{i,\psi}-f_i\|_{\mathcal{H}}^2\\
&
;\tilde{f}_{i,\psi}=\text{arg-min}_{f\in \mathcal{H}}\|Kf-
g_i\|_{\mathcal{H}'}^2+\psi(f),
\tag*{$(2\cdot 9)$}
\end{align*} 
which is different from the problem presented in Introduction. However, in 
the special case where $K$ happens to be diagonal in the $\varphi_\gamma$-basis,
$K(\varphi_\gamma)=k_\gamma\varphi_\gamma$,   the 
minimization problem (2.8), thus, uncouples into a family of one-dimensional 
minimizations and is easily solved similar to what was done in Lemma 1.1:
$$\inf_{f\in\mathcal{H}}\sum_{\gamma\in \Gamma}[|k_\gamma f_\gamma-g_\gamma| 
^2+\psi_\gamma(f_\gamma)].$$
Therefore, for this case, since this minimization problem is completely similar to the minimization 
problem (1.6), the learning problem corresponding to the problem (2.8) 
will be similar to the learning method presented in this paper with a slight 
difference.

Solving the bilevel optimization problem (2.9) can have different types of applications from the application that already introduced. The type of application discussed in Introduction refers to the ones which are related to the modeling before 
the denoising process and to explain that various models can be obtained 
according to different conditions. However, in addition to the view in which 
the model has an ability to adapt, there is another perspective by which new 
constraints can be created for Inverse Problems. For instance, as we know the following famous minimization problem has 
different applications in applied sciences (see e.g. [17]):
$$\inf_{f\in \mathcal{H}}\|K(f)-g\|^2+\lambda|||f|||_{W,p}^p$$
where $\mathcal{H}$ is a Hilbert space, $K:{\mathcal{H}}\longrightarrow {\mathcal{H'}}$
is a linear and bounded function, $\lambda$ is some positive constant called 
the 
regularization parameter and
\begin{equation*}
|||f|||_{W,p}=\Biggl(\sum_{\gamma\in\Gamma} w_\gamma|\langle 
f,\varphi_\gamma\rangle|^p\Biggr)^{\frac{1}{p}}
\end{equation*}
for $1\leq p\leq 2$, 
is a weighted $\ell_p$-norm of the coefficients of $f$  with respect to an 
orthonormal basis $(\varphi_\gamma)_{\gamma\in\Gamma}$
of $\mathcal{H}$,
 and a sequence of strictly positive weights $W=(w_\gamma)_{\gamma\in 
\Gamma}$. Now,
assuming  $\psi_\circ(f)=\sum_{\gamma\in\Gamma}\psi_{\circ,\gamma}(f_\gamma)$ where
$\psi_{\circ,\gamma}(x)=\lambda w_\gamma |x|^p$  for every $\gamma\in\Gamma$, we have
$$\psi_\circ(f)=\sum_{\gamma\in\Gamma}\lambda w_\gamma| \langle 
f,\varphi_\gamma\rangle|^p=\lambda |||f|||^{p}_{W,p}.$$
Since $\psi_{\circ,\gamma}$
for every $\gamma\in\Gamma$ are Regulated and quasiconvex functions, the 
function  
$\psi_\circ$ 
is an element of the set $\textbf{C}$  defined in Introduction. Therefore, in an application that a model with $|||f|||^{p}_{W,p}$ penalty term work well and a large enough training set is available in that particular application, it is hoped that a regularization term $\psi$ in relation to that training set is found that functions better than $|||f|||_{W,p}$-norm, because here $\psi$
is taught based on a training set 
for a specific application type. In other words, by the learning process, a fixed 
constraint more suitable than $|||f|||_{W,p}$-norm can be produced. 
\end{remark}
\begin{remark}\label{2.10}
The problem (2.9) in the previous remark can be considered as a special case of the following problem:
\begin{align*}
& \Phi:\textbf{C}\times\textbf{U}\longrightarrow \mathbb{R}\\
& \Phi(\psi,K)=\sum_{i=1}^n\|\tilde{f}_{i,\psi,K}-f_i\|_{\mathcal{H}}^2\\
&
;\tilde{f}_{i,\psi,K}=\text{arg-min}_{f\in \mathcal{H}}\|Kf-
g_i\|_{\mathcal{H}'}^2+\psi(f)
\tag*{$(2\cdot 10)$}
\end{align*}  
where $\textbf{U}$ is a category of operators from the Hilbert space $\mathcal{H}$ to the Hilbert space $\mathcal{H^{\prime}}$. In fact, the problem (2.10) by assuming a single-member category becomes the problem (2.9).  The learning problems for applications in which both noise degradation and blurring degradations affect the inputs (objects) should be formulated as the bilevel optimization problem (2.10). To illustrate, consider the application in the satellite image processing which is presented in Section 7. Suppose the target of this application is limited to imaging the Earth's surface and the assumptions about the inputs and the outputs are changed as follows: 
objects (inputs) are images of some parts of the Earth's surface that we know their exact details, Labels (outputs) are images that have been taken of these objects by satellite's camera and then have been sent to the earth. By these assumptions, the outputs (labels) $g$'s not only have the noise degradation but also the blurring degradation which is due to the relative motion between the satellite and the earth. These assumptions are different from what is in Section 7 where outputs just have noise degradation. In this situation, how to obtain the right model with regard to the training set $ \lbrace (f_i, g_i)\vert ~i=1,\cdots, m \rbrace $ will be different under two different conditions: 

{\it Case 1)} if the blurring degradation pattern (the operator $K$) can be exactly formulated by some available features - namely the distance between the satellite to earth, the satellite speed, the shooting angle and so on - for attaining a right model in this case, the optimization problem (2.9) should be utilized. In other words, to obtain an image without degradation from a $g$ that is an image from another part of the earth's surface which has been taken by the satellite and then has been sent to the earth, the following minimization problem should be used: 
$$\text{inf}_{f 
 \in \mathcal{H}} 
\Vert K(f)- g \Vert ^2 +\psi^*(f).$$  
where $\psi^*$ is the minimizer of the learning problem (2.9) and $K$ is the same operator used in the bilevel optimization problem (2.9) which indicates the blurring degradation pattern. 

{\it Case 2)} if the blurring degradation pattern cannot be exactly formulated and we just know it is an element of a category of the operators, for attaining a right model in this case, the optimization problem (2.10) should be utilized. In other words, to obtain an image without degradation from $g$ that is an image from another part of the earth's surface which has been taken by the satellite and then has been sent to the earth, the following minimization problem should be used:
$$\text{inf}_{f 
 \in \mathcal{H}} 
\Vert K^*(f)- g \Vert ^2 +\psi^*(f).$$
where $(\psi^*, K^*)$ is the minimizer of the learning problem (2.10). 

Taken as a whole, in all applications in which, in addition to the noise degradation, the blurring degradation exists and these two degradations affect the input signals, for reaching a right model for signal denoising - given that whether the operator corresponding to the blurring degradation could be precisely specified or not - the bilevel optimization problems (2.9) or (2.10) should be considered. For a further discussion of the problems (2.9) and (2.10), see Section 6.
\end{remark}
\begin{remark}\label{2.11}
The approach for solving the problem (1.7) by considering $ \mathbf{Z}= \mathbf{C} $ was based on converting the unknown function $\psi$ to some unknown coefficients by utilizing step functions. 
The point arising here is that different approaches with different challenges would be possible by utilizing other types of functions different from step functions. 
For instance, ``sums of exponential functions" can be used by putting the set $ \mathbf{Z}= \mathbf{C}^{\prime}$ which is defined below in the problem (1.7):
\begin{align*}
\textbf{C}^{\prime}=\{\psi:\mathcal{H}\longrightarrow \mathbb{R}|~\psi(f)=\sum_{\gamma\in\Gamma}
\psi_\gamma(f_\gamma);~\forall 
\gamma\in\Gamma:~\psi_\gamma:\mathbb{R}\longrightarrow \mathbb{R},~\psi_\gamma \in {B} \}
\end{align*}
where $B=\overline {\cup_{n=1}^{+\infty} {B_n}}$, $B_n=\{\psi \vert ~\psi(x)=\sum_{i=1}^{n} {\alpha_i e^{\beta_i x}};~\alpha_i,\beta_i \in \mathbb{R},~\psi^{\prime\prime} >0\}$. 
The condition $ \psi^{\prime\prime} >0$ is added to ensure the existence of a minimizer for the problem (1.6) because in this case, $\psi$ would be convex. For solving (1.7) by considering $ \mathbf{Z}= \mathbf{C}^{\prime} $, the following bilevel optimization problems should be solved for every $n$: 
\begin{align*}
&\mathfrak{I}:{\textbf{C}_{n}^{\prime}}\longrightarrow \mathbb{R}\\
& \mathfrak{I}(\psi)=
\sum_{i=1}^m\|\tilde{f}_{i,\psi}-f_i\|^2\\
&
;\tilde{f}_{i,\psi}=\text{arg-min}_{f\in 
	\mathcal{H}}\|f-g_i\|^2+
\psi(f).
\tag*{$(2\cdot 11)$}
\end{align*}
where the $\mathbf{C}_{n}^{\prime}$ is as follows:
\begin{align*}
&\textbf{C}_{n}^{\prime}=\{\psi:\mathcal{H}\longrightarrow \mathbb{R}|~\psi(f)=\sum_{\gamma\in\Gamma}
\psi_\gamma(f_\gamma);~\forall 
\gamma\in\Gamma:~\psi_\gamma:\mathbb{R}\longrightarrow \mathbb{R},~\psi_\gamma \in {B_n} \}
\end{align*}
because, supposing $ \psi_{n}$ are the minimizers of (2.11) for every $n$, since $B=\overline {\cup_{n=1}^{+\infty}B_n}$, we will have:
$$\mathfrak{I}(\psi_n)\longrightarrow \inf_{\psi\in \textbf{C}^{\prime}} \mathfrak{I} (\psi)~~\text{when} 
~~n\longrightarrow 
\infty.$$
However, for solving the problems (2.11), we need a different approach from what presented here. In fact, the minimizer of the following differentiable function  from $ D\subseteq \mathbb{R}^{2n}$ to $\mathbb{R}$ should be found: 
\begin{align*}
&\mathfrak{I}:D\longrightarrow \mathbb{R}\\
& 
\mathfrak{I}(\alpha_1,\cdots, \alpha_n, \beta_1,\cdots, \beta_n)=\sum_{i=1}^m\sum_{\gamma\in\Gamma}|G_{(\alpha_1,\cdots,\alpha_n, \beta_1, \cdots, \beta_n)}^{-1}((g_i)_\gamma)-(f_i)_\gamma|^2
\end{align*}
where 
\begin{align*}
&D=\{(\alpha_1,\cdots,\alpha_n, \beta_1, \cdots, \beta_n) \in \mathbb{R}^{2n} \vert ~(\sum_{i=1}^{n} {\alpha_i e^{\beta_i x}})^{\prime \prime} > 0~for~x \in \mathbb{R}\},\\
&G_{(\alpha_1,\cdots,\alpha_n, \beta_1,\cdots, \beta_n)}(x)=x+\frac{\sum_{i=1}^n \alpha_i \beta_i e^{\beta_i x}}{2}.
\end{align*}
\end{remark}

\section{An additional  analytical method  for opting  the most appropriate basis in the learning process}
In this section, we will try to select a basis among a category of bases alongside with the regularization term regarding a training set. Firstly; it should be mentioned it is expected that using the matrix basis for the model proposed in this 
article will not bring the desired results. The reason is that by using the 
matrix basis, the minimization process will be performed on each entry of 
matrix separately. Thus, by applying it, the total variation denoising model 
[43] will act better than our learning method, because it considers 
the variations between entries in the minimization process.  

However, those bases from which desired outcomes are expected in 
computational works are multiscale representation systems such as Wavelets [16], 
Curvelets [11], Shearlets [28], and so on within the area of Applied Harmonic 
Analysis or the bases in Fourier Analysis. Now, it is a good notion to examine by ``trial and error" which one of these bases 
has more favorable results assuming to have 
a training set. Moreover, analytical efforts can also be made to get the best 
basis among a category of bases that will be addressed in this section:
For a theoretical setting, suppose $\mathcal{H}$ is a Hilbert space, $E\subseteq \mathbb{R}^N$ and the following 
set is a 
category of bases:
$$\amalg =\Big\lbrace \{ \varphi_{\gamma}^{(y_1,\cdots , y_N)}\}_{\gamma \in 
	\Gamma} \big |~(y_1,\cdots , y_N) \in E \Big\rbrace $$
where $N\in\mathbb{N}$ and $\{\varphi_{\gamma}^{(y_1,\cdots , 
y_N)}\}_{\gamma\in\Gamma}$ is a basis for the Hilbert space $\mathcal{H}$  for every 
$(y_1,\cdots , y_N)\in 
E$. Again, as it was introduced earlier, we consider a training set of pairs 
$(g_i,f_i)$; $i=1,\cdots , m.$ Also suppose $f_{i,\psi,(y_1,\cdots ,y_N)}$ is 
a minimizer of the following  minimization problem 
$$\begin{array}{c}
\text{inf}_{f\in{\mathcal{H}}_{(y_1,\cdots , y_N)}}\|f-g_i\|^2_{{\mathcal{H}}_{(y_1,\cdots ,y_N)}}+\psi(f)
\end{array}$$
where  $i\in \{1,2,\cdots , m\}$ and ${\mathcal{H}}_{(y_1,\cdots , y_N)}$ is 
the 
Hilbert space $\mathcal{H}$ by considering the basis 
$\{\varphi_{\gamma}^{(y_1,\cdots, y_N)}\}_{\gamma\in \Gamma}$ for it. Now, we peresent the following function:
\begin{align*}
&\mathfrak{I}:\textbf{C}\times E\longrightarrow \mathbb{R}\\
&\mathfrak{I}(\psi, (y_1,\cdots ,y_N)) 
=\sum_{i=1}^m\|f_{i,\psi,(y_1,\cdots,y_N)}-f_i\|^2_{{\mathcal{H}}_{(y_1,\cdots 
		,y_N)}}\\
&
;f_{i,\psi,(y_1,\cdots, y_N)}=\text{arg-min}_{f\in{\mathcal{H}}_{(y_1,\cdots ,y_N)}}
\|f-g_i\|^2_{{\mathcal{H}}_{(y_1,\cdots,y_N)}}+\psi(f).
\end{align*}

In Theorem 3.2 whose  proof is similar to Theorem 2.5, a sequence 
$\{\psi_n,(y_1^n,\cdots,y_N^n)\}$ will be presented from $\textbf{C}\times E$ such 
that 
$$\mathfrak{I}(\psi_n,(y_1^n,\cdots,y_N^n))\longrightarrow 
\text{inf}_{\psi\in \textbf{C},(y_1,\cdots,y_N)\in E}\mathfrak{I}(\psi,(y_1,\cdots,y_N)) 
~~\text{when} ~~n\longrightarrow\infty.$$

The following proposition is needed for the proof of Theorem 3.2. In this proposition, the 
shorthand notation $f_{\gamma, (y_1,\cdots,y_N)}$
  is used for $\langle 
f,\psi_{\gamma}^{(y_1,\cdots,y_N)}\rangle_{\mathcal{H}_{(y_1,\cdots,y_N)}}$
 where  $\{\psi_\gamma^{(y_1,\cdots,y_N)}\}_{\gamma\in\Gamma}\in
 \amalg$:

\begin{proposition}\label{3.1}
 Suppose $\mathcal{H}$ is a Hilbert space and  $\amalg$ is a 
category of bases 
as follows:
$$\amalg =\Big\lbrace \{ \varphi_{\gamma}^{(y_1,\cdots , y_N)}\}_{\gamma \in 
	\Gamma} \big |~(y_1,\cdots , y_N) \in E \Big\rbrace $$
where $E\subseteq \mathbb{R}^N$, $N \in \mathbb{N}$ and $\{ 
\varphi_{\gamma}^{(y_1,\cdots , y_N)}\}_{\gamma \in \Gamma}$ is an orthonormal 
basis for the Hilbert space $\mathcal{H}$ for every $(y_1,\cdots , 
y_N) \in E$. 
Now, Suppose $f_1,f_2,g_1,g_2 \in \mathcal{H}$ and $M_1, M_2 \in \mathbb{Z}$ are 
large enough such that $( f_1 )_{\gamma,(y_1,\cdots , y_N)},( f_2 )_{\gamma,(y_1,\cdots , y_N)},( g_1 )_{\gamma,(y_1,\cdots , y_N)},( 
g_2 )_{\gamma,(y_1,\cdots , y_N)} \in (M_1,M_2]$ for every $\gamma \in \Gamma$ and every $(y_1,\cdots , y_N) \in E$. 
Moreover, define the set  $\mathcal{A}$ and the functions
$F_{\psi}^{\gamma , 
	(y_1,\cdots , y_N)},~G_{\psi}^{\gamma , (y_1,\cdots , y_N)}$   from  $(M_1,M_2]$  
to $\mathbb{R}$ by:
\begin{align*}
&\mathcal{A}=\lbrace \psi|~\psi:(M_1,M_2]
\rightarrow \mathbb{R}; ~\psi~\text{is 
	a regulated and quasiconvex function}\rbrace ,\\
\\
&F_{\psi}^{\gamma , (y_1,\cdots , y_N)}(x)=\left( (g_1)_{\gamma , (y_1,\cdots 
	, y_N)}-x \right)^2 + \psi(x),\\
&G_{\psi}^{\gamma , (y_1,\cdots , y_N)}(x)=\left( (g_2)_{\gamma , (y_1,\cdots 
	, y_N)}-x \right)^2 + \psi(x)
\end{align*}
where  $\psi \in \mathcal{A}$. Furthermore, we define the set $\mathcal{A}^m$ and the function $h$ as follows:
\begin{align*}
&\mathcal{A}^m=\underbrace{\mathcal{A} 
		\times \cdots \times \mathcal{A}}_{ \vert \Gamma \vert =m},\\
\\
&h:\mathcal{A}^m \times 
E\longrightarrow \mathbb{R}\\
&h\left((\psi_{\gamma_1},\cdots ,\psi_{\gamma_m} ),(y_1,\cdots , y_N) \right) \\
&=\sum_{\gamma \in \Gamma} \Big\vert F_{\psi_{\gamma}}^{\gamma , (y_1,\cdots , 
	y_N)} ((f_1)_{\gamma , (y_1,\cdots , y_N)}) 
-\min_{x \in (M_1,M_2]} F_{\psi_{\gamma}}^{\gamma , (y_1,\cdots , y_N)}(x) 
\Big\vert^2\\
&~~~~~~~~+ \Big\vert G_{\psi_{\gamma}}^{\gamma , (y_1,\cdots , y_N)} ((f_2)_{\gamma , 
	(y_1,\cdots , y_N)}) 
-\min_{x \in (M_1,M_2]} G_{\psi_{\gamma}}^{\gamma , (y_1,\cdots , y_N)}(x) 
\Big\vert^2.
\end{align*}
Then there exists a sequence of step functions $\lbrace (\psi_{n,\gamma_1},\cdots ,\psi_{n,\gamma_m} ) \rbrace_{n \in \mathbb{N}}$ in $\mathcal{A}^m$ and $\lbrace (y_1^n,\cdots , y_N^n) \rbrace_{n \in \mathbb{N}}$ in $E$ such that 
\begin{align*}
&h\left((\psi_{n,\gamma_1},\cdots ,\psi_{n,\gamma_m} ),(y_1^n,\cdots , y_N^n) 
\right)\rightarrow\\
& \inf_{(\psi_{\gamma_1},\cdots ,\psi_{\gamma_m} ) \in \mathcal{A}^m , (y_1,\cdots , 
y_N) 
	\in E}~h\left( (\psi_{\gamma_1},\cdots ,\psi_{\gamma_m}),(y_1,\cdots , y_N) 
\right)\\
&\qquad  \qquad \quad \text{when} \quad n\rightarrow \infty.
\end{align*}
\end{proposition}
\begin{proof}
Suppose $S_n=\{1,\cdots ,(M_2-M_1)\times 2^n\}$ and
$\lbrace \epsilon_n \rbrace_{n=1}^{\infty}$ is a sequence such that $\lim 
_{n\rightarrow \infty}\epsilon_n =0 $. Also we have:
\begin{align*}
&\langle f_i,\varphi_{\gamma}^{(y_1^n,\cdots , y_N^n)}\rangle\longrightarrow 
\langle f_i,\varphi_{\gamma}^{(y_1^0,\cdots , y_N^0)}\rangle \quad \text{when} \quad 
(y_1^n,\cdots , y_N^n)\rightarrow ( y_1^0,\cdots , y_N^0) 
\tag*{$(3\cdot 1)$}\\
&\langle g_i,\varphi_{\gamma}^{(y_1^n,\cdots , y_N^n)}\rangle\longrightarrow 
\langle g_i,\varphi_{\gamma}^{(y_1^0,\cdots , y_N^0)}\rangle \quad \text{when} \quad 
(y_1^n,\cdots , y_N^n)\rightarrow ( y_1^0,\cdots , y_N^0)
\tag*{$(3\cdot 2)$}
\end{align*}
for $i=1,2$ and every $\gamma \in \Gamma$.\\
Define the set  $D_n$ as defined in Theorem 2.5. Moreover, suppose $t_{1,\gamma , (y_1,\cdots , 
	y_N)}=[((g_1)_{\gamma
	, (y_1,\cdots , y_N)} -M_1) \times 2^n]+1$, $s_{1,\gamma , 
	(y_1,\cdots , y_N)}=[((f_1)_{\gamma , (y_1,\cdots , y_N)} -M_1) \times 
2^n]+1$, $t_{2,\gamma
 , (y_1,\cdots , y_N)}=[((g_2)_{\gamma , (y_1,\cdots , y_N)} 
-M_1) \times 2^n]+1$, $s_{2,\gamma , (y_1,\cdots , y_N)}=[((f_2)_{\gamma , 
	(y_1,\cdots , y_N)} -M_1) \times 2^n]+1$.\\
In addition, suppose  $E_n\subseteq E,~E=\cup_{n=1}^{+\infty}E_n$ and define the function 
$K_n^{\gamma}$ from $D_n \times E_n$ to $\mathbb{R}$ for $\gamma \in \Gamma$ by
\begin{align*}
&K_n^{\gamma} \left(x_1^{\gamma},x_2^{\gamma}, \cdots , x_{(M_2-M_1) \times 
	2^n}^{\gamma} , y_1,\cdots , y_N \right) =\\
&\Big\vert \min \lbrace (M_1 + \dfrac{t}{2^n}-( g_1 )_{\gamma ,(y_1,\cdots , 
	y_N)})^2+x_t^{\gamma} -(( f_1 )_{\gamma , (y_1,\cdots , y_N)}-( g_1 )_{\gamma 
	,(y_1,\cdots , y_N)})^2 -x_{s_{1,\gamma ,(y_1,\cdots , y_N)}} \big | \\
&~1\leq t \leq t_{1,\gamma ,(y_1,\cdots , y_N)}-1 \rbrace \cup \lbrace 
x_{t_{1,\gamma ,(y_1,\cdots , y_N)}}-(( f_1 )_{\gamma ,(y_1,\cdots , y_N)}-( g_1 
)_{\gamma ,(y_1,\cdots , y_N)})^2 -x_{s_{1,\gamma ,(y_1,\cdots , y_N)}} \rbrace \\
&\cup \lbrace (M_1 + \dfrac{t-1}{2^n}+\epsilon_n -( g_1 )_{\gamma ,(y_1,\cdots 
	, y_N)})^2 +x_t^{\gamma} -(( f_1 )_{\gamma ,(y_1,\cdots , y_N)}-( g_1 
)_{\gamma 
	,(y_1,\cdots , y_N)})^2 -x_{s_{1,\gamma ,(y_1,\cdots , y_N)}}\big |\\
&t_{1,\gamma ,(y_1,\cdots , y_N)}+1\leq t \leq (M_2-M_1) \times 2^n \rbrace 
\Big\vert^2\\
&+ \Big\vert \min \lbrace (M_1 + \dfrac{t}{2^n}-( g_2 )_{\gamma ,(y_1,\cdots , 
	y_N)})^2+x_t^{\gamma} -(( f_2 )_{\gamma ,(y_1,\cdots , y_N)}-( g_2 )_{\gamma 
	,(y_1,\cdots , y_N)})^2 -x_{s_{2,\gamma ,(y_1,\cdots , y_N)}} \big |\\
&~1\leq t \leq t_{2,\gamma ,(y_1,\cdots , y_N)}-1 \rbrace  \cup \lbrace 
x_{t_{2,\gamma ,(y_1,\cdots , y_N)}}-(( f_2 )_{\gamma ,(y_1,\cdots , y_N)}-( g_2 
)_{\gamma ,(y_1,\cdots , y_N)})^2 -x_{s_{2,\gamma ,(y_1,\cdots , y_N)}} \rbrace \\
& \cup \{ (M_1 + \dfrac{t-1}{2^n}+\epsilon_n -( g_2 )_{\gamma ,(y_1,\cdots , 
	y_N)})^2 +x_t ^{\gamma}-(( f_2 )_{\gamma ,(y_1,\cdots , y_N)}-( g_2 )_{\gamma 
	,(y_1,\cdots , y_N)})^2 -x_{s_{2,\gamma ,(y_1,\cdots , y_N)}}\big |\\
&t_{2,\gamma ,(y_1,\cdots , y_N)}+1\leq t \leq (M_2-M_1) \times 2^n \} 
\Big\vert ^2.
\end{align*}
Suppose  $\Gamma$ has finite elements $(\Gamma =\lbrace \gamma_1, \cdots , 
\gamma_m \rbrace)$ and define the function  $K_n$ on  $\underbrace{D_n \times 
	\cdots \times D_n}_{\vert \Gamma \vert =m} \times E_n$ as follows:
\begin{align*}
&K_n\left(x_1^{\gamma_1},x_2^{\gamma_1}, \cdots , x_{(M_2-M_1) \times 
	2^n}^{\gamma_1},\cdots , x_1^{\gamma_m},x_2^{\gamma_m}, \cdots , x_{(M_2-M_1) 
	\times 2^n}^{\gamma_m}, y_1,\cdots , y_N \right) =\\
&\sum_{\gamma \in \Gamma}K_n^{\gamma} \left(x_1^{\gamma},x_2^{\gamma}, \cdots , 
x_{(M_2-M_1) \times 2^n}^{\gamma} , y_1,\cdots , y_N \right).
\end{align*}
Let 
\begin{align*}
& (\mathcal{C}_1^{n,\gamma_1},\mathcal{C}_2^{n,\gamma_1}, \cdots , 
\mathcal{C}_{(M_2-M_1) \times 2^n}^{n,\gamma_1} 
,\mathcal{C}_1^{n,\gamma_2},\mathcal{C}_2^{n,\gamma_2}, \cdots , \\
& \mathcal{C}_{(M_2-M_1) \times 2^n}^{n,\gamma_2} , \cdots , 
\mathcal{C}_1^{n,\gamma_m},\mathcal{C}_2^{n,\gamma_m}, \cdots , 
\mathcal{C}_{(M_2-M_1) \times 2^n}^{n,\gamma_m}, y_1^n,\cdots , y_N^n )
\end{align*}
be a minimizer of the function $K_n$ for every $n\in \mathbb{N}$. Define the 
step function $\psi_{n,\gamma}$ for every $n\in \mathbb{N}$ and $\gamma \in 
\Gamma$ from  $(M_1,M_2]$ to  $\mathbb{R}$ by
\begin{equation*}
\psi_{n,\gamma} = \sum _{t=1}^{(M_2-M_1) \times 2^n} \mathcal{C}_t^{n,\gamma} 
\chi_{{A}_t}
\end{equation*}
such that 
${A}_t= \left( M_1 + \dfrac{t-1}{2^n} , M_1 + \dfrac{t}{2^n} 
\right];~1\leq t \leq (M_2-M_1) \times 2^n.$
Now, we define the sets  $\mathcal{A}_n$ and $(\mathcal{A}_n)^m$   for  $n \in \mathbb{N}$:
\begin{align*}
&\mathcal{A}_n= \{\varphi \big |~ \varphi = \sum _{t=1}^{(M_2-M_1) \times 2^n} 
x_t \chi_{{A}_t};~ \forall (r,s,t) \in \lbrace (r,s,t)|~r,s,t\in S_n,~r<s<t \rbrace ;\\
&~~~~~~~~~~~~~~x_s\leq \max \{x_r,x_t \} \},\\
&(\mathcal{A}_n)^m=\underbrace{\mathcal{A}_n \times \cdots \times \mathcal{A}_n}_{\vert \Gamma 
		\vert =m}.
\end{align*}
Given {\it Fact 1}, it will be enough to prove the following items:
\begin{itemize}
	\item[\it (i)]
	By considering  {\it Supremum metric} for $\mathcal{A}$ and {\it Euclidean metric} 
	for  $E$, $\mathcal{A}^m \times E $ with the product metric is a metric space.
	\item[\it (ii)]
	$h:\mathcal{A}^m \times E\rightarrow  \mathbb{R}$ is a continuous function and 
bounded 
	from below.
	\item[\it (iii)]
	$(\mathcal{A}_n)^m \times E_ n \subseteq (\mathcal{A}_{n+1})^m \times E_ {n+1} $ for 
every  
	$n \in \mathbb{N}$.
	\item[\it (iv)]
	$\overline{\cup_{n=1}^{\infty} (\mathcal{A}_n)^m\times E_ n 
}=\mathcal{A}^m \times E.$
	\item[\it (v)]
	$\Big\vert h\left((\psi_{n,\gamma_1},\cdots ,\psi_{n,\gamma_m} 
),(y_1^n,\cdots 
	, y_N^n) \right)\\
	- \inf_{(\psi_{\gamma_1},\cdots ,\psi_{\gamma_m} ) \in 
		(\mathcal{A}_n)^m , (y_1,\cdots , y_N) 
		\in E_n}~h\left( (\psi_{\gamma_1},\cdots ,\psi_{\gamma_m}),(y_1,\cdots , y_N) 
	\right) \Big\vert \\
	\text{tends to} ~0 ~ \text{as}~ n\rightarrow \infty.$
\end{itemize}
$(i)$ and  $(iii)$ are clearly correct and we have  $(iv)$ by Lemma  2.1. 
Moreover, by 
(3.1) and (3.2) in assumptions, the function $h$ will 
be continuous. It remains to prove only $(v)$. By a process
similar to what was 
done in Proposition 2.3, 
we  will have:
\begin{align*}
&h\Big|_{(\mathcal{A}_n)^m \times E_n}\left( (\psi_{\gamma_1},\cdots 
,\psi_{\gamma_m}),(y_1,\cdots , y_N) \right)\simeq \\
&K_n\left(x_1^{\gamma_1},x_2^{\gamma_1}, \cdots , x_{(M_2-M_1) \times 
	2^n}^{\gamma_1},\cdots , x_1^{\gamma_m},x_2^{\gamma_m}, \cdots , x_{(M_2-M_1) 
	\times 2^n}^{\gamma_m},y_1,\cdots , y_N \right),
\end{align*}
where  $\psi_{\gamma_i}=\sum_{t=1}^{(M_2-M_1) \times 2^n}x_t^{\gamma_i} 
\chi_{{A}_t}$ for every $i \in \lbrace 1,\cdots ,m \rbrace$. 
Consequently 
\begin{align*}
&h\left( (\psi_{n,\gamma_1},\cdots ,\psi_{n,\gamma_m}),(y_1^n,\cdots , y_N^n) 
\right)\simeq \\
&\inf_{(\psi_{\gamma_1},\cdots ,\psi_{\gamma_m} ) \in (\mathcal{A}_n)^m , (y_1,\cdots , y_N) 
	\in E_n}~h\left( (\psi_{\gamma_1},\cdots ,\psi_{\gamma_m}),(y_1,\cdots , y_N) 
\right).
\end{align*}
Therefore, since  $\epsilon_n$~tends to $0$~as $n$ tends to $\infty$, we 
conclude:
\begin{align*}
&\Big\vert 
h\left( (\psi_{n,\gamma_1},\cdots ,\psi_{n,\gamma_m}),(y_1^n,\cdots , y_N^n) 
\right) \\
&-\inf_{(\psi_{\gamma_1},\cdots ,\psi_{\gamma_m} ) \in 
	(\mathcal{A}_n)^m , (y_1,\cdots , y_N) 
	\in E_n}~h\left( (\psi_{\gamma_1},\cdots ,\psi_{\gamma_m}),(y_1,\cdots , y_N) 
\right)\Big\vert  \rightarrow 0,\\
& \qquad  \text{when} \qquad n\rightarrow \infty.
\end{align*}
\end{proof}


\begin{theorem}\label{3.2}
Suppose we have assumptions
in Proposition  3.1. 
Moreover, define the 
function $\psi_n$ for every $n \in \mathbb{N}$ from  $\mathcal{H}$ to 
$\mathbb{R}$ as follows:
\begin{align*}
&\psi_n: \mathcal{H}\longrightarrow \mathbb{R}\\
&\psi_n(f)=\sum_{\gamma \in \Gamma} \psi_{n,\gamma}f_{\gamma}.
\end{align*}
Also, define the set $\textbf{C}$ and the function $\mathfrak{I}$ from 
$\textbf{C} \times E$ to $\mathbb{R}$ by:
\\
\begin{align*}
\textbf{C}=\lbrace \psi:\mathcal{H} \rightarrow \mathbb{R}\big 
|~&\psi(f)=\sum_{\gamma \in \Gamma} \psi_{\gamma}(f_{\gamma});~ \forall \gamma 
\in \Gamma : ~\psi_{\gamma}:(M_1,M_2]\rightarrow \mathbb{R}\\
~ &\text{is a regulated and quasiconvex function}\rbrace ,\\
\\
&\mathfrak{I}:\textbf{C} \times E\longrightarrow \mathbb{R} \\
&\mathfrak{I} \left( \psi ,(y_1,\cdots , y_N) \right) = \sum_{i=1}^2 \Big\Vert 
\tilde{f}_{i,\psi , (y_1,\cdots , y_N)} -f_i 
\Big\Vert_{\mathcal{H}_{(y_1,\cdots , y_N)}}^2\\
& ;\tilde{f}_{i,\psi , (y_1,\cdots , y_N)} =\text{arg-min}_{f \in 
	\mathcal{H}_{(M_1,M_2]}} \Big\Vert f- g_i 
\Big\Vert_{\mathcal{H}_{(y_1,\cdots 
		, y_N)}}^2 + \psi(f).
\end{align*}
Then we have 
\begin{equation*}
\mathfrak{I}\left( \psi_n ,(y_1^n,\cdots , y_N^n) \right)\longrightarrow 
\inf_{\psi \in \textbf{C},(y_1,\cdots , y_N) \in E}
\mathfrak{I} \left( \psi ,(y_1,\cdots , y_N) \right)~~
\text{as} \qquad n\rightarrow \infty.
\end{equation*}
\end{theorem}

\begin{proof}
By Proposition 3.1
we have :
\begin{align*}
&\sum_{\gamma \in \Gamma} \Big\vert F_{\psi_{n,\gamma}}^{\gamma , 
	(y_1^n,\cdots , y_N^n)} ((f_1)_{\gamma , (y_1^n,\cdots , y_N^n)}) 
-\min_{x \in (M_1,M_2]} F_{\psi_{n,\gamma}}^{\gamma , (y_1^n,\cdots , 
	y_N^n)}(x) \Big\vert^2\\
&+ \Big\vert G_{\psi_{n,\gamma}}^{\gamma , (y_1^n,\cdots , y_N^n)} 
((f_2)_{\gamma , (y_1^n,\cdots , y_N^n)}) 
-\min_{x \in (M_1,M_2]} G_{\psi_{n,\gamma}}^{\gamma , (y_1^n,\cdots , 
	y_N^n)}(x) \Big\vert^2 \\
&\longrightarrow \inf_{(\psi_{\gamma_1},\cdots ,\psi_{\gamma_m} ) \in 
	\mathcal{A}^m , 
	(y_1,\cdots , y_N) \in E} \sum_{\gamma \in \Gamma}
\Big\vert F_{\psi_{\gamma}}^{\gamma , (y_1,\cdots , y_N)} ((f_1)_{\gamma , 
	(y_1,\cdots , y_N)}) 
-\min_{x \in (M_1,M_2]} F_{\psi_{\gamma}}^{\gamma , (y_1,\cdots , y_N)}(x) 
\Big\vert^2\\
&+\Big\vert G_{\psi_{\gamma}}^{\gamma , (y_1,\cdots , y_N)} ((f_2)_{\gamma , 
	(y_1,\cdots , y_N)}) 
-\min_{x \in (M_1,M_2]} G_{\psi_{\gamma}}^{\gamma , (y_1,\cdots , y_N)}(x) 
\Big\vert^2 , \\
& \qquad \qquad  \qquad \qquad \qquad \text{as} \qquad n\rightarrow \infty.
\end{align*}
Since by {\it Fact 2}  $~\psi_{n ,\gamma}$ is quasiconvex for every $\gamma \in 
\Gamma$, by Lemma 2.2 and  Remark 2.4 we can deduce that 
\begin{align*}
&\sum_{\gamma \in \Gamma}\Big\vert (f_1)_{\gamma , (y_1^n,\cdots , y_N^n)} 
-\text{arg-min}_{x \in (M_1,M_2]}\left[((g_1)_{\gamma , (y_1^n,\cdots , 
	y_N^n)}-x)^2+\psi_{n ,\gamma}(x) \right] \Big\vert^2 \\
&~~~~~+\Big\vert (f_2)_{\gamma , (y_1^n,\cdots , y_N^n)} -\text{arg-min}_{x \in 
	(M_1,M_2]}\left[((g_2)_{\gamma , (y_1^n,\cdots , y_N^n)}-x)^2+\psi_{n 
	,\gamma}(x) \right] \Big\vert^2 \\
&\longrightarrow\\
& \inf_{(\psi_{\gamma_1},\cdots ,\psi_{\gamma_m} ) \in \mathcal{A}^m
	,(y_1,\cdots ,y_N) \in E}\\
&\sum_{\gamma \in \Gamma}\Big\vert (f_1)_{\gamma ,(y_1,\cdots ,y_N)} 
-\text{arg-min}_{x \in (M_1,M_2]}\left[((g_1)_{\gamma , (y_1,\cdots 
	,y_n)}-x)^2+\psi_{\gamma}(x) \right] \Big\vert^2 \\
&~~~~~+\Big\vert (f_2)_{\gamma , (y_1,\cdots ,y_N)} -\text{arg-min}_{x \in 
	(M_1,M_2]}\left[((g_2)_{\gamma , (y_1,\cdots ,y_N)}-x)^2+\psi_{\gamma}(x) 
\right] \Big\vert^2, \\
& \qquad \qquad \qquad  \qquad \qquad \qquad  \text{as} \quad n\rightarrow 
\infty.
\end{align*}\\
\\
\\
Then by {\it Fact 3}, we will have
\begin{align*}
&\Big\Vert \sum_{\gamma \in \Gamma}\left[ (f_1)_{\gamma , (y_1^n,\cdots , 
	y_N^n)} -\text{arg-min}_{x \in (M_1,M_2]}((g_1)_{\gamma , (y_1^n,\cdots , 
	y_N^n)}-x)^2+\psi_{n ,\gamma}(x) \right]\varphi_{\gamma}^{(y_1^n,\cdots , 
	y_N^n)} \Big\Vert_{\mathcal{H}_{(y_1^n,\cdots , y_N^n)}}^2 \\
&+\Big\Vert \sum_{\gamma \in \Gamma}\left[ (f_2)_{\gamma , (y_1^n,\cdots , 
	y_N^n)} -\text{arg-min}_{x \in (M_1,M_2]}((g_2)_{\gamma , (y_1^n,\cdots , 
	y_N^n)}-x)^2+\psi_{n ,\gamma}(x) \right] \varphi_{\gamma}^{(y_1^n,\cdots , 
	y_N^n)} \Big\Vert_{\mathcal{H}_{(y_1^n,\cdots , y_N^n)}}^2 \\ 
&\longrightarrow \\
&\inf_{(\psi_{\gamma_1},\cdots ,\psi_{\gamma_m} ) \in \mathcal{A}^m
	,(y_1,\cdots ,y_N) \in E}\\
&\Big\Vert \sum_{\gamma \in \Gamma}\left[ (f_1)_{\gamma ,(y_1,\cdots ,y_N)} 
-\text{arg-min}_{x \in (M_1,M_2]}((g_1)_{\gamma , (y_1,\cdots 
	,y_N)}-x)^2+\psi_{\gamma}(x)
\right] \varphi_{\gamma}^{(y_1,\cdots , y_N)} 
\Big\Vert_{\mathcal{H}_{(y_1,\cdots , y_N)}}^2 \\
&+\Big\Vert \sum_{\gamma \in \Gamma} \left[ (f_2)_{\gamma , (y_1,\cdots 
,y_N)} 
-\text{arg-min}_{x \in (M_1,M_2]}((g_2)_{\gamma , (y_1,\cdots 
	,y_N)}-x)^2+\psi_{\gamma}(x)
\right] \varphi_{\gamma}^{(y_1,\cdots , y_N)} 
\Big\Vert_{\mathcal{H}_{(y_1,\cdots , y_N)}}^2 \\
& \qquad \qquad \qquad  \qquad \qquad \qquad  \text{as} \quad n\rightarrow 
\infty.
\end{align*}\\
Therefore, by {\it Fact 4} we have 
\begin{align*}
&\Big\Vert f_1 -\text{arg-min}_{f
	\in \mathcal{H}_{(M_1,M_2]}}\sum_{\gamma \in 
	\Gamma}(f_{\gamma  , (y_1^n,\cdots , y_N^n)}-(g_1)_{\gamma  , (y_1^n,\cdots , 
	y_N^n)})^2+\sum_{\gamma \in \Gamma}\psi_{n ,\gamma}(f_{\gamma})  
\Big\Vert_{\mathcal{H}_{(y_1^n,\cdots , y_N^n)}}^2 \\
&+\Big\Vert f_2 -\text{arg-min}_{f \in \mathcal{H}_{(M_1,M_2]}} 
\sum_{r\in\Gamma}(f_{\gamma , 
	(y_1^n,\cdots , y_N^n)}-(g_2)_{\gamma , (y_1^n,\cdots , 
y_N^n)})^2+\sum_{\gamma 
	\in \Gamma} \psi_{n ,\gamma}(f_{\gamma})   
\Big\Vert_{\mathcal{H}_{(y_1^n,\cdots , y_N^n)}}^2 \\
&\longrightarrow \\
&\inf_{(\psi_{\gamma_1},\cdots ,\psi_{\gamma_m} ) \in \mathcal{A}^m
	,(y_1,\cdots ,y_N) \in E}\\
&\Big\Vert f_1 -\text{arg-min}_{f
	\in \mathcal{H}_{(M_1,M_2]}}\sum_{\gamma \in 
	\Gamma}(f_{\gamma ,(y_1,\cdots , y_N)}-(g_1)_{\gamma , (y_1,\cdots , 
	y_N)})^2+\sum_{\gamma \in \Gamma}\psi_{\gamma}(f_{\gamma}) 
\Big\Vert_{\mathcal{H}_{(y_1,\cdots , y_N)}}^2 \\
&+ \Big\Vert f_2 -\text{arg-min}_{f \in 
\mathcal{H}_{(M_1,M_2]}}\sum_{\gamma\in\Gamma}(f_{\gamma
	,(y_1,\cdots , y_N)}-(g_2)_{\gamma ,(y_1,\cdots , y_N)})^2+\sum_{\gamma \in 
	\Gamma}\psi_{\gamma}(f_{\gamma})  \Big\Vert_{\mathcal{H}_{(y_1,\cdots , 
		y_N)}}^2 \\
& \qquad \qquad \qquad  \qquad \qquad \qquad  \text{as} \quad n\rightarrow 
\infty.
\end{align*}\\
Finally, since 
$\textbf{C}=\Big\lbrace \psi \big|~\psi(f)=\sum_{\gamma \in \Gamma} 
\psi_{\gamma}(f_{\gamma});~\psi_{\gamma} \in \mathcal{A}~ \text{for every} 
~\gamma \in \lbrace \gamma_1, \cdots , \gamma_m\rbrace \Big\rbrace$,
we conclude:\\
\begin{align*}
&\Big\Vert f_1 -\text{arg-min}_{f \in \mathcal{H}_{(M_1,M_2]}} \Vert f-g_1 
\Vert_{\mathcal{H}_{(y_1^n,\cdots , y_N^n)}}^2 +\psi_n(f) 
\Big\Vert_{\mathcal{H}_{(y_1^n,\cdots , y_N^n)}}^2\\
&+\Big\Vert f_2 -\text{arg-min}_{f \in \mathcal{H}_{(M_1,M_2]}} \Vert f-g_2 
\Vert_{\mathcal{H}_{(y_1^n,\cdots , y_N^n)}}^2 +\psi_n(f) 
\Big\Vert_{\mathcal{H}_{(y_1^n,\cdots , y_N^n)}}^2\\
&\longrightarrow 
\inf_{\psi \in \textbf{C},(y_1,\cdots , y_N) \in E} \Big\Vert f_1 
-\text{arg-min}_{f \in \mathcal{H}_{(M_1,M_2]}} \Vert f-g_1 
\Vert_{\mathcal{H}_{(y_1,\cdots , y_N)}}^2
+\psi(f) \Big\Vert_{\mathcal{H}_{(y_1,\cdots , y_N)}}^2 \\ 
&+\Big\Vert f_2 -\text{arg-min}_{f \in \mathcal{H}_{(M_1,M_2]}} \Vert f-g_2 
\Vert_{\mathcal{H}_{(y_1,\cdots , y_N)}}^2 +\psi(f) 
\Big\Vert_{\mathcal{H}_{(y_1,\cdots , y_N)}}^2 \\
& \qquad \qquad \qquad  \qquad \qquad \qquad  \text{as} \quad n\rightarrow 
\infty .
\end{align*}
\end{proof}


It should be mentioned that the constrained optimization problems that arise to achieve the step functions in Theorem 3.2 (or Proposition 3.1) do not have big differences with the corresponding constrained optimization problems with the step functions in Theorem 2.5, a sample of which was investigated in  Example 2.6. The only difference is that the variables for selecting the basis type should also be added to the variables of the optimization problems. The following example is provided to clarify this:


\begin{example}\label{3.3}
Suppose that in the assumptions of Theorem 3.2 (or Proposition 3.1), the set $E_1$ is as follows:
\begin{align*}
E_1=\{(y_1,\cdots , y_N)\in \mathbb{R}^N|&~g_i(y_1,\cdots , y_N)=c_i;~1\leq i\leq n,\\ 
&~h_j(y_1,\cdots , y_N)\geq d_j;~1\leq j\leq m\}.
\end{align*}
Also, suppose
$$\Gamma=\{\gamma_1,\gamma_2\},~ M_1=1,~M_2=3,~ \epsilon_n=\frac{1}{2^n}$$
 and 
$(f_1)_{\gamma_1,(y_1,\cdots , y_N)},(f_2)_{\gamma_1,(y_1,\cdots , y_N)},
(g_1)_{\gamma_1,(y_1,\cdots , y_N)}, (g_2)_{\gamma_1,(y_1,\cdots , y_N)}, (f_1)_{\gamma_2,(y_1,\cdots , y_N)},\\ 
(f_2)_{\gamma_2,(y_1,\cdots , y_N)}, (g_1)_{\gamma_2,(y_1,\cdots , y_N)}, (g_2)_{\gamma_2,(y_1,\cdots , y_N)}$ belonge to $(M_1,M_2]$ for every $(y_1,\cdots , y_N) \in E$. Then we will have:
$$\psi_{1,\gamma_1}=\sum_{t=1}^4\mathcal{C}_t^{1,\gamma_1}\chi_{A_t}, ~~
\psi_{1,\gamma_2}=\sum_{t=1}^4\mathcal{C}_t^{1,\gamma_2}\chi_{A_t}$$
where 
$$A_t=(1+\frac{t-1}{2}, 1+\frac{t}{2}]; ~1\leq t\leq 4$$
and $(\mathcal{C}_1^{1,\gamma_1},\mathcal{C}_2^{1,\gamma_1},\mathcal{C}_3^{1,\gamma_
1},{\mathcal{C}}_4^{1,\gamma_1}, \mathcal{C}_1^{1,\gamma_2}, 
\mathcal{C}_2^{1,\gamma_2},\mathcal{C}_3^{1,\gamma_2},\mathcal{C}_4^{1,\gamma_2}
, y_1^1,\cdots , y_N^1)$ 
is the solution of the following constrained optimization problem:
\begin{align*}
\min K_1& (x_1^{\gamma_1},x_2^{\gamma_1},x_3^{\gamma_1}, x_4^{\gamma_1}, 
x_1^{\gamma_2}, x_2^{\gamma_2}, x_3^{\gamma_2}, x_4^{\gamma_2}, y_1,\cdots , 
y_N)\\
& |x_1^{\gamma_1}-x_3^{\gamma_1}|+x_1^{\gamma_1}+x_3^{\gamma_1}-2x_2^{\gamma_1}
\geq 0\\
& |x_1^{\gamma_1}-x_4^{\gamma_1}|+x_1^{\gamma_1}+x_4^{\gamma_1}
-2x_2^{\gamma_1}\geq 0,\\
& |x_1^{\gamma_1}-x_4^{\gamma_1}|+x_1^{\gamma_1}+x_4^{\gamma_1}
-2x_3^{\gamma_1}\geq 0,\\
& |x_2^{\gamma_1}-x_4^{\gamma_1}|+x_2^{\gamma_1}+x_4^{\gamma_1}
-2x_3^{\gamma_1}\geq 0,\\
& |x_1^{\gamma_2}-x_3^{\gamma_2}|+x_1^{\gamma_2}+x_3^{\gamma_2}
-2x_2^{\gamma_2}\geq 0,\\
& |x_1^{\gamma_2}-x_4^{\gamma_2}|+x_1^{\gamma_2}+x_4^{\gamma_2}
-2x_2^{\gamma_2}\geq 0,\\
& |x_1^{\gamma_2}-x_4^{\gamma_2}|+x_1^{\gamma_2}+x_4^{\gamma_2}
-2x_3^{\gamma_2}\geq 0,\\
& |x_2^{\gamma_2}-x_4^{\gamma_2}|+x_2^{\gamma_2}+x_4^{\gamma_2}
-2x_3^{\gamma_2}\geq 0,\\
& g_i(y_1,\cdots , y_N)=c_i;~~ 1\leq i\leq n,\\
& h_j(y_1,\cdots , y_N)\geq d_j;~~ 1\leq j\leq m.
\end{align*}
where $K_1$ is as defined in Proposition 3.1.\\
\end{example}
\begin{remark}\label{3.4}
 It should be noted that Theorem 2.5
can be deduced from Theorem 3.2. 
In fact, by assigning a single-member category of bases to  
assumptions of Theorem 3.2, it will be converted to Theorem 2.5.\\
\end{remark}
\begin{remark}\label{3.5}
 It is clear that if the category of bases $\amalg $ is larger, the value of the $\text{inf}_{\psi\in \textbf{C},(y_1,\cdots,y_N)\in E}\mathfrak{I}(\psi,(y_1,\cdots,y_N))$ will be less. In other words, if $E_1 \subseteq E_2$ then $$\text{inf}_{\psi\in \textbf{C},(y_1,\cdots,y_N)\in E_2}\mathfrak{I}(\psi,(y_1,\cdots,y_N))\leq\text{inf}_{\psi\in \textbf{C},(y_1,\cdots,y_N)\in E_1}\mathfrak{I}(\psi,(y_1,\cdots,y_N)).$$ 
Hence, those papers in which an extension for a well-known basis is provided - such as the extensions of the {\it Chebyshev polynomials} [31, 47] - can be significant in leading to a better performance by the method proposed in this section as well as a framework including two or more kinds of bases - such as [27] that presents a framework which encompasses most multiscale systems from Applied Harmonic Analysis.  
 
Moreover, the following well-known theorem in the multi-resolution analysis literature (see Theorem 5.23 in [8], for example) in which a category of multi-resolution analyses is presented could also be useful:
\paragraph{Theorem ($\ast$)} 
Suppose $P(z)=(1/2){ \sum_k {p_k z^k}}$ is a polynomial that satisfies
the following conditions:
\begin{itemize}
\item[{\it (i)}] $P(1)=1$.

\item[{\it (ii)}] $\vert P(z) \vert^2 + \vert P(-z) \vert^2 = 1~for~\vert z \vert =1$.

\item[{\it (iii)}] $\vert P(e^{it}) \vert > 0 ~for~\vert t \vert \leq \pi/2$.

\end{itemize}
Let $\phi_0(x)$ be the Haar scaling function and let $\phi_n(x)=\sum_{k} p_k \phi_{n-1}(2x-k)$
for $n\geq1$. Then, the sequence $\phi_n$ converges pointwise and in $L^2$ to a function $\phi$,
which satisfies the normalization condition, $\hat{\phi}(0)=1/ \sqrt{2\pi} $, the orthonormality
condition 
$$\int_{-\infty}^{+\infty} \phi(x-n) \phi(x-m) dx = \delta_{nm}$$
and the scaling equation, $\phi(x)=\sum_{k} p_k \phi(2x-k)$.

This theorem will be applied in Section 5.2 for writing an algorithm corresponding to the learning method provided in this section. 
\end{remark}

\section{An alternative approach}
In this section, we present another method for attaining the constraint $\psi$ that differs from the approach described in Section 2. As mentioned in Remark 2.4, some further theoretical efforts were needed for the approach presented in Section 2, although they will probably not cause any challenges in the practical results. However, in this new approach that more computational tasks will probably be needed, no theoretical scratches exist. This new method will be addressed in Proposition 4.1 and  Theorem 4.2. Finally, the end of the section has been allocated to the discussion of Signal classification and Supervised Learning. 
\begin{proposition}\label{4.1}
Suppose $a_1,a_2,b_1,b_2 \in \mathbb{R}$ and $M_1,M_2 \in \mathbb{Z}$ are large 
enough such that $a_1,a_2,b_1,b_2 \in (M_1,M_2]$. 
Moreover, define the set  $\mathcal{A}$ and the function  $h$  from  
$\mathcal{A}$ to $\mathbb{R}$ by
\begin{align*}
\mathcal{A}=\{\psi \big|~\psi:(M_1,M_2]\rightarrow \mathbb{R};
~\psi ~ \text{is a regulated and quasiconvex function}\}, 
\end{align*}
\begin{align*}
&h:\mathcal{A}\rightarrow  \mathbb{R}\\
&h(\psi ) =\Big\vert a_1-\text{arg-min}_{x \in (M_1,M_2]}
 F_{\psi}(x)|^2+|a_2-\text{arg-min}_{x \in (M_1,M_2]}
G_{\psi}(x) \Big\vert^2 
\end{align*}
where $F_{\psi}(x)  = (x - b_1)^2 + \psi(x) $ and $G_{\psi}(x) =(x-b_2)^2 
+\psi(x) $.
Then there exists a sequence of step functions $\lbrace \psi_n \rbrace_{n \in \mathbb{N}}$ in $\mathcal{A}$ such that
\begin{equation*}
h(\psi_n)\longrightarrow \inf_{\psi\in\mathcal{A}}
 h(\psi) \qquad \text{as} \quad n\rightarrow 
\infty.
\end{equation*}

\end{proposition}
\begin{proof}
Suppose $S_n=\lbrace 1,\cdots, (M_2-M_1)\times 
2^n\rbrace$ and $\lbrace 
\epsilon_n \rbrace_{n=1}^{\infty}$ is  a sequence such that 
$\lim _{n\rightarrow \infty} \epsilon _n =0$. Define the set $D_n$ by 
\begin{align*}
D_n=\Big\{&(x_1,x_2, \cdots , x_{(M_2-M_1) \times 2^n}) \in 
\mathbb{R}^{(M_2-M_1) \times 2^n}\big |\\
& \forall (r,s,t) \in \lbrace (r,s,t)\big| ~r,s,t \in S_n,~ r<s<t \rbrace ;~x_s\leq \max \{x_r,x_t \}
\Big\}.
\end{align*}
Also Suppose $s_1=\left[ (b_1-M_1)\times 2^n \right] +1$, $s_2=\left[ 
(b_2-M_1)\times 2^n \right] +1$
 and define the set $W_n$ and the function $L_n$ 
from $W_n$ to $\mathbb{R}$ by 
\begin{sloppypar}
\begin{align*}
&  W_n=\Big \lbrace (i,j)\in S_n\times S_n|~\exists (x_1,\cdots ,x_{(M_2-M_1)\times 
2^n})\in D_n;\\
\\
& \forall s\in \lbrace 1,\cdots , (M_2-M_1)\times 2^n\rbrace \backslash 
\{i\}:\\
& \Biggl( \left[ 
M_1+\frac{2i+sign(s_1-i)-1}{2^{n+1}}+(\frac{1-sign(s_1-i)}{2
})\epsilon_n\right]\\
&(1-\chi_{\{s_1\}}(i))  + b_1(\chi_{\{s_1\}}(i))-b_1\Biggr)^2
\\
& +x_i\leq 
\Biggl(\left[M_1+\frac{2s+sign(s_1-s)-1}{2^{n+1}}+(\frac{1-sign(s_
1-s)}{2})\epsilon_n\right]\\
&(1-\chi_{\{s_1\}}(s)) +b_1
(\chi_{\{s_1\}}(s))-b_1\Biggr)^2
 +x_s, \\
 \\
 & \forall r\in\{1,\cdots ,(M_2-M_1)\times 2^n\}\backslash \{j\}:\\
&\Biggl( \left[ 
M_1+\frac{2j+sign(s_2-j)-1}{2^{n+1}}+(\frac{1-sign(s_2-j)}
{2})\epsilon_n\right]\\
&(1-\chi_{\{s_2\}}(j)) + b_2(\chi_{\{s_2\}}(j))-b_2\Biggr)^2\\
& +x_j\leq 
\Biggl(\left[M_1+\frac{2r+sign(s_2-r)-1}{2^{n+1}}+(\frac{1-sign(s_
2-r)}{2})\epsilon_n\right]\\
&(1-\chi_{\{s_2\}}(r)) +b_2
(\chi_{\{s_2\}}(r))-b_2\Biggr)^2 +x_r \Big \rbrace , \\
\end{align*}
\begin{align*}
&L_n(i,j)=|a_1-\Biggl(\left[ 
M_1+\frac{2i+sign(s_1-i)-1}{2^{n+1}}+(\frac{1-sign(s_1-i)}{2
})\epsilon_n\right]\\
&(1-
\chi_{\{s_1\}}(i)) + b_1(\chi_{\{s_1\}}(i))\Biggr)|^2\\
& + |a_2-\Biggl(\left[ 
M_1+\frac{2j+sign(s_2-j)-1}{2^{n+1}}+(\frac{1-sign(s_2-j)}{2
})\epsilon_n\right]\\
&(1-
\chi_{\{s_2\}}(j)) + b_2(\chi_{\{s_2\}}(j))\Biggr)|^2.
\end{align*}
\end{sloppypar}
Let $(i_n,j_n)$
 be a minimizer of the function $L_n$ for every  $n \in\mathbb{N}$
and \\
$(\mathcal{C}_1^n , \mathcal{C}_2^n , \cdots , \mathcal{C}^n_{(M_2-M_1)\times 
2^n})$ belongs to $D_n$ and be a solution to the following system of linear 
inequalities:
\begin{align*}
& \Biggl( \left[ 
M_1+\frac{2i_n+sign(s_1-i_n)-1}{2^{n+1}}+(\frac{1-sign(s_1-i_n)}{2})
\epsilon_n\right]\\
&(1-\chi_{\{s_1\}}(i_n)) + b_1(\chi_{\{s_1\}}(i_n))-b_1\Biggr)^2\\
& +x_{i_n}\leq 
\Biggl(\left[M_1+\frac{2s+sign(s_1-s)-1}{2^{n+1}}+(\frac{1-sign(s_
1-s)}{2})\epsilon_n\right]\\
&(1-\chi_{\{s_1\}}(s))
+b_1(\chi_{\{s_1\}}(s))-b_1\Biggr)^2 +x_s ~~ for~ s\in 
\{1,\cdots 
,(M_2-M_1)\times 2^n\}\backslash \{i_n\},\\
\\
& \Biggl( \left[ 
M_1+\frac{2j_n+sign(s_2-j_n)-1}{2^{n+1}}+(\frac{1-sign(s_2-j_n)}{2})
\epsilon_n\right]\\
&(1-\chi_{\{s_2\}}(j_n)) + b_2(\chi_{\{s_2\}}(j_n))-b_2\Biggr)^2\\
& +x_{j_n}\leq 
\Biggl(\left[M_1+\frac{2r+sign(s_2-r)-1}{2^{n+1}}+(\frac{1-sign(s_
2-r)}{2})\epsilon_n\right]\\
&(1-\chi_{\{s_2\}}(r)) +b_2(\chi_{
\{s_2\}}(r))-b_2\Biggr)^2 +x_r ~~ for~
r\in \{1,\cdots  
,(M_2-M_1)\times 2^n\}\backslash \{j_n\}.
\end{align*}
\\
Define the step function $\psi_n$ for every $n\in\mathbb{N}$ on $(M_1,M_2]$ as 
follows:
\begin{equation*}
\psi_n = \sum _{t=1}^{(M_2-M_1) \times 2^n} \mathcal{C}_t^n 
\chi_{{A}_t}
\end{equation*}
such that 
$${A}_t = \left( M_1 + \dfrac{t-1}{2^n} , M_1 + \dfrac{t}{2^n} 
\right];~1\leq t \leq (M_2-M_1) \times 2^n .$$\\
Now, we define the sets  $\mathcal{A}_n$ for  $n \in \mathbb{N}$:
\begin{align*}
\mathcal{A}_n&= \{\varphi \big | ~\varphi = \sum _{t=1}^{(M_2-M_1) \times 2^n} 
x_t \chi_{{A}_t};~ \forall r,s,t \in S_n,~ r<s<t  ;~x_s\leq \max \{x_r,x_t \} \}.
\end{align*}
Given {\it Fact 1}, it will be enough to prove that
$|h(\psi_n)-\inf_{\psi\in\mathcal{A}_n} h(\psi)|$  tends to $0$ as $n\rightarrow 
\infty.$\\
\\
Since the function $F_{\psi_n}$ is as follows:
$$F_{\psi_n}(x)=\begin{cases}
(x-b_1)^2+\mathcal{C}_1^n & x\in (M_1,M_1+\frac{1}{2^n}]\\
(x-b_1)^2+\mathcal{C}_2^n & x\in (M_1+\frac{1}{2^n},M_1+\frac{2}{2^n}]\\
\vdots & \vdots \\
(x-b_1)^2+\mathcal{C}_{(M_2-M_1)\times 2^n}^n & x\in 
(M_2-\frac{1}{2^n},M_2]
\end{cases}$$
 and moreover $(x-b_1)^2$ is decreasing over the interval $(M_1,b_1]$ and increasing over the interval $(b_1,M_2]$, we will have:
\begin{align*}
\exists c_n & \in \{(M_1+\frac{s}{2^n}-b_1)^2+\mathcal{C}^n_s|~s\in \{1,\cdots 
,s_1-1\}\}\cup\{\mathcal{C}^n_{s_1}\}\\
& \cup \{(M_1+\frac{s-1}{2^n}+\epsilon_n-b_1)^2+\mathcal{C}^n_s|~s\in \{s_1+1,\cdots 
,(M_2-M_1)\times 2^n\}\}
\end{align*}
such that:
$$|\min\{F_{\psi_n}(x)|~x\in (M_1,M_2]\}-c_n|\longrightarrow 0\quad \text{as} ~~~ 
n\longrightarrow \infty.$$\\
On the other hand, since we have 
\begin{align*}
 & \{(M_1+\frac{s}{2^n}-b_1)^2+\mathcal{C}^n_s|~s\in \{1,\cdots 
,s_1-1\}\}\cup\{\mathcal{C}^n_{s_1}\}\\
& \cup \{(M_1+\frac{s-1}{2^n}+\epsilon_n-b_1)^2+\mathcal{C}^n_s|~s\in \{s_1+1,\cdots 
,(M_2-M_1)\times 2^n\}\}\\
& = 
\{(M_1+\frac{2s+sign(s_1-s)-1}{2^{n+1}}+(\frac{1-sign(s_
1-s)}{2})\epsilon_n-b_1)^2\\
&~~~~(1-\chi_{\{s_1\}}(s))+
\mathcal{C}^n_s| ~s \in \{1,\cdots , (M_2-M_1)\times 2^n\}\},
\end{align*}
by assumptions, we deduce
\begin{align*}
&|\text{arg-min}_{x\in(M_1,M_2]}F_{\psi_n}(x)
-[(
M_1+\frac{2i_n+sign(s_1-i_n)-1}{2^{n+1}}+(\frac{1-sign(s_1-i_n)}{2})\epsilon_n
)\\
&(1-\chi_{\{s_1\}}(i_n)) + b_1(\chi_{\{s_1\}}(i_n))]|\longrightarrow 
0~~\text{as }~~n\rightarrow \infty.
\tag*{$(4\cdot 1)$}
\end{align*}\\
By a similar process like what was done for $F_{\psi_n}$, we will have:
\begin{align*}
&\vert \text{arg-min}_{x\in (M_1,M_2]}G_{\psi_n}(x)
-[(M_1+\frac{2j_n+sign(s_2-j_n)-1}{2^
{n+1}}
+(\frac{1-sign(s_2-j_n)}{2})
\epsilon_n)\\
&(1-\chi_{\{s_2\}}(j_n))
+ b_2(\chi_{\{s_2\}}(j_n))]|\longrightarrow 
0~~ \text{as}~~n\rightarrow +\infty.
\tag*{$(4\cdot 2)$}
\end{align*}\\
Finally, since by assumptions we have $(i_n,j_n)=\inf_{(i,j)\in W_n}L_n(i,j)$, 
hence
\begin{align*}
&|\inf_{\psi\in\mathcal{A}_n} h(\psi)-|a_1-
[(
M_1+\frac{2i_n+sign(s_1-i_n)-1}{2^{n+1}}+(\frac{1-sign(s_1-
i_n)}{2})\epsilon_n)
\\
&(1-\chi_{\{s_1\}}(i_n)) + b_1(\chi_{\{s_1\}}(i_n))]|^2-|a_2-
[(
M_1+\frac{2j_n+sign(s_2-j_n)-1}{2^{n+1}}\\
&+(\frac{1-sign(s_2-j_n)}{2})\epsilon_n)
(1-\chi_{\{s_2\}}(j_n))
 + 
b_2(\chi_{\{s_2\}}(j_n))]|^2\big|\longrightarrow 0
~~\text{as} ~~ 
~ n\rightarrow +\infty.
\tag*{$(4\cdot 3)$}
\end{align*}\\
Then, we conclude from (4.1), (4.2) and (4.3) that 
$$|h(\psi_n)-\inf_{\psi\in\mathcal{A}_n}
 h(\psi)|\rightarrow 0 \quad \text{as} ~~n\rightarrow 
\infty.$$
\end{proof}
\begin{remark}\label{4.2}
In Regression, sometimes for better performance, a penalty term is added to the {\it error} function. To illustrate, suppose we have many finite points from the graph of the {\it sine} function over an interval and decided to measure the ability of the regression in retrieving this underlying function by the polynomials using these finite points. If the natural number $n$ is large enough, a polynomial of degree $n$ can be found to pass exactly through each point. But, the polynomial may be a very poor representation of the {\it sine} function, because it may have wild fluctuations or extremely large values in the interval, while the {\it sine} function is smooth and its values are only in the range of -1 to 1. One technique that is often used to control this behavior in such cases is that of {\it regularization}, which involves adding a penalty term to the {\it error} function in order to discourage the polynomial coefficients from reaching large values; see Section 1.1 in [7], for more elaboration.\\
  
This background from Regression was used as a simile to imply that here where
we are looking for a learning model by a training set in a Hilbert space, it may also
be necessary to make further efforts to obtain more desirable results. In the above
example, the proper selection of the model (a polynomial) depended on the proper
selection of the polynomial coefficients. However, since having the right model here
depends on choosing the right functions $\psi_n$ then we have to be careful in choosing
these functions. First of all, it should be noted that although the primary reason
for adding the quasiconvexity to the functions $\psi_n$ was to make the function (1.6)
have a minimizer, it is also a powerful condition that solely can prevent the occurrence of unusual forms for these functions' graphs (for instance, it can prevent wildly
oscillating). In addition, another attempt can be made for a further improvement:
Recall that the values $(\mathcal{C}_1^n , \mathcal{C}_2^n , \cdots , \mathcal{C}^n_{(M_2-M_1)\times 
2^n})$ to obtain the functions $\psi_n$ in Proposition 4.1 were solutions for linear systems of inequalities. For functions $\psi_n$ to be more appropriate choices, namely to have graphs without abrupt and extreme transverse changes when longitudinal changes are minor, we suggest that these values, rather than be arbitrary solutions for the linear systems in Proposition 4.1, be that solution of the linear system that minimizes the function $K_n$ presented in Proposition 2.3 on the set of this linear system's solutions. In other words, it would be a solution to the following constrained optimization problems: 
$$\min ~ K_n(x_1,\cdots ,x_{(M_2-M_1)\times 2^n})$$
such that
\begin{sloppypar}
\begin{align*}
& \Biggl( \left[ 
M_1+\frac{2i_n+sign(s_1-i_n)-1}{2^{n+1}}+(\frac{1-sign(s_1-i_n)}{2})
\epsilon_n\right]\\
&(1-
\chi_{\{s_1\}}(i_n)) + b_1(\chi_{\{s_1\}}(i_n))-b_1\Biggr)^2\\
& +x_{i_n}\leq 
\Biggl(\left[M_1+\frac{2s+sign(s_1-s)-1}{2^{n+1}}+(\frac{1-sign(s_
1-s)}{2})\epsilon_n\right]\\
&(1-\chi_{\{s_1\}}(s))
+b_1(\chi_{\{s_1\}}(s))-b_1\Biggr)^2 +x_s\\
& for~ s\in 
\{1,\cdots 
,(M_2-M_1)\times 2^n\}\backslash \{i_n\},\\
\\
& \Biggl( \left[ 
M_1+\frac{2j_n+sign(s_2-j_n)-1}{2^{n+1}}+(\frac{1-sign(s_2-j_n)}{2})
\epsilon_n\right]\\
&(1-\chi_{\{s_2\}}(j_n)) + b_2(\chi_{\{s_2\}}(j_n))-b_2\Biggr)^2
\\
& +x_{j_n}\leq 
\Biggl(\left[M_1+\frac{2r+sign(s_2-r)-1}{2^{n+1}}+(\frac{1-sign(s_
2-r)}{2})\epsilon_n\right]\\
&(1-\chi_{\{s_2\}}(r))  +b_2(\chi_{
\{s_2\}}(r))-b_2\Biggr)^2 +x_r\\ 
&for~r\in
 \{1,\cdots  
,(M_2-M_1)\times 2^n\}\backslash \{j_n\}, \\
\\
& (x_1,\cdots ,x_{(M_2-M_1)\times 2^n})\in D_n.
\end{align*}
\end{sloppypar}
\end{remark}

\begin{theorem}\label{4.3}
Suppose $f_1,f_2,g_1,g_2 \in \mathcal{H}$ and $M_1, M_2 \in \mathbb{Z}$ is  
large enough such that $( f_1 )_{\gamma},( f_2 )_{\gamma},( g_1 )_{\gamma},( 
g_2 )_{\gamma} \in (M_1,M_2]$ for every $\gamma \in \Gamma$. 
Moreover, define the set  $\textbf{C}$ and the function  $\mathfrak{I}$  from  
$\textbf{C}$ to $\mathbb{R}$ by
\begin{align*}
\textbf{C}=\lbrace \psi:\mathcal{H} \rightarrow \mathbb{R}\big 
|~&\psi(f)=\sum_{\gamma \in \Gamma} \psi_{\gamma}(f_{\gamma});~ \forall \gamma 
\in \Gamma: ~\psi_{\gamma}:(M_1,M_2]\rightarrow \mathbb{R}\\
~ &\text{is a regulated and quasiconvex function}\rbrace ,\\
\\
&\mathfrak{I}:\textbf{C}\longrightarrow \mathbb{R}\\
&\mathfrak{I}(\psi)=\sum_{i=1}^2 \Vert \tilde{f}_{i,\psi} - f_i \Vert^2\\
&;\tilde{f}_{i,\psi} =\text{arg-min}_{f
	\in \mathcal{H}_{(M_1,M_2]}} 
\Vert f- g_i \Vert ^2 +\psi(f).
\end{align*}
Then there exists a sequence $\lbrace\psi_n\rbrace_{n \in {\mathbb{N}}}$ in $\textbf{C}$ such that the functions $\lbrace\psi_\gamma\rbrace_{\gamma \in {\Gamma}}$ corresponding to each $\psi_n$ are step functions and
$$\mathfrak{I}(\psi_n)\longrightarrow \inf_{\psi \in 
	\textbf{C}}\mathfrak{I}(\psi)\quad \text{as} \quad n\rightarrow \infty.$$

\end{theorem}
\begin{proof}
Suppose $S_n=\{1,\cdots , (M_2-M_1)\times 2^n\}$ and  
$\lbrace \epsilon_n \rbrace_{n=1}^{\infty}$ is a sequence such that $\lim 
_{n\rightarrow \infty}\epsilon_n =0$. Define the set $D_n$ by
\begin{align*}
D_n&= \Big\{ (x_1,x_2, \cdots , x_{(M_2-M_1) \times 2^n}) \in 
\mathbb{R}^{(M_2-M_1) \times 2^n}\big |\\
&~~~~~~\forall (r,s,t) \in \lbrace (r,s,t)\big |~ r,s,t \in S_n,~ r<s<t \rbrace ;~x_s\leq \max \{x_r,x_t \} \Big\}.
\end{align*}
Also suppose $s_{1,\gamma}=[((g_1)_{\gamma} -M_1)\times 
2^n]+1$, $s_{2,\gamma}=[((
 g_2 )_{\gamma} -M_1)\times 2^n]+1$ 
and define the set $W_n^\gamma$ 
and  the function $L_n^\gamma$ from $W_n^\gamma$ to $\mathbb{R}$ by 
\begin{align*}
&  W_n^\gamma=\Big \lbrace
 (t_1,t_2)\in S_n\times S_n|~\exists (x_1,\cdots ,x_{(M_2-M_1)\times 
2^n})\in \mathbb{R}^{(M_2-M_1)\times 2^n};\\
\\
& \forall s\in \lbrace 1,\cdots , (M_2-M_1)\times 2^n\rbrace \backslash 
\{t_1\}:\\
& \Biggl( \left[ 
M_1+\frac{2t_1+sign(s_{1,\gamma}-t_1)-1}{2^{n+1}}+(\frac{1-sign(s_{1,\gamma}-t_1
)}{2})\epsilon_n\right]\\
&(1-\chi_{\{s_{1,\gamma}\}}(t_1)) + 
(g_1)_\gamma(\chi_{\{s_{1,\gamma}\}}(t_1))-(g_1)_\gamma\Biggr)^2\\
& +x_{t_1}\leq 
\Biggl(\left[M_1+\frac{2s+sign(s_{1,\gamma}-s)-1}{2^{n+1}}+(\frac{1-sign(s_
{1,\gamma}-s)}{2})\epsilon_n\right]\\
&(1-\chi_{\{s_{1,\gamma}\}}(s)
)+(g_1)_\gamma(\chi_{
\{s_{1,\gamma}\}}(s))-(g_1)_\gamma\Biggr)^2+x_s,\\
\\
&   \forall r\in\{1,\cdots ,(M_2-M_1)\times 2^n\}\backslash 
\{t_2\}:\\
& \Biggl( \left[ 
M_1+\frac{2t_2+sign(s_{2,\gamma}-t_2)-1}{2^{n+1}}+(\frac{1-sign(s_{2,\gamma}-t_2
)}{2})\epsilon_n\right]\\
&(1-\chi_{\{s_{2,\gamma}\}}(t_2)) + 
(g_2)_\gamma(\chi_{\{s_{2,\gamma}\}}(t_2))-(g_2)_\gamma\Biggr)^2\\
& +x_{t_2}\leq 
\Biggl(\left[M_1+\frac{2r+sign(s_{2,\gamma}-r)-1}{2^{n+1}}+(\frac{1-sign(s_
{2,\gamma}-r)}{2})\epsilon_n\right]\\
&(1-\chi_{\{s_{2,\gamma}\}}(r)
)+(g_2)_\gamma(\chi_{
\{s_{2,\gamma}\}}(r))-(g_2)_\gamma\Biggr)^2 +x_r \Big \rbrace,
\end{align*}
\begin{align*}
L_n^\gamma(t_1,t_2)=&|(f_1)_\gamma-
 \Biggl( \left[ 
M_1+\frac{2t_1+sign(s_{1,\gamma}-t_1)-1}{2^{n+1}}+(\frac{1-sign(s_{1,\gamma}-t_1
)}{2})\epsilon_n\right]\\
&(1-
\chi_{\{s_{1,\gamma}\}}(t_1)) + 
(g_1)_\gamma(\chi_{\{s_{1,\gamma}\}}(t_1))\Biggr)|^2\\
& + |(f_2)_\gamma-
 \Biggl( \left[ 
M_1+\frac{2t_2+sign(s_{2,\gamma}-t_2)-1}{2^{n+1}}+(\frac{1-sign(s_{2,\gamma}-t_2
)}{2})\epsilon_n\right]\\
&(1-
\chi_{\{s_{2,\gamma}\}}(t_2)) + 
(g_2)_\gamma(\chi_{\{s_{2,\gamma}\}}(t_2))\Biggr)|^2.
\end{align*}
Let $(t_1^{n,\gamma},t_2^{n,\gamma})$
 be a minimizer of the function $L_n^\gamma$ for every  $n \in\mathbb{N}$
and\\ 
$(\mathcal{C}_1^{n,\gamma} , \mathcal{C}_2^{n,\gamma} , \cdots , \mathcal{C}^{n,\gamma}_{(M_2-M_1)\times 
2^n})$ belongs to $D_n$ and be a solution to the following system of linear 
inequalities:
\\
\begin{align*}
& \Biggl( \left[ 
M_1+\frac{2t_1^{n,\gamma}+sign(s_{1,\gamma}-t_1^{n,\gamma})-1}{2^{n+1}}+
(\frac{1-sign(s_{1,\gamma}-t_1^{n,\gamma})}
{2})\epsilon_n\right]\\
&(1-\chi_{\{s_{1,\gamma}\}}(t_1^{n,\gamma}))
 + 
(g_1)_\gamma(\chi_{\{s_{1,\gamma}\}}(t_1^{n,\gamma}))-(g_1)_\gamma\Biggr)^2\\
& +x_{t_1^{n,\gamma}}\leq 
\Biggl(\left[M_1+\frac{2s+sign(s_{1,\gamma}-s)-1}{2^{n+1}}+(\frac{1-sign(s_
{1,\gamma}-s)}{2})\epsilon_n\right]\\
&(1-\chi_{\{s_{1,\gamma}\}}(s)
)+(g_1)_\gamma(\chi_{
\{s_{1,\gamma}\}}(s))-(g_1)_\gamma\Biggr)^2+x_s\\
& \text{for}~ s\in 
\{1,\cdots  
,(M_2-M_1)\times 2^n\}\backslash \{t_1^{n,\gamma}\},\\
\\
& \Biggl( \left[ 
M_1+\frac{2t_2^{n,\gamma}+sign(s_{2,\gamma}-t_2^{n,\gamma})-1}{2^{n+1}}+
(\frac{1-sign(s_{2,\gamma}-t_2^{n,\gamma})}
{2})\epsilon_n\right]\\
&(1-\chi_{\{s_{2,\gamma}\}}(t_2^{n,\gamma})
 + 
(g_2)_\gamma(\chi_{\{s_{2,\gamma}\}}(t_2^{n,\gamma}))-(g_2)_\gamma\Biggr)^2\\
& +x_{t_2^{n,\gamma}}\leq 
\Biggl(\left[M_1+\frac{2r+sign(s_{2,\gamma}-r)-1}{2^{n+1}}+(\frac{1-sign(s_
{2,\gamma}-r)}{2})\epsilon_n\right]\\
&(1-\chi_{\{s_{2,\gamma}\}}(r)
)+(g_2)_\gamma(\chi_{
\{s_{2,\gamma}\}}(r))-(g_2)_\gamma\Biggr)^2 +x_r\\
& \text{for}~
r\in \{1,\cdots  
,(M_2-M_1)\times 2^n\}\backslash \{t_2^{n,\gamma}\}.
\end{align*}\\
Define the  function $\psi_n$ for every $n\in\mathbb{N}$ from $\mathcal{H}$ to 
$\mathbb{R}$ as follows:
\begin{align*}
&\psi_n:\mathcal{H}\rightarrow \mathbb{R}\\
&\psi_n(f)=\sum_{\gamma \in \Gamma} \psi_{n,\gamma}(f_{\gamma}),
\end{align*}
where  $\psi_{n,\gamma}$ for every $n\in \mathbb{N}$ and $\gamma \in \Gamma$ is 
a step function from $(M_1,M_2]$ to  $\mathbb{R}$ which is defined by 
\begin{equation*}
\psi_{n,\gamma} = \sum _{t=1}^{(M_2-M_1) \times 2^n} \mathcal{C}_t^{n,\gamma} 
\chi_{{A}_t}
\end{equation*}
such that 
${A}_t = \left( M_1 + \dfrac{t-1}{2^n} , M_1 + \dfrac{t}{2^n} 
\right];~1\leq t \leq (M_2-M_1) \times 2^n .$
Now, we define the set  $\mathcal{A}$ for every  $\gamma \in \Gamma$:
\begin{equation*}
\mathcal{A}=\{ \psi \big|~\psi:(M_1,M_2]\rightarrow \mathbb{R},~\psi ~ 
\text{is a regulated and quasiconvex function}\}.
\end{equation*}
By Proposition 4.1 
for every $\gamma\in \Gamma$ we have :
\begin{align*}
&\Big\vert(f_1)_{\gamma} -
\text{arg- min}_{x \in 
	(M_1,M_2]}\left[(x-(g_1)_\gamma)^2+ 
\psi_{n ,\gamma}(x)\right]|^2\\
& +\Big\vert(f_2)_{\gamma} -
\text{arg-min}_{x \in (M_1,M_2]}\left[(x-(g_2)_\gamma)^2+
\psi_{n ,\gamma}(x)\right]|^2
\\
&\longrightarrow 
\inf_{\psi_{\gamma} \in \mathcal{A}} \Big\vert (f_1)_{\gamma}-\text{arg-min}_{x \in (M_1,M_2]}
\left[(x-(g_1)_\gamma)^2+
\psi_\gamma(x)\right]|^2\\
&+\Big\vert(f_2)_\gamma- \text{arg-min}_{x \in (M_1,M_2]}
\left[(x-(g_2)_\gamma)^2+
\psi_\gamma(x)\right]\Big\vert^2.
\end{align*}
The remainder of the proof is compeletly similar to what was done for Theroem 
2.5 in Section 2. 
\end{proof}
What was addressed in this section and Section 2 was efforts to solve special cases of the problem (2.10) in terms of an application in signal denoising. However, it could possibly be viewed as other perspectives and applied in other applications. Moreover, this problem may possibly be useful in classification. Ideas around this are presented in the following remark: 
\begin{remark}\label{4.4}
In this remark, we present some schemes that it is desired to be beneficial for the classification of some signal sets in which the signals of each set have a common feature. For a theoretical setting, suppose $\mathcal{A}_j,~ j=1,\dots,m$ are sets of image samples that images in each set have been degraded by a common noise model. Now, regarding these data, we want to find an approach to determine that an image $f^\circ$ with an unknown noise model belongs to which of these sets. By considering the minimization problem (2-10), two approaches could be possible:\\ 
The first one is similar to other classification approaches in literature (see e.g., [50, 51]): considering $\mathcal{H'}=\mathbb{R}$ and $\textbf{C}$ is a single-member set (e.g., $L^2$-norm), and corresponding images in $\mathcal{A}_j=\lbrace f^j_i\rbrace_{i=1}^{n}$ to $j$ for $j=1,..,m$, i.e., finding the minimization of the following function: 
\begin{align*}
& \Phi:\textbf{U}\longrightarrow \mathbb{R}\\
& \Phi(K)=\sum_{j=1}^m \sum_{i=1}^n\|\tilde{f}_{j}-f^j_i\|_{\mathcal{H}}^2\\
&
;\tilde{f}_{j}=\text{arg-min}_{f\in \mathcal{H}}\|Kf-
j\|_{\mathcal{H}'}^2+\Vert f \Vert^2
\end{align*}  
where $U$ is a class of operators. If $K^*$ is the minimizer of the (4-10), by the value of $K^*(f^\circ)$, we can determine that $f^\circ$ to which set belongs. \\
The second approach could be by corresponding all images in the set $\mathcal{A}_j$ to a fixed image $g_j$, which could be an image with the same level of darkness, or it could be defined as follows: 
                    $$g_j=\sum_{\gamma \in \Gamma} j\varphi_{\gamma}.$$ 
In this case, the following minimization problem should be solved: 
\begin{align*}
& \Phi:\textbf{C}\times\textbf{U}\longrightarrow \mathbb{R}\\
& \Phi(\psi,K)=\sum_{j=1}^n \sum_{i=1}^n\|\tilde{f}_{j}-f_i^j\|_{\mathcal{H}}^2\\
&
;\tilde{f}_{j}=\text{arg-min}_{f\in \mathcal{H}}\|Kf-
g_j\|_{\mathcal{H}'}^2+\psi(f)
\tag*{$(4\cdot 4)$}
\end{align*}  
where $\textbf{C}$ is again as defined in Introduction and $U$ is a class of operators. Then, if $(K^*, \psi^*)$ is the minimizer, $f^\circ$ will belong to $\mathcal{A}_{j_\circ}$ for 
       $$j_\circ=\text{arg-min}_{j \in \lbrace1,\dots,m\rbrace}\|g^*-g_j\|_{\mathcal{H}'}^2$$ 
where 
              $$g^*=\text{arg-min}_{f\in \mathcal{H}}\|K^*(f)-f^\circ\|_{\mathcal{H}'}^2+\psi^*(f).$$ 
In the second approach, in addition to the operator, the regularization term also is selected among a class of functions. In (4.4), considering $B_{W,P}$ instead of $\textbf{C}$ for the class of regularization terms will convert the problem to an easier one for solving as will be discussed later in Section 6. 
\end{remark}
Finally, this section is concluded by the following remark about unifying different regularization terms: 
\begin{remark}\label{4.5}
This remark is dedicated to clarify how utilizing classification can promote a better learning method in the case that signal samples have different kinds of noises. Generally speaking, suppose there are $m$ different noise models and $m$ sample sets of images. Moreover, suppose $m$ sample sets of noisy images have been achieved by adding the different noise models to the sample sets of images. Indeed, providing a learning method by these sample sets with the capability of modifying all kinds of noises will be addressed. For a theoretical setting, suppose $\lbrace f_{1, j},\dots,f_{n, j}\rbrace$ for $j=1,…,m$ are $m$ different sample sets of original images and $g_{i,j}=f_{i,j}+e_{i,j}$ for $i=1,\dots,n,~ j=1,\dots,m$ are noisy images where $\lbrace e_{i,j} \rbrace_{i=1}^n$ for every $j$ is a set of noises with same models. Now, suppose the noisy image $g^\circ$ which degraded by one of these noise models randomly; that means, we are not aware of the noise models. The aim is providing a desirable minimization problem with an appropriate regularization penalty for denoising the $g^\circ$ regarding $(f_{i,j}, g_{i,j});~ i=1,\dots,n,~ j=1,\dots,m$.\\
First, it should be noticed that the following minimization problem:
                      $$ \text{arg-min}_{f\in \mathcal{H}}\|f-g^\circ\|^2+\psi_\circ(f),$$
where the regularization $\psi_\circ$ is achieved by $\lbrace(f_{i,j}, g_{i,j});~ i=1,\dots,n,~ j=1,\dots,m\rbrace$, could not be the best choice because the training sets regarding noise models which are different from the $g^\circ$’s noise model contribute to building $\psi_\circ$. To fix this shortcoming, we suggest the following minimization problems:
                         $$\text{arg-min}_{f\in \mathcal{H}}\|f-g^\circ\|^2+\psi_j(f);~~1\leq j \leq m,$$ 
where the regularization $\psi_j$ is created regarding $\lbrace(f_{i,j}, g_{i,j}); ~i=1,\dots,n\rbrace$. We know one of these minimization problems is appropriate for denoising $g^\circ$; however, since the kind of noise model is unknown, the problem is `which one should be applied'. We can solve this problem by employing one of the classification approaches provided in Remark 4.4 or another suitable possible one that possibly exists in the literature (see e.g., [52, 53]); namely, we can classify $\lbrace g_{i,j}; ~i=1,..,n \rbrace$ for $j=1,..,m$ i.e. there exists a $\varphi$ in a way that 
                                   $$\varphi (g_{i,j}) \simeq j~~\text{for}~~ j = 1,\dots,m.$$ 
Therefore, if $\lbrace g_{i,j}; ~i=1,..,n \rbrace$ for $j=1,..,m$ is large enough and the classification approach has a well performance, it is expected that 
                            $$\varphi(g) \in (j-1/2, j+1/2)$$ 
where $g$ is a noisy image degraded by the $j$th noise model. Consequently, the following minimization problem will be an appropriate one for denoising $g^\circ$: 
                      $$\text{arg-min}_{f\in \mathcal{H}}\|f-g^\circ\|^2+\lambda_1\psi_1(f)+\dots+\lambda_m \psi_m(f),$$ 
where $\lambda_j=1-\text{sign}[2\vert\varphi(g^\circ)-j\vert]$. \\
For the end of this remark, it has been provided some discussions or questions in the hope that some ideas are derived from them for future works: \\
Firstly, it should be mentioned that it seems a similar approach like ones addressed here – i.e. adding classification to the learning process – could possibly have been done about other learning methods in which the approach like this paper is matching the pair elements from two sample sets; like the {\it{image to image translation}} in {\it{Deep Neural Network}}; see [49]. \\
Secondly, it is about the possibility of creating a denoising formula by the capability of modifying any noisy image regardless of the type of its noise. We know that the humans' perception of noise and the unpleasant feeling they have from it is related to the information they already have in their mind. The collection of images; stored in the memory of a person over time from birth, leads to the formation of an expectation for them which is in contrast to what it receives from a noisy image. However, the volume of information stored in memory is too high, they are limited. On the other hand, the power of human vision is limited as well; in other words from a point on, increasing the number of pixels in an image does not change the image quality from a human point of view. The question now is whether, based on these limitations, it can be argued that a denoising formula with the ability to remove a finite number of noises’ type could possibly modify any image that is noisy from the human point of view. If the answer to this question is yes, another question arises that: Is it possible to provide a super formula that can fulfill  this task with the idea presented in this remark? \\
The third is a discussion that comes around the subject of upgrading a learning model that we have prior to a better one after the arrival of a new training set. Suppose we have found a regularization term regarding a training set in prior. Now, a new training set has come in hand and we want to promote our learning method regarding this new one. In the case that the noise model that was added to original signals in the new training set is not the same as the noise model of the prior training set, we will not have a serious challenge because in this case we merely need to build a new regularization regarding the new training set and add it to the minimization problem which is explained in this remark. However, a discussion is needed when the noise models are the same. In this situation, regardless of the volume of computational tasks, upgrading could be happening by replacing the regularization term with a new one that is regarded as a training set including both the new training set and the prior one ; namely, considering both training sets as one unique training set. But questions or discussions arise around the possibility of approaches to promote the model when utilizing the prior efforts and computational tasks that have been done for the model is a matter of concern; in other words , it is important to save on the volume of computational tasks. In this case, one solution (not probably the best one) could be adjusting the regularization of the model by the new training set. For instance, trying to find the solution to the following minimization problem: 
\begin{align*}
&\mathfrak{I}:\mathbb{R}^2 \longrightarrow \mathbb{R}\\
&\mathfrak{I}(a, b)=\sum_{i=1}^2 \Vert \tilde{f}_{i, a, b} - f_i \Vert^2\\
&;\tilde{f}_{i, a, b} =\text{arg-min}_{f
	\in \mathcal{H}} 
\Vert f- g_i \Vert ^2 +P_{a,b}(\psi(f)).
\end{align*}
where $\psi$ is the regularization term regarding the prior training set, $\lbrace(f_i,g_i)\vert ~ i=1,\dots,n\rbrace$ is the new training set and $P_{a,b}(\psi(f))=a(\psi(f))+b$ where $(a,b) \in \mathbb{R}^2$. However, it seems further discussions are needed on this. 
\end{remark}


\section{Implementing the learning methods }

In this section, we describe algorithms associated with the learning methods presented in Sections 2, 3 and 4. Because Theorem 2.5 in Section 2, Theorem 3.2 in Section 3 and Theorem 4.3 in Section 4 must be executed in the form of a computer program, to clarify how to write a program related to these theorems, we will state each of them in the steps of an algorithm. However, here $m$ in the training set $(f_i, g_i)$ for $i=1,\cdots,m$ is considered an arbitrary element of $\mathbb{N}$ instead of $m=2$.\\ 

It should be noted that the learning method provided in this paper could also be built in terms of another view of using data, namely, to apply a set of ground signals $\lbrace{f_i}\vert~ i=1,\cdots,m\rbrace$ and a specific artificial noise model. For instance, suppose that we want to find a regularization for removing Gaussian noise. By a minor change in our approach, we can reach to an algorithm for this aim. If we build $g_i$'s by adding Gaussian noise to $f_i$'s, the best regularization for Gaussian noise can be attained, however, a very big set of $f_i$'s is probably needed to have a desirable outcome in this case. All methods in Sections 2, 3 and 4 could be rewritten from this perspective. To illustrate, the algorithm in Subsection 5.2 will be written accordingly.\\

\subsection{An algorithm for finding the optimal regularization by the approach presented in Section 2}
In this subsection, we present the steps of an algorithm that is associated with Theorem 2.5. The volume of computational tasks that are needed for the computer program regarding this algorithm is less than what is needed for the algorithms in the two next subsections. 
As we have already mentioned in Section 3, Multi-resolution analyses can be the options to be considered as the bases in the learning method.  Therefore, in one of the first steps of the algorithm a ``{\it scaling function}" is taken to build a multi-resolution analysis $\lbrace V_i \rbrace$  regarding it. Then, the signals are projected to a space $V_j$ for a suitable $ j$. Afterwards, remaining steps would consist of signals' decomposition, processing and reconstruction:
\paragraph{Step 1:}
get $n,m,\epsilon$ from input. 
\begin{itemize}
\item[ ]
{\small{\textit{
``$\epsilon$ should be the minimum value that can be 
defined for the computer."}}}
\end{itemize}
\paragraph{Step 2:}
get $(f_i,g_i)$ for $i=1,\cdots , m$ from input.
\begin{itemize}
\item[ ]
{\small{\textit{
 ``$g_i$'s are noisy signals and $f_i$'s are the original signals i.e. 
$g_i=f_i+e_i$."
}}}
\end{itemize}
\paragraph{Step 3:}
get $g$ from input.
\begin{itemize}
\item[ ]
{\small{\textit{
``$g$ is the noisy signal that its noise is supposed to be removed by the model 
that will be obtained from the set $\{(f_i,g_i)|~i=1,\cdots , m\}$."
}}}
\end{itemize}
\paragraph{Step 4:}
get the scaling function $\phi$ from input.
\begin{itemize}
\item[ ]
{\small{\textit{
``By the scaling function $\phi$, the associated multi-resolution analysis $\lbrace V_j; ~j\in\mathbb{Z} \rbrace$ is clarified:
$$V_j=\langle \lbrace 2^{j/2}\phi(2^jx-k); ~k\in\mathbb{Z} \rbrace \rangle.$$
However, the exact definition of it is not needed, because only the associated scaling coefficients are used in the decomposition algorithm and reconstruction algorithm."
}}}
\end{itemize}
\paragraph{Step 5:}
get $j_\circ$ from input.
\begin{itemize}
\item[ ]
{\small{\textit{
``$V_{j_\circ}$ is the space on which $ f_i $'s, $ g_i $'s and $ g $ are supposed to be projected."
}}}
\end{itemize}
\paragraph{Step 6:}
\begin{align*}
&a_{k, i}:=\int_{-\infty}^{+\infty} f_i(x)\overline{\phi(2^{j_\circ}x-k)} dx;\quad k\in {\mathbb{Z}},~~i=1,\cdots,m.\\
&
b_{k, i}:=\int_{-\infty}^{+\infty} g_i(x)\overline{\phi(2^{j_\circ}x-k)} dx;\quad k\in {\mathbb{Z}},~~i=1,\cdots,m.\\
&
b_{k}:=\int_{-\infty}^{+\infty} g(x)\overline{\phi(2^{j_\circ}x-k)} dx;\quad k\in {\mathbb{Z}}.\\
\end{align*}

\begin{itemize}
\item[ ]
{\small{\textit{
``Note that we have:
\begin{align*}
&P_{j_\circ}f_i=\sum_{k\in {\mathbb{Z}}} a_{k, i}\phi(2^{j_\circ}x-k) ;\quad i=1,\cdots,m,\\
&
P_{j_\circ}g_i=\sum_{k\in {\mathbb{Z}}} b_{k, i}\phi(2^{j_\circ}x-k) ;\quad i=1,\cdots,m,\\
&
P_{j_\circ}g=\sum_{k\in {\mathbb{Z}}} b_{k}\phi(2^{j_\circ}x-k) 
\end{align*}
where $P_{j_\circ}f_i$, $P_{j_\circ}g_i$ and $P_{j_\circ}g$ are the orthogonal projections of $f_i$, $g_i$ and $g$ onto $V_{j_\circ}$."
}}}
\end{itemize}
\paragraph{Step 7:}
determine sets $\lbrace(f_i)_\gamma\vert~\gamma \in {\Gamma}\rbrace$, $\lbrace(g_i)_\gamma \vert ~ \gamma \in {\Gamma}\rbrace$, $\lbrace g_\gamma \vert ~ \gamma \in {\Gamma}\rbrace$ by decomposing $(a_{k, i})_{k \in \mathbb{Z}}$, $(b_{k, i})_{k \in \mathbb{Z}}$, $(b_{i})_{k \in \mathbb{Z}}$ respectively by the iterative process for $i=1,\cdots,m$.
\begin{itemize}
\item[ ]
{\small{\textit{
``This step is the decomposition algorithm by using the scaling coefficients."
}}}
\end{itemize}
\paragraph{Step 8:}
determine integers $M_1$ and $M_2$ such that $\forall \gamma \in \Gamma, ~
\forall i;~ i=1,\cdots , m:$
$$ (f_i)_\gamma, (g_i)_\gamma \in (M_1,M_2].$$
\paragraph{Step 9:}
$S_n:=\{1,\cdots ,(M_2-M_1)\times 2^n\}$.\\
\paragraph{Step 10:}
$A_t:=(M_1+\frac{t-1}{2^n},M_1+\frac{t}{2^n}]; ~1\leq t\leq (M_2-M_1)\times 
2^n$.\\
\paragraph{Step 11:}
Perform substeps 11.1 to 11.7 for every $\gamma\in \Gamma$: \\
\\
\\
{\it Substep 11.1}\\ 
$$s_i:=[((g_i)_\gamma-M_1)\times 2^n]+1; ~1\leq i\leq m.$$
{\it Substep 11.2}\\
$$t_i:=[((f_i)_\gamma-M_1)\times 2^n]+1; ~1\leq i\leq m.$$ 
{\it Substep 11.3}\\
\begin{align*}
K_n \left(x_1,x_2, \cdots , x_{(M_2-M_1) \times 2^n} \right) 
:=\sum_{i=1}^m&\Big\vert \min \lbrace (M_1 + \dfrac{s}{2^n}-( g_i )_{\gamma})^2+x_s -(( f_i 
)_{\gamma}-( g_i )_{\gamma})^2 -x_{t_i} \big | \\
&~1\leq s \leq {s_i}-1 \rbrace \cup \lbrace x_{s_i}-(( f_i 
)_{\gamma}-( g_i )_{\gamma})^2 -x_{t_i} \rbrace \\
&\cup \lbrace (M_1 + \dfrac{s-1}{2^n}+\epsilon -( g_i )_{\gamma})^2 +x_s -(( 
f_i )_{\gamma}-( g_i )_{\gamma})^2 -x_{t_i}\big |\\
&{s_i}+1\leq s \leq (M_2-M_1) \times 2^n \rbrace \Big\vert^2.\\
\end{align*}
{\it Subsetp 11.4}~~
determine a solution of the following system and call it 
$(\mathcal{C}_1,\cdots , \mathcal{C}_{(M_2-M_1)\times 2^n})$:
$$\min ~ K_n(x_1,\cdots ,x_{(M_2-M_1)\times 2^n})$$
such that:
\begin{align*}
& |x_r-x_t|+x_r+x_t-2x_s\geq 0
\quad \text{for} \quad (r,s,t)\in\{(r,s,t)|~r,s,t\in S_n,~ r<s<t\}.\\
\end{align*}

\noindent
{\it Substep 11.5}\\
$$\psi(x):=\sum_{t=1}^{(M_2-M_1)\times 2^n}
\mathcal{C}_t\chi_{A_t}(x).$$
{\it Substep 11.6}\\
$$\Phi(x)=(x-g_\gamma)^2+\psi(x).$$
{\it Substep 11.7}\\ 
$$f_\gamma:=\text{arg-min}_{x\in (M_1,M_2]}\Phi(x).$$
\paragraph{Step 12:}
determine $(a_k)_{k \in \mathbb{Z}}$ from $ \lbrace f_{\gamma}\vert~\gamma \in \Gamma\rbrace $ by the reconstruction algorithm. 
\begin{itemize}
\item[ ]
{{\small\textit{
``This step is the reconstruction algorithm by using the scaling coefficients."
}}}
\end{itemize}
\paragraph{Step 13:}
\begin{align*}
&f:=\sum_{k\in {\mathbb{Z}}} a_{k}\phi(2^{j_\circ}x-k). 
\end{align*}

\paragraph{Step 14:}
give $f$ as output.
\begin{itemize}
\item[ ]
{{\small\textit{
``$f$ is a modified signal for the signal $g$ that was taken at the step 3."
}}}
\end{itemize}

\subsection{The optimal ``multi-resolution system" and ``regularization"}
The importance of a multi-resolution analysis and choosing it for an application depends on the characteristics of the signals' class in that application. For instance, in all multi-resolution analysis literature, the following elucidation about the importance of inventing the Daubechies wavelets can be seen: ``The drawback with the Haar decomposition algorithm is that both Haar scaling function and Haar wavelet are discontinuous. As a result, the Haar decomposition algorithm provides only crude approximations to a continuously varying signal". Other similar discussions around comparing other different multi-resolution systems exist in the literature as well. Moreover, the kind of noise model probably could also be important in choosing a multi-resolution system. As a whole, a multi-resolution system has not an intrinsic value, however, its value is related to the kind of data or the noise model that supposed to be removed from signals. This shows the importance of presenting a framework by which the ``scaling function" is attained from data. This is what is done here by providing an algorithm extracted from Theorem 3.2. \\

In this subsection, the optimal ``regularization" from set $\textbf{C}$ and ``multi-resolution analysis" from a category of multi-resolution analyses, provided in Theorem ($\ast$) in Remark 3.5, will be achieved for removing a specific noise model from a set of signals - that are in the same class. In other words, the input of the algorithm is a collection of noises $\lbrace e_i \rbrace$ with same probability distributions and a collection of signals $\lbrace f_i \rbrace$, whereas the outputs are a regularization term $\psi$ and a set $\lbrace p_1, p_2,\cdots, p_n\rbrace$ that includes coefficients in the scaling relation of the best scaling function $\phi$. The values of the scaling function $\phi$ could later be attained by the iterative method given in Theorem ($\ast$) or by a similar method presented in Section 6.4 in [8] with less computation tasks. However, the fact that the coefficients in the scaling relation are only utilized in the decomposition and reconstruction process is a point that has a key role in writing the algorithm: 
\paragraph{Step 1:}
get $n,m, l,\epsilon$ from input. 
\begin{itemize}
\item[ ]
{\small{\textit{
``$\epsilon$ should be the minimum value that can be 
defined for the computer. Moreover, l is the number of scaling coefficients."}}}
\end{itemize}
\paragraph{Step 2:}
get $(f_i,e_i)$ for $i=1,\cdots , m$ from input.\\
\paragraph{Step 3:}
get $\lbrace k_1,\cdots, k_l\rbrace \subseteq \mathbb{Z} $ from input.
\begin{itemize}
\item[ ]
{\small{\textit{
``we work under the assumption that $P(z)=(1/2){ \sum_k {p_k z^k}}$ in Theorem ($\ast$) in Remark 3.5 is a polynomial with finite nonzero coefficients i.e. for $k \in {\mathbb{Z}}\setminus {\lbrace k_1,\cdots, k_l\rbrace};~p_k=0 $."}}}
\end{itemize}
\paragraph{Step 4:}
$g_i:=f_i+e_i;~i=1,\cdots , m. $ 
\begin{itemize}
\item[ ]
{\small{\textit{
``$ g_i $'s are noisy signals corresponding to $ f_i $'s. It is supposed to obtain a learning method from the training set $ \lbrace (f_i,g_i)\vert~i=1,\cdots,m\rbrace $ ."
}}}
\end{itemize}
\paragraph{Step 5:}
get $j_\circ$ from input.
\begin{itemize}
\item[ ]
{\small{\textit{
``$V_{j_\circ}$ is the space on which $ f_i $'s, $ g_i $'s and $ g $ are supposed to be projected."
}}}
\end{itemize}
\paragraph{Step 6:}
\begin{align*}
&a_{k, i}:= f_i(k/{2^{j_\circ}});\quad k\in {\mathbb{Z}},~~i=1,\cdots,m.\\
&
b_{k, i}:= g_i(k/{2^{j_\circ}});\quad k\in {\mathbb{Z}},~~i=1,\cdots,m.\\
\end{align*}

\begin{itemize}
\item[ ]
{\small{\textit{
``Note that we have:
\begin{align*}
&2^{j_\circ}\int_{-\infty}^{+\infty} f(x)\overline{\phi(2^{j_\circ}x-k)} dx\simeq m f(k/{2^{j_\circ}})\\
\end{align*}
where $m= \int \overline{\phi(x)}dx$. Since we are working under the assumptions in Theorem ($\ast$) in Remark 3.5, we will have $\int \overline{\phi(x)}dx=1.$"
}}}
\end{itemize}
\paragraph{Step 7:}
determine sets $\lbrace (f_i)_{\gamma, (p_1,\cdots,p_l)}\vert ~ \gamma \in \Gamma \rbrace $, $\lbrace (g_i)_{\gamma, (p_1,\cdots,p_l)}\vert ~ \gamma \in \Gamma \rbrace $ by decomposing $ (a_{k, i})_{k \in \mathbb{Z}} $, $ (b_{k, i})_{k \in \mathbb{Z}} $ respectively by considering $ p_1,\cdots, p_l $ as scaling coefficients.
\begin{itemize}
\item[ ]
{\small{\textit{
``This step is the decomposition algorithm by using the scaling coefficients. Since $ p_1,\cdots, p_l $ are unknown quantities to be found, then the outcomes of decomposition algorithm consist of these variables. That's why these outcomes were represented by $ (f_i)_{\gamma, (p_1,\cdots,p_l)}$, $(g_i)_{\gamma, (p_1,\cdots,p_l)}$ for $ i=1,\cdots,m.$"
}}}
\end{itemize}
\paragraph{Step 8:}
determine integers $M_1$ and $M_2$ such that $\forall \gamma \in \Gamma$,~ 
$\forall i;~ i=1,\cdots , m:$\\ 
$$(f_i)_{\gamma, (p_1,\cdots,p_l)}, (g_i)_{\gamma, (p_1,\cdots,p_l)} \in (M_1,M_2].$$
\begin{itemize}
\item[ ]
{\small{\textit{
``Note that $ p_1,\cdots, p_l $ are variables which their values are supposed to be specified at the end of the algorithm. Therefore, we don't have the values of $ (f_i)_{\gamma, (p_1,\cdots, p_l)}$, $(g_i)_{\gamma, (p_1,\cdots, p_l)}$ in this step. However, $(M_1, M_2]$ should be large enough such that for all reasonable choices of $ (p_1,\cdots, p_l) $, we have $(f_i)_{\gamma, (p_1,\cdots,p_l)}, (g_i)_{\gamma, (p_1,\cdots,p_l)} \in (M_1,M_2]$ for $ i=1,\cdots,m, \gamma \in \Gamma. $"
}}}
\end{itemize}
\paragraph{Step 9:}
$S_n:=\{1,\cdots ,(M_2-M_1)\times 2^n\}$.\\
\paragraph{Step 10:}
$A_t:=(M_1+\frac{t-1}{2^n},M_1+\frac{t}{2^n}]; ~1\leq t\leq (M_2-M_1)\times 
2^n$.\\
\paragraph{Step 11:}
Perform subsetps 11.1 to 11.3 for every $\gamma\in \Gamma$: \\
\\
\\
{\it Substep 11.1}\\ 
$$s_{i, (p_1,\cdots, p_l)}:=[((g_i)_{\gamma, (p_1,\cdots, p_l)}-M_1)\times 2^n]+1; ~1\leq i\leq m.$$
{\it Substep 11.2}\\
$$t_{i, (p_1,\cdots, p_l)}:=[((f_i)_{\gamma, (p_1,\cdots, p_l)}-M_1)\times 2^n]+1; ~1\leq i\leq m.$$ 
{\it Substep 11.3}\\
\begin{align*}
&K_n^{\gamma} \left(x_1^{\gamma},x_2^{\gamma}, \cdots , x_{(M_2-M_1) \times 
 2^n}^{\gamma} , p_1,\cdots , p_l \right) :=\\
&\sum_{i=1}^{m}\Big\vert \min \lbrace (M_1 + \dfrac{s}{2^n}-( g_i )_{\gamma ,(p_1,\cdots , 
p_l)})^2+x_s^{\gamma} -(( f_i )_{\gamma , (p_1,\cdots , p_l)}-( g_i )_{\gamma 
,(p_1,\cdots , p_l)})^2 -x_{t_{i,(p_1,\cdots , p_l)}}^{\gamma} \big | \\
&~~~~~~1\leq s \leq {s_{i, (p_1,\cdots , p_l)}}-1 \rbrace \cup \lbrace 
x_{s_{i ,(p_1,\cdots , p_l)}}^{\gamma}-(( f_i )_{\gamma ,(p_1,\cdots , p_l)}-( g_i 
)_{\gamma ,(p_1,\cdots , p_l)})^2 -x_{t_{i ,(p_1,\cdots , p_l)}}^{\gamma} \rbrace \\
&~~~~~~\cup \lbrace (M_1 + \dfrac{s-1}{2^n}+\epsilon -( g_i )_{\gamma ,(p_1,\cdots 
, p_l)})^2 +x_s^{\gamma} -(( f_i )_{\gamma ,(p_1,\cdots , p_l)}-( g_i 
)_{\gamma 
,(p_1,\cdots , p_l)})^2 -x_{t_{i ,(p_1,\cdots , p_l)}}^{\gamma}\big |\\
&~~~~~~s_{i, (p_1,\cdots , p_l)}+1\leq s \leq (M_2-M_1) \times 2^n \rbrace 
\Big\vert^2.\\
\end{align*}
\paragraph{Step 12:}
determine a solution of the following system and call it\\ 
$(\mathcal{C}_1^{\gamma_1},\cdots , \mathcal{C}_{(M_2-M_1)\times 2^n}^{\gamma_1}, \cdots, \mathcal{C}_1^{\gamma_{\vert \Gamma \vert}},\cdots , \mathcal{C}_{(M_2-M_1)\times 2^n}^{\gamma_{\vert \Gamma \vert}}, p_1^\circ,\cdots, p_l^\circ)$:
$$\min ~\sum_{\gamma \in \Gamma} K_n^\gamma(x_1^\gamma,\cdots ,x_{(M_2-M_1)\times 2^n}^\gamma, p_1,\cdots, p_l)$$
such that:
\begin{align*}
& |x_r^\gamma-x_t^\gamma|+x_r^\gamma+x_t^\gamma-2x_s^\gamma\geq 0
\quad \text{for} \quad (r,s,t)\in\{(r,s,t)|~r,s,t\in S_n,~r<s<t\},~\gamma \in \Gamma,\\
&
1/2\sum_{i=1}^l p_i=1,\\
&
\vert 1/2 \sum_{i=1}^l p_i z^{k_i} \vert^2+\vert 1/2 \sum_{i=1}^l p_i (-z)^{k_i} \vert^2=1~~for~~\vert z \vert=1,\\
&\vert 1/2 \sum_{i=1}^l p_i e^{itk_i} \vert >0~~for~~\vert t \vert\leq \pi/2.
\end{align*}

\paragraph{Step 13:}
$$\psi_\gamma(x):=\sum_{t=1}^{(M_2-M_1)\times 2^n}
\mathcal{C}_t^\gamma\chi_{A_t}(x);~\gamma \in \Gamma.$$
\paragraph{Step 14:}
$$\psi(x):=\sum_{\gamma \in \Gamma} \psi_{\gamma}(x).$$
\paragraph{Step 15:}
give $\psi$ and $ (p_1^\circ,\cdots, p_l^\circ) $ as output.
\begin{itemize}
\item[ ]
{{\small\textit{
``$\psi$ is the optimal regularization and $ p_1^\circ,\cdots, p_l^\circ $ are scaling coefficients corresponding to the optimal multi-resolution analysis."
}}}
\end{itemize}
\subsection{An algorithm for finding the optimal regularization by the approach presented in Section 4}

In this subsection, we present the steps of an algorithm that is associated with Theorem 4.3. Subsequently, the changes that are needed for having a better result regarding Remark 4.2 will be provided. Finally, because this algorithm likely needs a large volume of computational tasks, we propose an idea of creating a simple training set for testing the algorithm. 

\paragraph{Step 1.}
get $n,m,\epsilon$ from input. 
\begin{itemize}
\item[ ]
{\small{\textit{
``$\epsilon$ should be the minimum value that can be 
defined for the computer."}}}
\end{itemize}
\paragraph{Step 2.}
get $(f_i,g_i)$ for $i=1,\cdots , m$ from input.
\begin{itemize}
\item[ ]
{\small{\textit{
``$g_i$'s are noisy signals and $f_i$'s are the original signals i.e. 
$g_i=f_i+e_i$."
}}}
\end{itemize}
\paragraph{Step 3.}
get $g$ from input.
\begin{itemize}
\item[ ]
{\small{\textit{
``$g$ is the noisy signal that its noise is supposed to be removed by the model 
that will be obtained from the training set $\{(f_i,g_i)|~i=1,\cdots , m\}$."
}}}
\end{itemize}
\paragraph{Step 4:}
get the scaling function $\phi$ from input.
\begin{itemize}
\item[ ]
{\small{\textit{
``By the scaling function $\phi$, the associated multi-resolution analysis $\lbrace V_j;~ j\in\mathbb{Z} \rbrace$ is clarified:
$$V_j=\langle \lbrace 2^{j/2}\phi(2^jx-k); ~k\in\mathbb{Z} \rbrace \rangle.$$
However, the exact definition of it is not needed, because only the associated scaling coefficients are used in the decomposition algorithm and reconstruction algorithm."
}}}
\end{itemize}
\paragraph{Step 5:}
get $j_\circ$ from input.
\begin{itemize}
\item[ ]
{\small{\textit{
``$V_{j_\circ}$ is the space on which $ f_i $'s, $ g_i $'s and $ g $ are supposed to be projected."
}}}
\end{itemize}
\paragraph{Step 6:}
\begin{align*}
&a_{k, i}:=\int_{-\infty}^{+\infty} f_i(x)\overline{\phi(2^{j_\circ}x-k)} dx;\quad k\in {\mathbb{Z}},~~i=1,\cdots,m.\\
&
b_{k, i}:=\int_{-\infty}^{+\infty} g_i(x)\overline{\phi(2^{j_\circ}x-k)} dx;\quad k\in {\mathbb{Z}},~~i=1,\cdots,m.\\
&
b_{k}:=\int_{-\infty}^{+\infty} g(x)\overline{\phi(2^{j_\circ}x-k)} dx;\quad k\in {\mathbb{Z}}.\\
\end{align*}

\begin{itemize}
\item[ ]
{\small{\textit{
``Note that we have:
\begin{align*}
&P_{j_\circ}f_i=\sum_{k\in {\mathbb{Z}}} a_{k, i}\phi(2^{j_\circ}x-k) ;\quad i=1,\cdots,m,\\
&
P_{j_\circ}g_i=\sum_{k\in {\mathbb{Z}}} b_{k, i}\phi(2^{j_\circ}x-k) ;\quad i=1,\cdots,m,\\
&
P_{j_\circ}g=\sum_{k\in {\mathbb{Z}}} b_{k}\phi(2^{j_\circ}x-k) 
\end{align*}
where $P_{j_\circ}f_i$, $P_{j_\circ}g_i$ and $P_{j_\circ}g$ are the orthogonal projections of $f_i$, $g_i$ and $g$ onto $V_{j_\circ}$."
}}}
\end{itemize}

\paragraph{Step 7:}
determine sets $\lbrace(f_i)_\gamma\vert~\gamma \in {\Gamma}\rbrace$, $\lbrace(g_i)_\gamma \vert ~ \gamma \in {\Gamma}\rbrace$, $\lbrace g_\gamma \vert ~ \gamma \in {\Gamma}\rbrace$ by decomposing $(a_{k, i})_{k \in \mathbb{Z}}$, $(b_{k, i})_{k \in \mathbb{Z}}$, $(b_{i})_{k \in \mathbb{Z}}$ respectively by the iterative process for $i=1,\cdots,m$.
\begin{itemize}
\item[ ]
{\small{\textit{
``This step is the decomposition algorithm by using the scaling coefficients."
}}}
\end{itemize}

\paragraph{Step 8:}
determine integers $M_1$ and $M_2$ such that $\forall \gamma \in \Gamma,~
\forall i; ~ i=1,\cdots, m:$
$$(f_i)_\gamma, (g_i)_\gamma \in (M_1,M_2].$$
\paragraph{Step 9:}
$S_n:=\{1,\cdots, (M_2-M_1)\times 2^n\}$.\\
\paragraph{Step 10:}
$A_t:=(M_1+\frac{t-1}{2^n},M_1+\frac{t}{2^n}]; ~1\leq t\leq (M_2-M_1)\times 
2^n$.\\
\paragraph{Step 11:}
Perform subsetps 11.1 to 11.8 for every $\gamma\in \Gamma$: \\
\\
\\
{\it Substep 11.1}\\ 
$$s_i:=[((g_i)_\gamma-M_1)\times 2^n]+1; ~1\leq i\leq m.$$
{\it Substep 11.2} 
\begin{align*}
W:=\Big \{&(r_1,\cdots , r_m)\in 
\overbrace{S_n\times \cdots \times S_n}^{m}|\\
&\exists (x_1,\cdots , x_{(M_2-M_1)\times 2^n})
\in \mathbb{R}^{(M_2-M_1)\times 2^n};\\
\\
& 
\forall i\in \{1,\cdots ,m\}, ~\forall s\in\{1,\cdots, (M_2-M_1)\times 2^n\}
\backslash \{r_i\}:\\
& \Biggl( \left[ 
M_1+\frac{2r_i+sign(s_i-r_i)-1}{2^{n+1}}+
(\frac{1-sign(s_i-r_i)}{2})\epsilon\right]\\
& (1-\chi_{\{s_i\}}(r_i)) + 
(g_i)_\gamma(\chi_{\{s_i\}}(r_i))-(g_i)_\gamma\Biggr)^2+x_{r_i}\\
& \leq 
\Biggl(\left[M_1+\frac{2s+sign(s_i-s)-1}{2^{n+1}}+
(\frac{1-sign(s_i-s)}{2})\epsilon\right]\\
& (1-\chi_{\{s_i\}}(s))+(g_i)_\gamma
(\chi_{\{s_i\}}(s))-(g_i)_\gamma\Biggr)^2+x_s,\\
\\
& \forall (r,s,t)\in\{(r,s,t)| ~r,s,t\in S_n,~ r<s<t\}:\\
&|x_r-x_t|+x_r+x_t-2x_s\geq 0 \Big \}.
\end{align*}\\
{\it Substep 11.3}
\begin{align*}
L(r_1,\cdots,r_m):= \sum_{i=1}^m &|(f_i)_\gamma -
\Biggl(
[M_1+\frac{2r_i+sign(s_i-r_i)-1}{2^{n+1}}+
(\frac{1-sign(s_i-r_i)}{2})\epsilon]\\
& (1-\chi_{\{s_i\}}(r_i)) + 
(g_i)_\gamma(\chi_{\{s_i\}}(r_i))\Biggr)|^2.
\end{align*}\\
{{\it Subsetp 11.4}}~~\\
$$(r_1^\gamma, \cdots , r_m^\gamma):=\text{arg-min}_{(r_1,\cdots , r_m)\in W}
L(r_1,\cdots , r_m).$$\\
{{\it Subsetp 11.5}}~~
determine a solution of the following system and call it 
$(\mathcal{C}_1,\cdots , \mathcal{C}_{(M_2-M_1)\times 2^n})$:
\begin{equation*}
\begin{cases}
\Biggl( \left[ 
M_1+\frac{2r_i^\gamma+sign(s_i-r_i^\gamma)-1}{2^{n+1}}+
(\frac{1-sign(s_i-r_i^\gamma)}{2})\epsilon\right]\\
(1-\chi_{\{s_i\}}(r_i^\gamma))
+ 
(g_i)_\gamma(\chi_{\{s_i\}}(r_i^\gamma))-(g_i)_\gamma\Biggr)^2\\
+x_{r_i^\gamma}\leq 
\Biggl(\left[M_1+\frac{2s+sign(s_i-s)-1}{2^{n+1}}+
(\frac{1-sign(s_i-s)}{2})\epsilon\right]\\
 (1-\chi_{\{s_i\}}(s)
)+(g_i)_\gamma(\chi_{\{s_i\}}(s))-(g_i)_\gamma\Biggr)^2+x_s\\
 \text{for}~ i\in \{1,\cdots ,m\},~ s\in 
\{1,\cdots ,(M_2-M_1)\times 2^n\}\backslash \{r_i^\gamma\},\\
\\
|x_r-x_t|+x_r+x_t-2x_s\geq 0\\
 \text{for}~ (r,s,t)\in\{(r,s,t)|~r,s,t\in S_n,~ r<s<t\}.
\end{cases}
\end{equation*}\\

\noindent
{\it Substep 11.6}\\
$$\psi(x):=\sum_{t=1}^{(M_2-M_1)\times 2^n}
\mathcal{C}_t\chi_{A_t}(x).$$
{\it Substep 11.7}\\
$$\Phi(x)=(x-g_\gamma)^2+\psi(x).$$\\
\\
{\it Substep 11.8}\\ 
$$f_\gamma:=\text{arg-min}_{x\in (M_1,M_2]}\Phi(x).$$\\
\paragraph{Step 12:}
determine $ (a_k)_{k \in \mathbb{Z}} $ from $ \lbrace f_{\gamma}\vert~\gamma \in \Gamma\rbrace $ by the reconstruction algorithm. 
\begin{itemize}
\item[ ]
{{\small\textit{
``This step is the reconstruction algorithm by using the scaling coefficients."
}}}
\end{itemize}
\paragraph{Step 13:}
\begin{align*}
&f:=\sum_{k\in {\mathbb{Z}}} a_{k}\phi(2^{j_\circ}x-k). 
\end{align*}

\paragraph{Step 14:}
give $f$ as output.
\begin{itemize}
\item[ ]
{{\small\textit{
``$f$ is a modified signal for the signal $g$ that was taken at the step 3."
}}}
\end{itemize}

Considering Remark 4.2, for a better result, the substep 11.5 should be replaced by the following substeps:\\ 
\\
{\it Substep (i)}

$$t_i:=[((f_i)_\gamma-M_1)\times 2^n]+1; ~1\leq i\leq m.$$
{\it Subsetp (ii)}

\begin{align*}
K_n \left(x_1,x_2, \cdots , x_{(M_2-M_1) \times 2^n} \right) 
:=\sum_{i=1}^m & \Big\vert \min \lbrace (M_1 + \dfrac{s}{2^n}-( g_i )_{\gamma})^2+x_s -(( f_i 
)_{\gamma}-( g_i )_{\gamma})^2 -x_{t_i} \big | \\
&~1\leq s \leq {s_i}-1 \rbrace \cup \lbrace x_{s_i}-(( f_i 
)_{\gamma}-( g_i )_{\gamma})^2 -x_{t_i} \rbrace \\
&\cup \lbrace (M_1 + \dfrac{s-1}{2^n}+\epsilon -( g_i )_{\gamma})^2 +x_s -(( 
f_i )_{\gamma}-( g_i )_{\gamma})^2 -x_{t_i}\big |\\
&{s_i}+1\leq s \leq (M_2-M_1) \times 2^n \rbrace \Big\vert^2.\\
\end{align*}
{{\it Subsetp (iii)}}~~
determine a solution of the following system and call it 
$(\mathcal{C}_1,\cdots , \mathcal{C}_{(M_2-M_1)\times 2^n})$:
$$\min ~ K_n(x_1,\cdots ,x_{(M_2-M_1)\times 2^n})$$
such that:
\begin{align*}
&\Biggl( \left[ 
M_1+\frac{2r_i^\gamma+sign(s_i-r_i^\gamma)-1}{2^{n+1}}+
(\frac{1-sign(s_i-r_i^\gamma)}{2})\epsilon\right]\\
&(1-\chi_{\{s_i\}}(r_i^\gamma))
+ 
(g_i)_\gamma(\chi_{\{s_i\}}(r_i^\gamma))-(g_i)_\gamma\Biggr)^2\\
&+x_{r_i^\gamma}\leq 
\Biggl(\left[M_1+\frac{2s+sign(s_i-s)-1}{2^{n+1}}+
(\frac{1-sign(s_i-s)}{2})\epsilon\right]\\
& (1-\chi_{\{s_i\}}(s)
)+(g_i)_\gamma(\chi_{\{s_i\}}(s))-(g_i)_\gamma\Biggr)^2+x_s\\
& \text{for}~ i\in \{1,\cdots ,m\},~ s\in 
\{1,\cdots ,(M_2-M_1)\times 2^n\}\backslash \{r_i^\gamma\},\\
\\
& |x_r-x_t|+x_r+x_t-2x_s\geq 0\\
&\text{for}~ (r,s,t)\in\{(r,s,t)|~r,s,t\in S_n, ~r<s<t\}.\\
\end{align*}

At the end of this section, we suggest creating a simple training set that can be written with a simple computer program in order to have less computation volume for testing the above algorithm: Suppose $f_i$'s are images containing such two (medium size) squares with uniform level of darkness that their sizes and positions in the images are randomly selected and $g_i$'s are images obtained by adding many tiny squares (that both their number and locations in the image are randomly selected) to these original images i.e. $f_i$'s. Here the tiny squares play the role of noises for the original images. Once the denoising model is created, we can evaluate the ability of the algorithm by creating another simple original image $f$, and an image $g$ that is the result of adding the randomly selected noises to $f$ as explained.  The difference between the denoised version of $g$ that comes out of the algorithm and $f$ could show the algorithm's performance. It is natural to expect this gap to decrease as the members' number of the training set increases. In fact, not only can creating a simple training set reduce the computation volume for testing the algorithm because of the simplicity of the images' structure but also it has the advantage that we can increase the number of training sets' elements as many as needed and easily test the algorithm with different training sets sizes. \\
\\
By considering different levels of darkness for tiny squares and limiting their distributions, we can simulate different noise models such as Gaussian noise, salt-pepper noise, and so on. Having a different training set could be utilized in testing the approach presented in Remark 4.5 – the approach by the capability to remove different noises. We mention that by algorithms presented in this subsection the algorithm for the approach in Remark 4.5 is achieved merely by little changes and adding some more details; therefore, it is omitted. 

\section{Signal processing by training ``linear inverse problems with multi-constraints"}
This section addresses the bi-level optimization problem (1.7) by considering $\mathbf{Z}=B_{W,P}$, $B^{\prime}_{W,P}$ or $B^{\prime\prime}_{W,p}$, where $B_{W,P}$ was introduced in Example 1.2 while $B^{\prime}_{W,P}$ and $B^{\prime\prime}_{W,p}$ will be later defined here. It will be discussed that in the applications that time for computational tasks is limited, considering these sets in the learning method will be beneficial. Moreover, in this section, we will present an iterative method for solving linear inverse problems with a constraint $\psi$ from the set $B^{\prime}_{W,P}$. As it will be explained, this iterative method will make both computational and theoretical views more clear.

However, first, we are going to explain the logic of how reached to the bi-level optimization problem (1.5). Assume that $p_i>1$ for $i=1,\dots,n$. We like $\lambda_1,\dots,\lambda_n$ to be in such a way that:  
$$f_i=\text{arg-min}_{f\in 
\mathcal{H}}\|f-g_i\|^2+
\sum_{j=1}^n\lambda_j|||f|||_{W_j,p_j}^{p_j};~~~~1 \leq i \leq m,$$ 
where $(f_i, g_i);~1 \leq i \leq m$ is the training set. By Lemma 1.1, we have: 
$$\sum_{\gamma\in\Gamma}F_{(\lambda_1,\cdots,\lambda_n)}^{-1}((g_i)_\gamma)\varphi_\gamma=f_i;~~~~1 \leq i \leq m,$$ 
where 
$$F_{(\lambda_1,\cdots,\lambda_n)}(x)=x+Sign(x)\frac{\sum_{j=1}^n p_j\lambda_jw_{j,\gamma}|x|^{p_j-1}}{2},$$ 
or equivalently 
$$F_{(\lambda_1,\cdots,\lambda_n)}^{-1}((g_i)_\gamma)=(f_i)_\gamma;~~1 \leq i \leq m,~~\gamma \in \Gamma$$ 
or 
$$F_{(\lambda_1,\cdots,\lambda_n)}((f_i)_\gamma)=(g_i)_\gamma;~~1 \leq i \leq m,~~\gamma \in \Gamma,$$ 
so 
$$(f_i)_\gamma+Sign((f_i)_\gamma)\frac{\sum_{j=1}^n p_j\lambda_jw_{j,\gamma}|(f_i)_\gamma|^{p_j-1}}{2}=(g_i)_\gamma;~~1 \leq i \leq m,~~\gamma \in \Gamma.$$ 
Now, we suppose $\Gamma$ is a finite set $(\Gamma=\lbrace\gamma_1,\dots,\gamma_t\rbrace)$ and pick $n$ in the way that $n \geqslant \vert \Gamma \vert \times m$. Then for achieving $\lambda_i$’s, we should solve the following system of $\vert \Gamma \vert \times m$ linear equations in $n$ unknowns: 
$$\begin{cases}
2(f_1)_{\gamma_1}+Sign((f_1)_{\gamma_1})\sum_{j=1}^n p_j\lambda_jw_{j,\gamma_1}|(f_1)_{\gamma_1}|^{p_j-1}=2(g_1)_{\gamma_1}\\
\\
2(f_1)_{\gamma_2}+Sign((f_1)_{\gamma_2})\sum_{j=1}^n p_j\lambda_jw_{j,\gamma_2}|{(f_1)_{\gamma_2}|}^{p_j-1}=2(g_1)_{\gamma_2}\\
\\
\vdots~~~~~~~~~~~~~~~~\vdots~~~~~~~~~~~~~~~~~~~\vdots~~~~~~~~~~~~~~~~~~~~~~~~~~~~\vdots \\
\\
2(f_m)_{\gamma_t}+Sign((f_m)_{\gamma_t})\sum_{j=1}^n p_j\lambda_jw_{j,\gamma_t}|(f_m)_{\gamma_t}|^{p_j-1}=2(g_m)_{\gamma_t}
\end{cases}$$
Even if $\Gamma$ be a single-member set, the system has solutions only under limited conditions on the training set, namely: 
$$\text{if}~~(f_{i_{k_1}})_\gamma<\dots<(f_{i_{k_m}})_\gamma~~\Rightarrow~~(g_{i_{k_1}})_\gamma<\dots<(g_{i_{k_m}})_\gamma$$ 
because $\sum_{j=1}^n \lambda_jw_{j,\gamma}|x|^{p_j}$ is convex. Therefore, we turn to the “loss function” method, namely the bilevel optimization problem (1.5) which was mentioned in the Introduction. 
Moreover, note that solving the bilevel optimization problem (1.5) is less challenging when $p_i>1~ \text{for} ~i=1,...,n$, because with these conditions the following equality will be established:
$$\sum_{\gamma\in\Gamma}S_ {(\lambda_1w_{1,\gamma},\cdots , \lambda_nw_{n,\gamma}),(p_1,\cdots 
p_n)}((g_i)_\gamma)\varphi_\gamma=\text{arg-min}_{f\in 
\mathcal{H}}\|f-g_i\|^2+
\sum_{j=1}^n\lambda_j|||f|||_{W_j,p_j}^{p_j}$$
where 
\begin{align*}
&S_ {(\lambda_1w_{1,\gamma},\cdots , \lambda_nw_{n,\gamma}),(p_1,\cdots 
p_n)}(x)=F^{-1}_{(\lambda_1,\cdots,\lambda_n)}(x),\\
\\
&F_{(\lambda_1,\cdots,\lambda_n)}(x)=x+Sign(x)\frac{\sum_{i=1}^n p_i\lambda_iw_{i,\gamma}|x|^{p_i-1}}{2}.
\end{align*}
Consequently
\begin{align*}
\mathfrak{I}(\lambda_1,\cdots,\lambda_n)&=\sum_{i=1}^m\|\tilde{f}_{i,(\lambda_1,\cdots,
\lambda_n)}-f_i\|^2\\
&=\sum_{i=1}^m\|\sum_{\gamma\in\Gamma}F_{(\lambda_1,\cdots,\lambda_n)}^{-1}((g_i)_\gamma)\varphi_\gamma-f_i\|^2\\
&=\sum_{i=1}^m\|\sum_{\gamma\in\Gamma}(F_{(\lambda_1,\cdots,\lambda_n)}^{-1}((g_i)_\gamma)-(f_i)_\gamma)\varphi_\gamma\|^2\\
&=\sum_{i=1}^m\sum_{\gamma\in\Gamma}|F_{(\lambda_1,\cdots,\lambda_n)}^{-1}((g_i)_\gamma)-(f_i)_\gamma|^2\\
\end{align*}
which is a differentiable function from $\mathbb{R}^n$ to $\mathbb{R}$. By a similar process, it can also be presented that the bilevel optimization problem (2.10) will have a simpler solution, however somehow challenging, provided that: (1) $B_{W,P}$ is replaced with \textbf{C}  and (2)  \textbf{U} is a set of all linear combinations of a finite number of linear operators $\lbrace K_1, K_2,\cdots, K_l \rbrace$ as follows:\\
$$ \textbf{U}=\lbrace \alpha_1 K_1 +\cdots+ \alpha_l K_l\vert~\alpha_1,\cdots,\alpha_l \in \mathbb{R} \rbrace $$
By these assumptions, the minimizer of the following function from $\mathbb{R}^{n+l}$ to $\mathbb{R}$ should be achieved: \\
\begin{align*}
&\mathfrak{I}:\mathbb{R}^{n+l}\longrightarrow \mathbb{R}\\
& 
\mathfrak{I}(\lambda_1,\cdots,\lambda_n, \alpha_1,\cdots,\alpha_l )=\sum_{i=1}^m\|\tilde{f}_{i,(\lambda_1,\cdots,
\lambda_n), (\alpha_1,\cdots,\alpha_l)}-f_i\|^2\\
&
;\tilde{f}_{i,(\lambda_1,\cdots ,\lambda_n), (\alpha_1,\cdots,\alpha_l)}=\text{arg-min}_{f\in 
\mathcal{H}}\|(\alpha_1 K_1 +\cdots+ \alpha_l K_l)(f)-g_i\|^2+
\sum_{j=1}^n\lambda_j|||f|||_{W_j,p_j}^{p_j}.
\end{align*}
However, there are some challenges in the above problem, one of which is obtaining $ \tilde{f}_{i,(\lambda_1,\cdots ,\lambda_n), (\alpha_1,\cdots,\alpha_l)} $. Since $K$ is in its general form (not merely $K=I$), an iterative process is needed for reaching $ \tilde{f}_{i,(\lambda_1,\cdots ,\lambda_n), (\alpha_1,\cdots,\alpha_l)} $. To clarify how this situation can be dealt with, the following simplified problem will be investigated:\\
\begin{align*}
&\mathfrak{I}:\mathbb{R}\longrightarrow \mathbb{R}\\
& \mathfrak{I}(\lambda)=
\sum_{i=1}^m\|\tilde{f}_{i,\lambda}-f_i\|^2\\
&
;\tilde{f}_{i,\lambda}=\text{arg-min}_{f\in 
 \mathcal{H}}\|Kf-g_i\|^2+\lambda|||f|||_{W, p}^{p}
\tag*{$(6\cdot 1)$}
\end{align*}
where $ |||f|||_{W, p}^{p} $ defined in (1.3). First, we define the function $ \textbf{T}_{i, \lambda} $ as follows:
$$ \textbf{T}_{i, \lambda}=\textbf{S}_{\lambda W, p}(h)=\sum_{\gamma \in \Gamma}S_{\lambda w_{\gamma}, p}(h_{\gamma}+[K^*(g_i-K(h))]_{\gamma})\varphi_{\gamma}$$ 
where $ S_{\lambda w_{\gamma}, p} $ for $ \gamma \in \Gamma $ is a function from $\mathbb{R}$ to itself which will be later defined in Lemma 6.1. As shown in [17], we have 
$$ \Vert \textbf{T}_{i, \lambda}^N(h_\circ)-\tilde{f}_{i,\lambda}\Vert \longrightarrow 0~~when~~N\longrightarrow 0$$
where $ h_\circ \in \mathcal{H} $ is arbitrarily selected. Consequently, $ \textbf{T}_{i, \lambda}^{N_\circ}(h_\circ) $ will be a highly accurate approximation for $ \tilde{f}_{i,\lambda} $ providing that $ N_\circ $ is large enough. Without loss of generality, we suppose $ N_\circ=2 $ i.e. $ \tilde{f}_{i,\lambda}\simeq \textbf{T}_{i, \lambda}^{2}(h_\circ) $. Then, by assuming $ h_{1,\lambda}^i=\textbf{T}_{i, \lambda}(h_\circ) $ for $ i=1,\cdots,m $, we have
\begin{align*}
\mathfrak{I}(\lambda)&=\sum_{i=1}^m\|\textbf{T}_{i, \lambda}(h_{1,\lambda}^i)-f_i\|^2\\
&=\sum_{i=1}^m\|\textbf{S}_{\lambda W, p}(h_{1,\lambda}^i)-f_i\|^2\\
&=\sum_{i=1}^m\|\sum_{\gamma\in\Gamma}S_{\lambda w_{\gamma}, p}((h_{1,\lambda}^i)_{\gamma}+[K^*(g_i-K(h_{1,\lambda}^i))]_{\gamma})\varphi_{\gamma}-f_i\|^2\\
&=\sum_{i=1}^m\|\sum_{\gamma\in\Gamma}(S_{\lambda w_{\gamma}, p}((h_{1,\lambda}^i)_{\gamma}+[K^*(g_i-K(h_{1,\lambda}^i))]_{\gamma})-(f_i)_\gamma)\varphi_\gamma\|^2\\
&=\sum_{i=1}^m\sum_{\gamma\in\Gamma}|S_{\lambda w_{\gamma}, p}((h_{1,\lambda}^i)_{\gamma}+[K^*(g_i-K(h_{1,\lambda}^i))]_{\gamma})-(f_i)_\gamma|^2.\\
\end{align*}
Now, by defining $\mathfrak{I}_{\gamma, i}$ as follows:
$$ \mathfrak{I}_{\gamma, i}(\lambda)=|S_{\lambda w_{\gamma}, p}((h_{1,\lambda}^i)_{\gamma}+[K^*(g_i-K(h_{1,\lambda}^i))]_{\gamma})-(f_i)_\gamma|^2 $$
we will have 
$$ \frac {\partial \mathfrak{I}}{\partial \lambda}=\sum_{i=1}^m\sum_{\gamma\in\Gamma} \frac {\partial {\mathfrak{I}_{\gamma, i}}}{\partial \lambda}. $$
In fact, the challenge converted to finding the derivative of $\mathfrak{I}_{\gamma, i}$ for $\gamma \in \Gamma$ and $i=1,\cdots,m$, which are differentiable functions from $\mathbb{R}$ to $\mathbb{R}$.\\

The problem (6.1) introduced above can be obtained by replacing the set \textbf{C} in (2.9) with the set $B^{\prime\prime}_{W,p}$ defined as follows: \\
$$B^{\prime\prime}_{W,p}=\{\psi:\mathcal{H}\longrightarrow
\mathbb{R}|~\psi(f)=\lambda|||f|||_{W,p}^{p}; ~\lambda\in \mathbb{R}\}$$
where $
|||f|||_{W,p}=\Biggl(\sum_{\gamma\in\Gamma} w_\gamma|\langle 
f,\varphi_\gamma\rangle|^p\Biggr)^{\frac{1}{p}}$
for $1\leq p\leq 2$, and $W=(w_\gamma)_{\gamma\in 
\Gamma}$ is a sequence of strictly positive weights.\\ 
The set $B^{\prime}_{W,P}$ consists of all constraints named ``mixed multi-constraints" in [30] is another alternative for \textbf{Z} in (1.7) and \textbf{C} in (2.9) and (2.10): \\
$$B^{\prime}_{W,P}=\{\psi:\mathcal{H}\longrightarrow
\mathbb{R}|~\psi(f)=\sum_{i=1}^n\lambda_i|\!|\!|f|\!|\!|^{p_i}_{W_i,p_i}; ~\lambda_i\in \mathbb{R}^+ ~ \text{for}~ i\in \{1,\cdots , n\}\}$$
where the sequence $W=(w_\ga)_{\ga\in \Ga}$ is uniformly bounded below away 
from zero, i.e. there exists a constant $c>0$ such that $\forall \ga\in\Ga;~ 
w_\ga\geq c$,
$\Ga=\Ga_1\cup \Ga_2\cup \dots\cup \Ga_{n}, W_i =\{w_\ga\}_{\ga\in\Ga_i},$ 
$P=\{p_1,p_2,\dots,p_n\}$ and  
$|\!|\!|f|\!|\!|^{p_i}_{W_i,p_i}=\sum_{\ga\in\Ga_i} w_\ga|f_\ga|^{p_i}$ for every $i; ~1\leq i\leq n$.

Besides, the following inclusion relations among $B_{W,P}$, $B^{\prime}_{W,P}$, $B^{\prime\prime}_{W,p}$ and \textbf{C} are obvious: 
$$B^{\prime\prime}_{W,p}\subseteq B_{W,P} \subseteq \textbf{C},$$
$$B^{\prime\prime}_{W,P}\subseteq B^{\prime}_{W,P} \subseteq \textbf{C}.$$

As already discussed in the introduction, the larger \textbf{Z} in (1.7), the more appropriate regularization term. However, utilizing the bilevel optimization problem (1.7) by considering $\textbf{Z}=B_{W,P}, B^{\prime}_{W,P} ~\text{or}~B^{\prime\prime}_{W,p}$ instead of $\textbf{Z}=\textbf{C}$ is beneficial in applications where getting the model faster is indispensable or where a very powerful processor (a processor powerful enough for obtaining a $\psi$ from \textbf{C}) is not available. The reason is, the elements of $B_{W,P}, B^{\prime}_{W,P} ~\text{or}~B^{\prime\prime}_{W,p}$  are linear combinations of norms that each of them is an approximation for a suitable smoothness space norm (such as Besov or Sobolev space) and consequently finding merely parameter $\lambda_i$'s is needed in the learning process which requires less computation compared with extremely high computational volume - in addition to having a great data set - for achieving a function $\psi$ from \textbf{C}. For instance, in cell phone application which will be addressed in Section 7, the learning process for reaching the model should be done - in a very short period before the conversation begins - by a cell phone processor that it may not be powerful enough to handle the difficult task of obtaining a sufficiently detailed $\psi$ from \textbf{C}. However, in satellite image processing's application in Section 7 utilizing \textbf{C} instead of $B_{W,P}, B^{\prime}_{W,P} ~\text{or}~B^{\prime\prime}_{W,p}$ is preferred, because once the image has reached the earth, image processing can be done at the right and with sufficient time with a powerful processor.

For completing the discussion of time or volume of computational tasks, in addition to the above discussion that was around the time it takes to obtain a regularization term $\psi$, the time required to solve the minimization problem (1.6) by considering the different regularization term from different sets ($B_{W,P}$, $B^{\prime}_{W,P}$, $B^{\prime\prime}_{W,p}~\text{and}~\textbf{C}$) should be compared. Namely, the time required to solve the following minimization problems should be compared: 
\begin{align*}
\tag*{$(6\cdot 2)$}
&\text{inf}_{f 
 \in \mathcal{H}} 
\Vert K(f)- g \Vert ^2 +\psi(f);~~\psi \in B^{\prime\prime}_{W,p},\\
\tag*{$(6\cdot 3)$}
&\text{inf}_{f 
 \in \mathcal{H}} 
\Vert K(f)- g \Vert ^2 +\psi(f);~~\psi \in B^{\prime}_{W,P},\\
\tag*{$(6\cdot 4)$}
&
\text{inf}_{f 
 \in \mathcal{H}} 
\Vert K(f)- g \Vert ^2 +\psi(f);~~\psi \in B_{W,P},\\
\tag*{$(6\cdot 5)$}
&
\text{inf}_{f 
 \in \mathcal{H}} 
\Vert K(f)- g \Vert ^2 +\psi(f);~~\psi \in \textbf{C}.
\end{align*}
As already mentioned the time needed for achieving a $\psi$ from $B_{W,P}, B^{\prime}_{W,P} ~\text{or}~B^{\prime\prime}_{W,p}$ is much less than the time for providing a $\psi$ from \textbf{C}. However, analyzing the amount of computational task required for the minimization problems corresponding to $B_{W,P}, B^{\prime}_{W,P} ~\text{and}~B^{\prime\prime}_{W,p}$ - namely the problems (6.2), (6.3) and (6.4) - is beneficial for applications in which fast computing is crucial. Therefore, we dedicate the rest of this section to present an independent solution for the problem (6.3), because it will make it clear that the volume of computing tasks needed for solving (6.2) and (6.3) are close to each other and far less than the amount required for (6.4). 

Since $B^{\prime\prime}_{W,p}\subseteq B_{W,P} $ and $B^{\prime\prime}_{W,P}\subseteq B^{\prime}_{W,P}$, the solution of problem (6.2) can be concluded from the solutions of both problems (6.3), (6.4) and consequently the iterative method for reaching to the minimizer will have less computational tasks than both (6.3), (6.4). Now the question that how much the difference is between the above-mentioned problems should be taken into consideration. Comparing the iterative processes for solving (6.2) and (6.4) in [17] and [30] respectively reveals that the distance between volumes of their computational tasks is significant. In the solution of (6.2) in [17], the function $ S_{ w_{\gamma}, p} $ in Lemma 6.1 repeats in all iteration steps, while for (6.4) the function $S_ {(w_{1,\gamma},\cdots ,w_{n,\gamma}),(p_1,\cdots 
 p_n)}$ in Lemma 2.1 in [30] appears. Since the computational complexity of finding the minimizer of the latter function is far more than the former one, the volume of computation for solving (6.4) will be far more than (6.2). However, it is expected that the volume of the computational tasks for solving (6.2) and (6.3) will be close to each other because by providing an independent iterative method for solving (6.3) in the following, we will show the function $ S_{ w_{\gamma}, p} $ in Lemma 6.1 has a key role in iterative process same as solving (6.2) in [17]:

Because of the similarity of proofs of Proposition 6.2, Lemma 6.7 and Proposition 6.18 with what was provided for solving (6.2) in their corresponding lemma and propositions in [17], we have skipped them. So in continuation, first, we begin with introducing the function $ S_{ w, p} $ by the following lemma proved in [17]:
\begin{lemma}\label{6.1} 
The minimizer of the function $M(x)=x^2-2bx+c|x|^p$ for $p\geq 1$, $c>0$ is 
$S_{c,p}(b)$, where
the function $S_{c,p}$ from $\mathbb{R}$ to itself is defined by 
$$S_{c,p}(t)=\begin{cases}
F^{-1}_{c,p}(t) & p>1\\
t-\f{c}{2} & p=1,~t> \f{c}{2}\\
0 & p=1,~ |t|\leq \f{c}{2} \\
t+\f{c}{2} & p=1,~t< -\f{c}{2}
\end{cases}$$
where the function $F_{c,p}$ is defined by 
$$F_{c,p}(t)=t+\f{cp}{2} \Sign (t) |t|^{p-1}.$$
\end{lemma}

The iterative process will be based on the function $\bS_{W,P}$ from $\cH$ to $\cH$ which is defined in the following proposition:

\begin{proposition}\label{6.2}
Suppose $K:\cH\lo \cH'$ is an operator, with $\|KK^*\|<1$,
$(\var_\ga)_{\ga\in\Ga}$ is an orthonormal basis for $\cH$, and 
$W=(w_\ga)_{\ga\in\Ga}$ is a sequence such that $\forall \ga\in\Ga$ 
$w_\ga>c>0$. Further suppose $g$ is an element of $\cH'$. Let 
$\Ga=\Ga_1\cup\Ga_2\cup \dots\cup \Ga_n$, $W=W_1\cup W_2\cup\dots\cup W_n$, 
$P =\{p_1,\dots,p_n\}$ such that $p_i\geq 1$ for $1\leq i\leq n$. 
Choose $a\in 
\cH$ and define the functional $\Phi^{S\cup R}_{W, P}(f;a)$ on $\cH$ by 
$$\Phi^{S\cup R}_{W, P} (f;a)=\|Kf-g\|^2+ \sum_{i=1}^n \sum_{\ga\in\Ga_i} 
w_\ga |f_\ga|^{p_i}+ \|f-a\|^2 - \|K(f-a)\|^2.$$
Also, define the operators $\bS_{W, P}$ by 
$$\bS_{W, P}(h)=\sum_{i=1}^n \sum_{\ga\in\Ga_i}S_{w_\ga,p_i}(h_\ga)\var_\ga$$
with functions $S_{w,p_i}$ from $\mathbb{R}$ to itself given by Lemma 6.1.

By these assumptions, we will have 

A) $f_{\min}$=minimizer of the functional $\Phi_{W, P}^{S\cup 
R}=\bS_{W, P}(a+K^*(g-Ka)).$

B) for all $h\in \cH$, one has
$$\Phi^{S\cup R}_{W, P}(f_{\min}+h;a) \geq \Phi^{S\cup 
R}_{W, P}(f_{\min};a)+\|h\|^2.$$
\end{proposition}
Below it will be introduced the sequence $\{f^n\}_{n=1}^{+\infty}$ whose convergence to a minimizer of (6.3) is supposed to be proven:
\begin{definition}\label{6.3}
Pick $f^0$ in $\cH$. We define the functions $f^n$
recursively by the following algorithm:
\begin{align*}
f^1 &=\text{arg-min}(\Phi^{S\cup R}_{W, P}(f;f^0)), \quad f^2=\text{arg-min}(\Phi^{S\cup
R}_{W, P}(f,f^1)),\dots,\\
f^n &=\text{arg-min}(\Phi^{S\cup R}_{W, P}(f;f^{n-1})).
\end{align*}
\end{definition}

\begin{corollary}\label{6.4}
By Definition 6.3, we clearly have
$$f^n=\bS_{W, P}(f^{n-1}+K^*(g-Kf^{n-1})).$$
\end{corollary}

\begin{definition}\label{6.5}
We define operator $\bT$ from $\cH$ to $\cH$ by 
$$\bT(f)=\bS_{W, P}(f+K^*(g-Kf)).$$
\end{definition}

\begin{corollary}\label{6.6}
By Definition 6.5, we have
$$\bT^n(f^0)=f^n.$$
\end{corollary}

The following lemma can be proved in the same way as its corresponding
lemma in [17]:
\begin{lemma}\label{6.7}
If $f^*\in \{f\in\cH|~\bT(f)=f\}$ then $f^*$ is a minimizer of $\Phi_{W, P}$.
\end{lemma}

By the following lemma, we state characteristics of the function $\mathbf{T}$ and the 
sequence $\{f^n\}_{n=1}^\infty$ which were introduced in Definition 6.3 and 
Definition 6.5.

\begin{lemma}\label{6.8}
Let $K$ be a bounded linear operator from $\cH$ to $\cH'$, with 
the norm strictly bounded by 1. 
Take $p_i\in [1,2]$ for $1\leq i\leq n$, and let $\bS_{W, P}$ be the operator 
defined by 
$$\bS_{W, P}(h)=\sum_{\ga\in\Ga} S_{w_\ga,p_\ga}(h_\ga)\var_\ga$$
where the sequence $W=(w_\ga)_{\ga\in \Ga}$ is uniformly bounded below away 
from zero, i.e. there exists a constant $c>0$ such that $\forall \ga\in\Ga;~ 
w_\ga\geq c$ and we have $p_i=p_\ga$ for every $\ga\in\Ga_i$. Let
$\Ga=\Ga_1\cup \Ga_2\cup \dots\cup \Ga_{n}, W_i =\{w_\ga\}_{\ga\in\Ga_i},$ 
$P=\{p_1,p_2,\dots,p_n\}$ and suppose 
$|\!|\!|f|\!|\!|^{p_i}_{W_i,p_i}=\sum_{\ga\in\Ga_i} w_\ga|f_\ga|^{p_i}$,
\begin{align*}
|\!|\!|f|\!|\!|^{P}_{W, P} &=\sum_{\ga\in\Ga} w_\ga |f_\ga|^{p_\ga} \\
&=\sum_{i=1}^n \sum_{\ga\in\Ga_i}w_\ga |f_\ga|^{p_i} =|\!|\!| 
f|\!|\!|^{p_1}_{W_1,p_1}+\dots+ |\!|\!|f|\!|\!|^{p_n}_{W_n,p_n}, 
\end{align*}
\begin{align*}
\Phi_{W, P}(f) &=\|Kf-g\|^2+|\!|\!|f|\!|\!|^{P}_{W, P},\\
\Phi^{S\cup R}_{W, P}(f)&=\Phi_{W, P}(f)+\|f-a\|^2-\|K(f-a)\|^2.
\end{align*}
Then 

A) $\forall v,v'\in \cH;~ \|\bS_{W, P}(v)-\bS_{W, P}(v')\|\leq \|v-v'\|$.

B) $\forall v,v'\in \cH;~ \|\bT(v)-\bT(v')\|\leq \|v-v'\|$.

C) Both $(\Phi_{W, P}(f^n))_{n\in\BN}$ and $(\Phi_{W, P}^{S\cup 
R}(f^{n+1};f^n))_{n\in\BN}$ are non-increasing sequences.

D) $\exists M>0;~ \forall n\in N;~ \|f^n\|<M$.

E) $\sum_{n=0}^\infty \|f^{n+1}-f^n\|^2<\infty$.

F) $\|\bT^{n+1}(f^0)-\bT^n(f^0)\|=\|f^{n+1}-f^n\|\lo 0 \quad \text{for} \quad n\lo\infty$.
\end{lemma}

\begin{proof} 
Parts A, B, C, E and F have similar proofs to their corresponding lemmas in [17]; thus,
we just prove part D, which has some differences in its proof: First, we have:
\begin{align*}
|\!|\!|f^n|\!|\!|^{p_i}_{W_i,p_i} \leq |\!|\!|f^n|\!|\!|^{P}_{W, P}& \leq 
\|K(f^n)-g\|^2+ 
|\!|\!|f^n|\!|\!|^{P}_{W, P}\\
&=\Phi_{W, P}(f^n)\leq \Phi_{W, P}(f^0).
\tag{$6.6$}
\end{align*}
Moreover, we have for every $\ga\in \Ga_i$:
\begin{align*}
w_\ga^{\f{2-p_i}{p_i}} |f_\ga|^{2-p_i} &=(w_\ga|f_\ga|^{p_i})^{\f{2-p_i}{p_i}} 
\leq \left( \sum_{\ga\in\Ga_i} w_\ga|f_\ga|^{p_i}\right)^{\f{2-p_i}{p_i}} = 
|\!|\!|f|\!|\!|^{2-p_i}_{W_i,p_i}
\end{align*}
or
\begin{align*}
\sup_{\ga\in\Ga_i} \left[ w_\ga^{\f{2-p_i}{p_i}} |f_\ga|^{2-p_i}\right] 
\leq |\!|\!|f|\!|\!|^{2-p_i}_{W_i,p_i} \tag{$6.7$}
\end{align*}
and for every $p_i\in [1,2]$:
\begin{align*}
\forall \ga\in \Ga_i \;\;\;
c\leq w_\ga & \Lo c^{\f{2}{p_i}} \leq w_\ga^{\f{2}{p_i}} 
=w_\ga w_\ga^{\f{2-p_i}{p_i}} \\
&\Lo \f{1}{w_\ga} \leq \f{w_\ga^{\f{2-p_i}{p_i}}}{c^{\f{2}{p_i}}} \tag{$6.8$}.
\end{align*}
We deduce from (6.7), (6.8) that
\begin{align*}
\sum_{\ga\in\Ga_i} |f_\ga|^2 &=\sum_{\ga\in\Ga_i} \f{1}{w_\ga} |f_\ga|^{2-p_i} 
w_\ga|f_\ga|^{p_i} \\
&\leq (\sup_{\ga\in\Ga_i} \f{1}{w_\ga} |f_\ga|^{2-p_i} )\sum_{\ga\in\Ga_i} 
w_\ga|f_\ga|^{p_i} \\
&\leq (\sup_{\ga\in\Ga_i} \f{w_\ga^{\f{2-p_i}{p_i}}}{c^{\f{2}{p_i}}} 
|f_\ga|^{2-p_i}) |\!|\!|f|\!|\!|^{p_i}_{W_i,p_i} \\
&\leq c^{-\f{2}{p_i}} |\!|\!|f|\!|\!|^{2-p_i}_{W_i,p_i} 
|\!|\!|f|\!|\!|^{p_i}_{W_i,p_i} 
= c^{-\f{2}{p_i}} |\!|\!|f|\!|\!|^2_{W_i,p_i}. 
\tag{$6.9$}
\end{align*}
Consequently, for every $i; ~1\leq i\leq n$, by (6.6), (6.9) we conclude
\begin{align*}
\sum_{\ga\in\Ga_i} |(f^n)_\ga|^2 & \leq c^{-\f{2}{p_i}} 
|\!|\!|f^n|\!|\!|^2_{W_i,p_i} \\
&\leq c^{-\f{2}{p_i}} \Phi_{W, P}(f^0)^{\f{2}{p_i}}:=M_i \\
\Lo~ \|f^n\|^2&=\sum_{\ga\in\Ga} |(f^n)_\ga|^2=\sum_{i=1}^n \sum_{\ga\in\Ga_i} 
|(f^n)_\ga|^2\leq \sum_{i=1}^n M_i\\
\Lo \exists M>&0;~ \forall n\in N:~ \|f^n\|\leq M.
\end{align*}
\end{proof}

We recall Lemma 6.9 and Lemma 6.10, from [40] and Appendix B in [17]
with a minor modification.

\begin{lemma}\label{6.9}
Suppose for mapping $\bA$ from $\cH$ to $\cH$ and $v_0\in\cH$ we have the
following conditions 
\begin{itemize}
\item[(i)] $\forall v,v'\in \cH$; $\|\bA(v)-\bA(v')\| \leq \|v-v'\|.$ 
\item[(ii)] $\|\bA^{n+1}(v_0)-\bA^n(v_0)\|\lo 0$ \quad as\quad $n\lo\infty.$
\item[(iii)] $\{v\in\cH|~\bA(v)=v\}\neq \phi$.
\end{itemize}
Then 
$$\exists v^*\in {\cH}; \quad \bA(v^*)=v^*,~A^n(v_0)\ovs{weakly}{\lo} v^*.$$
\end{lemma}

\begin{lemma}\label{6.10}
Suppose for mapping $\bA$ from $\cH$ to $\cH$ and $v_0\in\cH$ we have the
following conditions 
\begin{itemize}
\item[(i)] $\forall v,v'\in \cH$; $\|\bA(v)-\bA(v')\| \leq \|v-v'\|.$ 
\item[(ii)] $\|\bA^{n+1}(v_0)-\bA^n(v_0)\|\lo 0$ \quad as\quad $n\lo\infty.$
\item[(iii)] $\exists 
\{A^{n_k}(v_0)\}_{k=1}^{+\infty}\subs\{A^n(v_0)\}_{n=1}^{+\infty}$, $\exists 
v'\in{\cH}$; $A^{n_k}(v_0)\ovs{weakly}{\lo} v'.$
\end{itemize}
Then 
$$ A(v')= v'.$$
\end{lemma}
By Lemma 6.9 and Lemma 6.10, we clearly have the following 
corollary:

\begin{corollary}\label{6.11}
Suppose for mapping $\bA$ from $\cH$ to $\cH$ and $v_0\in\cH$ we have the
following conditions 
\begin{itemize}
\item[(i)] $\forall v,v'\in \cH$; $\|\bA(v)-\bA(v')\|\leq \|v-v'\|.$
\item[(ii)] $\|\bA^{n+1}(v_0)-\bA^n(v)\|\lo 0$\quad as\quad $n\lo\infty.$
\item[(iii)] $\exists \{\bA^{n_k}(v_0)\}_{k=1}^\infty \subs 
\{\bA^n(v_0)\}_{n=1}^{+\infty}$, $\exists v'\in\cH; ~\bA^{n_k}(v_0) 
\ovs{weakly}{\lo} v'$.
\end{itemize}
Then 
$$\exists v^*\in{\cH};\quad \bA(v^*)=v^*, ~\bA^n(v_0) \ovs{weakly}{\lo} v^*.$$
\end{corollary}
The following lemma which is a classical result from Functional Analysis (see Theorem 8.25 in [38], for example) has a key role to prove that $\{f^n\}_{n=1}^{+\infty}$ is weakly convergent to a minimizer of (6.3):\\
\begin{lemma}\label{6.12}
Every bounded sequence in a reflexive space has a weakly convergent 
subsequence.
\end{lemma}

To provide the weak convergence theorem, first, it should be presented the following lemma in which the weak convergence of $\{f^n\}_{n=1}^{+\infty}$ to a fixed point of the function $\mathbf{T}$ (introduced in Definition 6.5) is proven: 
\begin{lemma}\label{6.13}
There exists $f^*\in\cH$ such that $\bT(f^*)=f^*$, $\bT^n(f^0)\ovs{weakly}{\lo}
f^*$.
\end{lemma}

\begin{proof}
Parts B and F of Lemma 6.8
satisfy conditions (i), (ii) of Corollary 6.11. On the 
other hand, since $\cH$ is a reflexive space, Lemma 6.12
and part D of Lemma 6.8 
satisfy condition (iii) of Corollary 6.11.
Therefore, Lemma 6.13 is proved, i.e. 
there exists $f^*\in\cH$ such that $\bT(f^*)=f^*$, $f^n=\bT^n(f^0) 
\ovs{weakly}{\lo} f^*$.
\end{proof}

The weak convergence theorem of the sequence $\{f^n\}_{n=1}^{+\infty}$ that comes below is a prerequisite for proving Theorem 6.15 that is the main result: 

\begin{theorem}\label{6.14} 
{\bf (Weak Convergence)} Make the same assumptions as in Lemma 6.8, then we have 

A) There exists $f^*\in\cH$ such that $f^n\ovs{weakly}{\lo} f^*$ and $f^*$ is 
a minimizer of $\Phi_{W, P}$.

B) If either there exists $i;~1\leq i\leq n$ such that $p_i>1$ or $N(K)=0$, 
then $\Phi_{W, P}$ has a unique minimizer $f^*$.
\end{theorem}

\begin{proof}
Part (A) is a direct result of Lemma 6.7 and 6.13.
For part (B), note that the following 
inequality becomes strict for one $p_i>1$:
$$\sum_{\ga\in\Ga_i} w_\ga \left| (\f{f_1+f_2}{2})_\ga\right|^{p_i} \leq 
\f{\sum_{\ga\in\Ga_i} w_\ga|(f_1)_\ga|^{p_i}+ \sum_{\ga\in\Ga_i} 
w_\ga|(f_2)_\ga|^{p_i}}{2}.$$
Consequently
$$\sum_{i=1}^n |\!|\!|\f{f_1+f_2}{2}|\!|\!|^{p_i}_{W_i,p_i} < \f{\sum_{i=1}^n 
|\!|\!|f_1|\!|\!|^{p_i}_{W_i,p_i} +\sum_{i=1}^n 
|\!|\!|f_2|\!|\!|^{p_i}_{W_i,p_i}}{2}.$$
Then, we have
$$|\!|\!| \f{f_1+f_2}{2} |\!|\!|^{P}_{W, P} < 
\f{|\!|\!|f_1|\!|\!|^{P}_{W, P}+|\!|\!|f_2|\!|\!|^{P}_{W, P}}{2}.$$
\end{proof}
Now, we want to prove that $\{f^n\}_{n=1}^{+\infty}$ converges to $f^*$ (the minimizer of (6.3) whose existence clarified in Lemma 6.13) 
in the $\cH$-norm, i.e.
$$\lim_{k\rig\infty} \|v-v_k\|=0.$$\\
  It will be done in the below theorem. For simplicity, the following shorthand notations are used: 
$$f^*=w-\lim f^n,\quad u_n=f^n-f^*,\quad h=f^*+K^*(g-Kf^*).$$

\begin{theorem}\label{6.15}
{\bf (Strong Convergence)} Considering all assumptions of Lemma 6.8, we have 

A) $\|Ku^n\|\lo 0 \quad \text{for} \quad n\lo\infty$.

B) $\|\bS_{W, P}(h+u^n)-\bS_{W, P}(h)-u^n\|\lo 0 \quad \text{for} \quad n\lo\infty$.

C) If for some $a\in\cH$, and some sequence $(v^n)_{n\in\BN}$, 
$w-\lim_{n\rig\infty} v^n=0$ and $\lim_{n\rig\infty}\|\bS_{W, P}(a+v^n) 
-\bS_{W, P}(a)-v^n\|=0$ then $\|v^n\|\lo 0 \quad \text{for} \quad n\lo\infty$.

D) $\|u^n\|\lo 0 \quad \text{for} \quad n\lo\infty$.
\end{theorem}

\begin{proof}
Parts A and B have similar proofs to corresponding lemmas in [17]. Part D is a direct result of parts B and C. For part C, suppose $D_1=\{i|~ p_i>1\}$ and $D_2=\{i|~ p_i=1\}$. Consequently 
$$D_1\cup D_2=\{1,2,\dots,n\}.$$
To prove this part, it is sufficient to consider the two distinct cases $i\in 
D_1$ and $i\in D_2$, as has been done for $p>1$ and $p=1$, respectively, in [17]. 
We will have:
\begin{align*}
&\forall i\in D_1: \quad \sum_{\ga\in\Ga_i} |v_\ga^n|^2\lo 0\quad as\quad 
n\lo\infty,\\
&\forall i\in D_2: \quad \sum_{\ga\in\Ga_i} |v_\ga^n|^2\lo 0\quad as\quad 
n\lo\infty.
\end{align*}
Then 
$$\|v^n\|^2=\sum_{\ga\in\Ga}|v^n_\ga|^2=\sum_{i=1}^n \sum_{\ga\in\Ga_i} 
|v_\ga^n|^2\lo 0\quad as \quad n\lo\infty.$$
\end{proof}

Until now, we have found a sequence that converges to a minimizer for (6.3). In the remainder, by Theorem 6.19, we will also show that this minimizer is acceptable as a regularized solution of the (possibly ill-posed) inverse problem $Kf=g$.\\
To that point, we need to present the proposition 6.18 which comes later. Its proof is similar to what has been done in Theorem 4.1 in [17]. However, it should be clarified that Lemma 4.3 in [17] that has a key role in this theorem will remain true by replacing $|||f|||_{W,p}^{p}$ with $|\!|\!|v|\!|\!|
^{P}_{W, P}$. Below, you can find the very lemma:
\begin{lemma}\label{6.16}
If the sequence of vectors $(v_k)_{k\in\BN}$ converges weakly in $\cH$ to 
$v$, and $\lim_{k\rig\infty} |\!|\!|v_k|\!|\!|_{W, P}^{P}=|\!|\!|v|\!|\!|
^{P}_{W, P}$, 
then $(v_k)_{k\in\BN}$ converges to $v$ in the $\cH$-norm, i.e.
$$\lim_{k\rig\infty} \|v-v_k\|=0.$$
\end{lemma}

\begin{proof}
It is a standard result that if $w-\lim_{k\rig\infty} v_k=v$ and 
$\lim_{k\rig\infty} \|v_k\|=\|v\|$, then 
$$\lim_{k\rig\infty} \|v-v_k\|^2 =\lim_{k\rig\infty} (\|v\|^2+ 
\|v_k\|^2-2<v,v_k>) = (\|v\|^2+\|v\|^2-2<v,v>)=0.$$
We thus need to prove only that $\lim_{k\rig\infty} \|v_k\|=\|v\|$. 

Since, the $v_k$ converge weakly, they are uniformly bounded. It follows that 
the $|v_{k,\ga}|=|<v_k,\var_\ga>|$ are bounded uniformly in $k$ and $\ga$ by 
some finite number $M$. For $a,b>0$ and $1\leq r\leq 2$, we have 
$$|a^r-b^r| \leq r|a-b| \Max\{a,b\}^{r-1}.$$
Consequently 
$$\|v_{k,\ga}|^2-|v_\ga|^2 | = |(|v_{k,\ga}|^{p_\ga})^{\f{2}{p_\ga}} 
-(|v_\ga|^{p_\ga})^{\f{2}{p_\ga}}| \leq \f{2}{p_\ga} M^{2-p_\ga} 
\|v_{k,\ga}|^{p_\ga}-|v_\ga|^{p_\ga}|.$$
Now, let $M':=\Max\{\f{2}{p_i} M^{2-p_i};~ 1\leq i\leq n\}$. Then
\begin{align*} 
\forall \ga\in\Ga,~ \forall k\in N:~ ||v_{k,\ga}|^2-|v_\ga|^2| \leq 
M'||v_{k,\ga}|^{p_\ga}-|v_\ga|^{p_\ga}|.\tag{$6.10$}
\end{align*}
Since the $v_k$ convergence weakly, we have
\begin{align*}
\forall& \ga\in\Ga; ~<v_k,\var_\ga> \lo <v,\var_\ga> \quad as \quad k\lo\infty.\\
\Lo& \forall \ga\in\Ga ;~ |v_{k,\ga}| \lo |v_\ga| \quad as \quad k\lo\infty.
\end{align*}
Define now $u_{k,\ga}=\min(|v_{k,\ga}|,|v_\ga|)$. Clearly $\forall \ga\in\Ga$: 
$\lim_{k\rig\infty} u_{k,\ga}=|v_\ga|$; since $\sum_{\ga\in\Ga_i} 
w_\ga|v_\ga|^{p_i}<\infty$, it follows by the dominated convergence theorem 
that\\ 
$\lim_{k\rig\infty} \sum_{\ga\in\Ga_i}w_\ga u^{p_i}_{k,\ga}=\sum_{\ga\in\Ga_i} w_\ga|v_\ga|^{p_i}$, consequently
\begin{align*}
\lim_{k\rig\infty} \sum_{\ga\in\Ga} w_\ga u_{k,\ga}^{p_\ga} &=
\lim_{k\rig\infty} \sum_{i=1}^n\sum_{\ga\in\Ga_i}w_\ga u_{k,\ga}^{p_i}\\
&=\sum_{i=1}^n \lim_{k\rig\infty}\sum_{\ga\in\Ga_i}w_\ga u_{k,\ga}^{p_i}\\
&=\sum_{i=1}^n\sum_{\ga\in\Ga_i} w_\ga|v_\ga|^{p_i} \\
&=\sum_{\ga\in\Ga} w_\ga |v_\ga|^{p_\ga}. \tag{$6.11$}
\end{align*}
On the other hand, by (6.10) we have
\begin{align*}
| \|v_k\|^2-\|v\|^2|
& = \left| \sum_{\ga\in\Ga} |v_{k,\ga}|^2-\sum_{\ga\in\Ga}|v_\ga|^2\right| \\
&\leq \sum_{\ga\in\Ga} | |v_{k,\ga}|^2-|v_\ga|^2| \\
&\leq \f{M'}{c} \sum_{\ga\in\Ga} w_\ga | |v_{k,\ga}|^{p_\ga} - 
|v_\ga|^{p_\ga}|\\
&=\f{M'}{c} \left( \sum_{\ga\in\Ga} w_\ga |v_{k,\ga}|^{p_\ga} 
+\sum_{\ga\in\Ga} w_\ga | v_\ga|^{p_\ga} -2\sum_{\ga\in\Ga} w_\ga 
u_{k,\ga}^{p_\ga}\right).
\end{align*}
Since we have (6.11), the last expression tends to 0 as $k$ tends to $\infty$.
\end{proof}

The following existence lemma is necessary for providing Proposition 6.18 and 
Theorem 6.19.
\begin{lemma}\label{6.17}
Suppose $S=N(K)+f_0=\{f: K(f)=K(f_0)\}$ and assume that either there exists 
$j$ such that $p_j>1$ or $N(K)=\{0\}$. Then there is a
unique minimal element 
with regard to $|\!|\!|\cdot|\!|\!|^{P}_{W, P}$ in $S$.
\end{lemma}

\begin{proof}
{\bf uniqueness:} If
$N(K)=0$, then $S=\{f_0\}$. For the case where there exists $j$ such that 
$p_j>1$, suppose $f_1,f_2$ are minimal elements with regard to 
$|\!|\!|\cdot|\!|\!|^{P}_{W, P}$ in $S$ and $f_1\neq f_2$;
consequently $|\!|\!|f_1|\!|\!|^{P}_{W, P}=|\!|\!|f_2|\!|\!|_{W, P}^{P}$. Since for $p_j>1$
\begin{align*}
|\!|\!| \f{f_1+f_2}{2} |\!|\!|^{p_j}_{W_j,p_j}& < 
\f{|\!|\!|f_1|\!|\!|^{p_j}_{W_j,p_j}+|\!|\!|f_2|\!|\!|^{p_j}_{W_j,p_j}}{2},
\end{align*}
we have
\begin{align*}
|\!|\!|f_1|\!|\!|^{P}_{W, P} & \leq 
|\!|\!| \f{f_1+f_2}{2}|\!|\!|^{P}_{W, P} = 
|\!|\!| \f{f_1+f_2}{2}|\!|\!|^{p_1}_{W_1,p_1} +\dots+
|\!|\!| \f{f_1+f_2}{2}|\!|\!|^{p_n}_{W_n,p_n} \\
&< \f{|\!|\!|f_1|\!|\!|^{P}_{W, P} +|\!|\!|f_2|\!|\!|^{P}_{W, P}}{2} \\
&=|\!|\!| f_2|\!|\!|^{P}_{W, P}
\end{align*}
which is a contradiction.\\
{\bf Existence:}
Note that $|\!|\!|\cdot|\!|\!|^{P}_{W, P}$ is not a norm, but 
$|\!|\!|\cdot|\!|\!|_{W_i,p_i}$ 
are norms for $1\leq i\leq n$. Suppose $f_i$ is the element of minimum 
$|\!|\!|\cdot|\!|\!|_{W_i,p_i}$-norm in $S$. Now, let $f^\dagger=\sum_{i=1}^n \sum_{\ga\in\Ga_i} <f_i,\var_\ga>\var_\ga$,
then we have
\begin{align*}
|\!|\!|f^\dagger |\!|\!|^{P}_{W, P} &=|\!|\!| 
f^\dagger|\!|\!|^{p_1}_{W_1,p_1} 
+\dots+|\!|\!|f^\dagger|\!|\!|^{p_n}_{W_n,p_n} \\
&=|\!|\!|f_1|\!|\!|^{p_1}_{W_1,p_1} +\dots+|\!|\!| f_n|\!|\!|^{p_n}_{W_n,p_n}
\end{align*}
because $<f^\dagger,\var_\ga>=<f_i,\var_\ga>$ for $\ga\in\Ga_i$. Consequently, 
$f^\dagger$ is a minimum with regard to $|\!|\!|\cdot|\!|\!|^{P}_{W, P}$ in 
$S$.
\end{proof}

\begin{proposition}\label{6.18}
Assume that $K$ is a bounded operator from $\cH$ to $\cH'$ with $\|K\|<1$,
$(\var_\ga)_{\ga\in\Ga}$ is an orthonormal basis for $\cH$, and 
$W=(w_\ga)_{\ga\in\Ga}$ is a sequence such that $\forall \ga\in\Ga: ~ w_\ga>c>0$. 
Let $\Ga=\Ga_1\cup\dots\cup \Ga_n,~ W_i=(w_\ga)_{\ga\in\Ga_i}$, 
$P=\{p_1,\dots,p_n\}$ such that $1\leq p_i\leq 2$ for $1\leq i\leq n$. 
Suppose $g$ is an element of $\cH'$, $\al=(\al_1,\al_2,\dots,\al_n)$ such that 
$\al_i\geq 0$ for $1\leq i\leq n$, and that either there exists $j$ such 
that $p_j>1$ or $N(K)=\{0\}$. Define the functional $\Phi_{\al,W, P;g}$ on 
$\cH$ by 
$$\Phi_{\al,W, P;g}(f)=\|Kf-g\|^2+ \al_1|\!|\!|f|\!|\!|^{p_1}_{W_1,p_1}+ 
\al_2|\!|\!|f|\!|\!|^{p_2}_{W_2,p_2} +\dots+\al_n|\!|\!|f|\!|\!|^{p_n}_{W_n,p_n}.
$$
Also suppose $f^*_{\al,W, P;g}$ is the minimizer of the functional 
$\Phi_{\al,W, P;g},~ f_0\in\cH,~ S=N(K)+f_0=\{f|~K(f)=K(f_0)\}$ 
and $f^\dagger$ is the unique minimal element with regard to 
$|\!|\!|\cdot|\!|\!|^{P}_{W, P}$ in 
$S$. Let $\{\eps_t\}_{t=1}^\infty$ be a sequence of positive 
numbers convergent to 0
and $\al(\eps_t)=(\al_1(\eps_t),\dots,\al_n(\eps_t))$ such that 
$\lim_{t\rig+\infty} \al_i(\eps_t)=0$, $\lim_{t\rig+\infty} 
\f{\eps_t^2}{\al_i(\eps_t)} =0$, $\lim_{t\rig+\infty} 
\f{\al_i(\eps_t)}{\al_j(\eps_t)}=1$ for every $i,j$;~ $1\leq i,j\leq n$.
Suppose $\{g_n\}_{n=1}^{+\infty} \subs\cH'$ is a sequence such that, for every 
$n$, $\|g_n-Kf_0\|<\eps_n$. 
Then we have 
$$\|f^*_{\al(\eps_n);g_n}- f^\dagger\| \lo 0\quad as \quad n\lo\infty.$$
\end{proposition}

Now, we will conclude the discussion around the minimizer of (6.3) by the following regularization theorem: 

\begin{theorem}\label{6.19}
Assume that $K$ is a bounded operator from $\cH$ to $\cH'$ 
with $\|K\|<1$, $\{\var_\ga\}_{\ga\in\Ga}$ is an orthonormal basis for 
$\cH$, and $W=(w_\ga)_{\ga\in\Ga}$ is a sequence such that $\forall \ga\in\Ga:~ w_\ga>c>0$. Let $\Ga=\Ga_1\cup\Ga_2\cup\dots\cup \Ga_n$, 
$W_i=(w_\ga)_{\ga\in\Ga_i}$, $P=\{p_1,p_2,\dots,p_n\}$ such that $1\leq 
p_i\leq 2$ for $1\leq i\leq n$. Suppose that $g$ is an element of $\cH'$, 
$\al=(\al_1,\al_2,\dots,\al_n)$ such that $\al_i\geq 0$ for $1\leq i\leq n$, 
and that either there exists $j$ such that $p_j>1$ or $N(K)=\{0\}$. Define 
the functional $\Phi_{\al,W, P;g}$ on $\cH$ by 
$\Phi_{\al,W, P;g}(f)=\|Kf-g\|^2+\al_1|\!|\!|f|\!|\!|^{p_1}_{W_1,p_1}+\al_2 
|\!|\!|f|\!|\!|^{p_2}_{W_2,p_2}+\dots+\al_n |\!|\!|f|\!|\!|^{p_n}_{W_n,p_n}$.
Also assume
$f^*_{\al,W, P;g}$ is the minimizer of the functional $\Phi_{\al,W, P;g}.$

Let $\al(\eps)=(\al_1(\eps),\al_2(\eps),\dots,\al_n(\eps))$ such that 
$\lim_{\eps\lo 0}\al_i(\eps)= 0$, $\lim_{\eps\rig 0}\f{\eps^2}{\al_i(\eps)} 
=0$ and $\lim_{\eps\rig 0}\f{\al_i(\eps)}{\al_j(\eps)}=1$ for every $i,j;~ 1\leq 
i,j\leq n$. Then we have, for any $f_0\in\cH$,
$$\lim_{\eps\rig 0} \left[ \sup_{\|g-Kf_0\|<\eps} 
\|f^*_{\al(\eps),W, P;g}-f^\dagger\|\right]=0$$
where $f^\dagger$ is the unique minimal element with regard to 
$|\!|\!|\cdot|\!|\!|^{P}_{W, P}$ in $S=N(K)+f_0=\{f;~K(f)=K(f_0)\}$ as 
stated in Lemma 6.17.
\end{theorem}

\begin{proof}
Let $H(\eps):=\sup\{\|f^*_{\al(\eps),W, P;g}-f^\dagger\| |~g\in\cH',~ 
\|g-Kf_0\|<\eps\}.$ We should establish $\lim_{\eps\rig 0} H(\eps)=0$. If $\lim_{\eps\rig 0} 
H(\eps)\neq 0$, then there is a sequence $\{\eps_n\}_{n=1}^{+\infty}$ such that 
$\eps_n\lo 0$ and $\{H(\eps_n)\}_{n=1}^{+\infty}$ is not convergent to 0. 
Consequently,
\begin{align*}
&\exists
\del_0>0;~ \exists \{H(\eps_{n_k})\}_{k=1}^{+\infty};~ \forall k\in\BN\;~
H(\eps_{n_k})\geq \del_0\\
\Lo
&\exists
\del_0>0;~ \exists \{H(\eps_{n_k})\}_{n=1}^{+\infty};~ \forall k\in\BN\;~ 
\exists g_{n_k}\in\cH';~ \|g_{n_k}-Kf_0\|\leq 
\eps_{n_{k}},\\
&~~~~~~~~~~~~~~~~~~~~~~~~~~~~~~~~~~~~~~~~~~~~~~~~~~~~~~~~~~~ \|f^*_{\al(\eps_{n_k});g_{n_k}}-f^\dagger\|\geq \del_0.
\end{align*}
This is a contradiction, because
by Proposition 6.18 for $\{\eps_{n_k}\}_{k=1}^{+\infty}$, 
$\{g_{n_k}\}_{k=1}^{+\infty}$, we will have
$$\|f^*_{\al(\eps_{n_k});g_{n_k}} -f^\dagger\|\lo 0.$$
\end{proof}
This section is concluded with addressing another advantage of the above iterative method for solving the problem (6.3) or the structure of ``linear inverse problems with mixed multi-constraints" introduced in [30] - hereinafter we briefly call it $(ii)$ for simplicity. In addition to the computational clarification, this solution can also clarify the significance of the generalization provided in [30] over other generalizations for the structure of ``linear inverse problems with a sparsity constraint" in [17] - hereinafter we briefly call it $(i)$. 
First, it should be noted that the structure of ``linear inverse problems with multi-constraints" which arises in [30] - hereinafter we briefly call it $(iii)$ - is more general than the structure of the problem (6.4). Indeed, while the problem (6.3) is a special case of $(iii)$, it cannot be derived from (6.4) (because $B^{\prime}_{W,p}\nsubseteq B_{W,P} $). In fact, $(i)$, $(ii)$ and $(iii)$ have the following relationships: 
 \begin{equation*}
(iii)\Rightarrow(ii)\Rightarrow(i). 
\tag*{$(6\cdot 12)$}
\end{equation*}

After appearing $(i)$ which comes in the seminal paper [17] of Daubechies et al., so many computational and theoretical works around it, have been done by researchers. In [30], I have shown that some classes of generalization for $(i)$ (many problems addressed in [18, 19, 24, 33, 39, 41, 45, 48], for example) can be specific cases for $(ii)$. However, in [30], $(ii)$ is resulted from $(iii)$ without presenting a process for reaching its solution independently. By presentation of an independent solution with weaker (or the weakest) assumptions, that was done above, we can better judge how $(iii)$ theoretically is a more extensive generalization for $(i)$ in comparison with other works. In other words, although we have (6.12) and we know many generalizations for $(i)$ can be resulted from $(ii)$, since $(ii)$ was presented merely as a result of $(iii)$ in [30], it is not clear whether $(ii)$ - and consequently the other mentioned generalizations - is theoretically closer to $(i)$ or $(iii)$. However, the independent process for solving $(ii)$ revealed that $(ii)$ is theoretically so close to $(i)$ in a way that the iterative process for reaching a minimizer for $(ii)$ can simply be considered as a rewrite for the iterative process for $(i)$ presented in [17]. Consequently, both $(i)$ and $(ii)$ have a big distance theoretically from $(iii)$. As a result, $(iii)$ is a more extensive generalization for $(i)$ compared to other generalizations for it. Then, the learning method regarding the structure $(iii)$ has superiority to other possible methods in the area of ``Learning and Inverse Problems". Eventually, since $\textbf{C}$ is much larger than $B_{W,P}$, the significance of the bi-level optimization (1.7) by considering $\textbf{Z}=\textbf{C}$ is better recognized as well. 

\section{Conclusion}

In this paper, it was tried to achieve constraint $\psi$ in the following famous 
minimization problem regarding a supposed data set:
$$\text{inf}_{f 
 \in \mathcal{H}} 
\Vert K(f)- g \Vert ^2 +\psi(f).$$
However, the kind of $\psi$ has already always been chosen from some alternatives regarding the kind of data or applications, here, we suggest constructing the whole of  $\psi$ exactly by data in hand, namely the training set. The training sets could be both artificially made or coming from real-world applications such as “Satellite Image Processing” and “Speech Processing” explained in continue. As discussed, where $K=I$,
this training process naturally arrived at a bi-level 
optimization problem which for its solution, as was shown, solving several constrained optimization problems on $\mathbb{R}^n$ would be inevitable. In the introduction section, this model was discussed as a view of obtaining a denoising method. However, this is a particular case of a more general idea that in this paper 
only the case $K=I$ or when $K$ is diagonal was investigated.\\

As already mentioned in Introduction, another research area that focuses on addressing a learning method for denoising by a training set $\lbrace (f_i , g_i)\vert ~ i=1,\cdots,m \rbrace$ at hand is parameter learning approaches in variational models that started with [20, 34] and is still ongoing [23, 32]. In fact, as signal reconstruction by wavelet is an alternative for signal reconstruction by the variational method [10], our learning method - the idea of which is derived from nonlinear wavelet image processing in [12, 22] - is an alternative (and, of course, more powerful because $\textbf{C}$ is much larger than $B_{W,P}$) for parameter learning in variational models as well. Moreover, another research area concerning learning method for denoising by a training set $\lbrace (f_i , g_i)\vert ~ i=1,\cdots,m \rbrace$ at hand is related to the works that have been done in the realm of {\it deep neural networks} in recent years; e.g. [37, 49]. 

Although the comparison of the two conducted research areas is of great importance, it has been of little attention in academic literature (such as [36], for example). Besides, the establishment of connections between methods is of great value as well, since it brings an interpretation for developing Learning Theory. For instance, the following idea around the connection of our learning method and deep neural networks could lead to a new learning method for denoising:\\
Please note that in Sections 2 and 4, $\lbrace ((f_i )_\gamma, (g_i)_\gamma)\vert ~i=1,\cdots,m \rbrace$ for $ \gamma \in \Gamma $ were attained from the training set $\lbrace (f_i , g_i)\vert ~ i=1,\cdots,m \rbrace$  by considering a basis $ \lbrace \varphi_{\gamma} \rbrace_{\gamma \in \Gamma} $. Suppose the functions $ \Phi_{\gamma} $ are: 
$$ \Phi_{\gamma}(x) =\text{arg-min}_{y\in (M_1,M_2]}(y-x)^2+\psi_{\gamma}(y);~~\gamma \in \Gamma$$
where $\psi_{\gamma}$ is the regularization term corresponding to $\lbrace ((f_i )_\gamma, (g_i)_\gamma)\vert ~i=1,\cdots,m \rbrace$. Here the question is what DNNs $ \Phi_{\gamma}^* $ come true in the below inequalities: 
\begin{equation*}
\sum_{i=1}^m \vert \Phi_{\gamma}^*((g_i)_\gamma)-(f_i)_\gamma \vert^2 \leq \sum_{i=1}^m \vert \Phi_{\gamma}((g_i)_\gamma)-(f_i)_\gamma \vert^2;~~\gamma \in \Gamma 
\tag*{$(7\cdot 1)$}
\end{equation*}
The answer to this question leads to finding a DNN model whose performance is interpretable theoretically - not merely by experience. The reason is, due to (7.1), the performance is supposed to be, to say the least, as good as that of the learning method presented in this paper. On the other hand, considering having a large enough data set, our learning method is expected to have a better performance from the classic method from which it has been inspired. As a result, the DNN model functions better than the classic method. Besides, the latter is theoretically verified, so the good function of the former in the same applications is justified.\\

Having reached the denoising mathematical model, it can have different applications. For example, one application can be in cellular phones. The finite numbers of sample audio signals are set in cell phones and it is a common agreement among all of the mobile phones. Before starting the conversation, these sample audio signals are transferred to the cell phone on the receiving side. Once these signals reach the opposite cell phone, a denoising method will be built for that conversation depending on the type of noises which were added to these signals. In other words, the denoising method varies for each conversation compared to other ones. Since different noises arise in different conditions, it seems this method is better than all of the common methods according to which the processes of denoising are the same.

In some other applications, we will need an extra effort - beyond the learning process presented in this paper - for removing images' degradation. For instance, this comes true in processing of an image that a satellite records and sends to the earth. The images that are received from a satellite have two kinds of degradation: ``noise" and ``blurring". Unlike the ``noise" affecting a particular image which is a statistical process and consequently is not known, ``blurring" is a deterministic process and, in most cases, one has a sufficiently accurate mathematical model for its description. The simplest example of blurring is due to relative motion, during exposure, between the camera and the object being photographed in such a way that no part of the object goes out of the image domain during the motion. Examples arise when, for instance, photographing the earth, the moon, and planets using aerial vehicles such as aircraft and spacecraft. Examples can also be found in forensic science when potentially useful photographs are sometimes of extremely bad quality due to motion blur.  The problem of image restoration from these linear motion blurs can be formulated by solving the following linear equation:
                                                                                     $$K (f) = g$$
where $f$ is the original image, $g$ is the blurred image of $f$ and $K$ is a bounded linear operator (see [5] for more information). Now, suppose it is possible to put a series of images in the satellite before launching it into space. To that end, the satellite must send these series of images just before sending the images recorded in space. While being transferred to earth, these images become noisy. Since the original images are available on earth, the noises that are added to them during the transfer are quite detectable (known). This prior cognition can help the denoising process of recorded images by satellite. In fact, by the learning method provided in this paper, a $\psi$ which is an element of the set $\textbf{C}$ can be found regarding this training set. But here, to remove the blurring degradation as well, the minimizer of the following functional should be found instead of the functional (1.6): 
\begin{equation*}
\|Kf-g\|^2+\psi(f)
\tag*{$(7\cdot 2)$}
\end{equation*}
where $K$ is a bounded linear operator which interprets the blurring degradation coming from pictures' formation in the space, because of the relative motion between the satellite's camera and the object being photographed.  To solve (7.2), it seems following the iterative method provided in [17] or Section 6 can be helpful.\\

Solving the bilevel optimization problem presented in this paper can also have different types of applications from the ones already introduced. The applications discussed above refer to the ones that relate to modeling before the denoising process and to explain that various models can be obtained according to different conditions. However, as we addressed in Remark 2.9, in addition to the view in which the model has an ability to adapt, there is another perspective by which new constraints can be created for Inverse Problems. Moreover, a further view that was mentioned in Section 5 is about utilizing our approach to answer this sensible question that what the optimal regularizations term is regarding the various types of noise such as Gaussian noise, Poisson noise, Speckle noise, Salt and Pepper noise.\\  
\\


\end{document}